\documentclass[11pt]{article}
\usepackage{a4}
\usepackage{amsmath,amssymb,amsthm}
\usepackage{graphicx}
\usepackage{color}
\usepackage{mathrsfs,bbm,bm}
\usepackage{graphics,subfigure,booktabs,mathtools,stmaryrd}

\allowdisplaybreaks

\newtheorem{theorem}{Theorem}[section]
\newtheorem{lemma}[theorem]{Lemma}

\newtheorem{example}[theorem]{Example}
\newtheorem{remark}[theorem]{Remark}

\numberwithin{equation}{section}

\newcommand{\sq}{\varepsilon}
\newcommand{\ssq}{\sqrt{\varepsilon}}
\newcommand{\wph}{\bm{\mathbbm{W}}}
\newcommand{\uph}{\mathbbm{U}}
\newcommand{\pph}{\mathbbm{P}}
\newcommand{\qph}{\mathbbm{Q}}
\newcommand{\vph}{\bm{\mathbbm{z}}}
\newcommand{\vphu}{\mathbbm{v}}
\newcommand{\vphq}{\mathbbm{r}}
\newcommand{\vphp}{\mathbbm{s}}
\newcommand{\dual}[2]{\left\langle#1,#2\right\rangle}
\newcommand{\norm}   [1] {\left\Vert#1\right\Vert}
\newcommand{\enorm}   [1] {\interleave #1\interleave_{E}}
\newcommand{\average}[1] {\{\!\{#1\}\!\}}
\newcommand{\jump}   [1] {[\![#1]\!]}
\newcommand{\spc}{\mathcal{V}_N}
\newcommand{\ba}   [1] {\interleave #1\interleave_{B}}

\newcommand{\ureg}{\bar{u}}
\newcommand{\uep}{u_\sq}

\newcommand{\partdfij}{\partial_{x}^{i}\partial_{y}^{j}}

\newcommand{\prou}{\mathcal{I}}
\newcommand{\prop}{\mathcal{J}}
\newcommand{\proq}{\mathcal{K}}
\newcommand{\prog}{\mathcal{R}}
\newcommand{\prow}{\mathcal{P}}

\newcommand{\dreg}{\Omega_{22}}
\newcommand{\dws}{\Omega_{11}}
\newcommand{\dw}{\Omega_{12}}
\newcommand{\dwn}{\Omega_{13}}
\newcommand{\ds}{\Omega_{21}}
\newcommand{\dn}{\Omega_{23}}
\newcommand{\des}{\Omega_{31}}
\newcommand{\de}{\Omega_{32}}
\newcommand{\den}{\Omega_{33}}
\newcommand{\dxa}{\Omega^x_{1}}
\newcommand{\dxb}{\Omega^x_{2}}
\newcommand{\dxc}{\Omega^x_{3}}
\newcommand{\dya}{\Omega^y_{1}}
\newcommand{\dyb}{\Omega^y_{2}}
\newcommand{\dyc}{\Omega^y_{3}}

\title{Optimal balanced-norm error estimate of the LDG method for reaction-diffusion problems II: 
	the two-dimensional case with layer-upwind flux
	\thanks{The research of the first author was supported by NSFC grant 11801396,
		Natural Science Foundation of Jiangsu Province grant BK20170374
		and Qing Lan Project of Jiangsu University.
		The research of the second author was supported by NSFC grant 11801396.
		The research of the third author was supported by NSFC grants 12171025 and NSAF-U2230402.
}}

\author{Yao Cheng\thanks{School of Mathematical Sciences,
		Suzhou University of Science and Technology,
		Suzhou 215009, Jiangsu Province, China (\texttt{ycheng@usts.edu.cn}).}
	\and
	Xuesong Wang\thanks{School of Mathematical Sciences,
		Suzhou University of Science and Technology,
		Suzhou 215009, Jiangsu Province, China (\texttt{wangxuesong@post.usts.edu.cn}).}
	\and
	Martin Stynes\thanks{Applied and Computational Mathematics Division, Beijing Computational Science Research Center, Beijing 100193, China (\texttt{m.stynes@csrc.ac.cn}).}
}

\begin{document}

\maketitle

\begin{abstract}
A singularly perturbed reaction-diffusion problem posed on the unit square in $\mathbb{R}^2$
is solved numerically by a local discontinuous Galerkin (LDG) finite element method.
Typical solutions of this class of 2D problems exhibit boundary layers along the sides of the domain;
these layers generally cause difficulties for numerical methods.
Our LDG method handles the boundary layers by using a Shishkin mesh and also introducing the new concept
of a ``layer-upwind flux"---a discrete flux whose values are chosen on the fine mesh 
(which lies inside the boundary layers) in the direction where the layer weakens. 
On the coarse mesh, one can use a standard central flux. 
No penalty terms are needed with these fluxes, 
unlike many other variants of the LDG method.
Our choice of discrete flux makes it feasible to derive
an optimal-order error analysis in a balanced norm; 
this norm is stronger than the usual energy norm 
and is a more appropriate measure for errors in
computed solutions for singularly perturbed reaction-diffusion problems.
It will be proved that the LDG method is usually convergent of order $O((N^{-1}\ln N)^{k+1})$ 
in the balanced norm, where $N$ is the number of mesh intervals in each coordinate direction
and tensor-product piecewise polynomials of degree~$k$ in each coordinate variable are used in the LDG method.
This result is the first of its kind for the LDG method applied to this class of problem
and is optimal for convergence on a Shishkin mesh.
Its sharpness is confirmed by numerical experiments. 
\end{abstract}

\section{Introduction}\label{sec:intro}
Consider the singularly perturbed reaction-diffusion problem
\begin{subequations}\label{spp:R-D:2d}
	\begin{align}
	-\sq\Delta u(x,y)+ b(x,y)u(x,y)&=f(x,y)\  \text{ in } \Omega:=(0,1)\times(0,1),
	\label{R-D:model:1}
	\\
	u&=0 \hspace{1cm} \text{ on } \partial \Omega,
	\label{R-D:model:2}
	\end{align}
\end{subequations} 
where $\varepsilon>0$ is a small parameter, $b\geq 2\beta^2>0$
and $\beta$ is a positive constant.
We assume that $b$ and $f$ are sufficiently smooth; more details are given in Section~\ref{sec:scheme}.

Singularly perturbed problems like \eqref{spp:R-D:2d} appear in various applications, e.g., 
when solving nonlinear reaction-diffusion problems by Newton's method 
and in implicit time-discretizations  of parabolic reaction-diffusion problems 
with small time steps \cite{HeuerKarkulik2017}. 
The true solution of~\eqref{spp:R-D:2d} will usually exhibit sharp boundary layers 
at the boundary~$\partial\Omega$ of the domain~$\Omega$, 
which in general reduces the accuracy of numerical solutions of~\eqref{spp:R-D:2d}. 
Thus, many special numerical methods for~\eqref{spp:R-D:2d}
have been constructed and analysed; 
see \cite{Apel1999,Clavero2005,LiNavon1998,Lin2008SIAM,Linss10,Liu2009,Roos2008,Roos2003,Schatz1983,StySty18} 
and their references.

\subsection{Balanced norms}\label{sec:balnorm}
Unlike singularly perturbed convection-diffusion problems, the natural energy norm 
associated with \eqref{spp:R-D:2d} is too weak to capture the contribution of 
the layer component of the true solution, as was pointed out in~\cite{LinStynes2012},
which advocated the replacement of the energy norm by a stronger so-called \emph{balanced norm}.
In this new norm, the $H^1(\Omega)$ component of the energy norm is rescaled in such a way 
that the layer and smooth components of the true solution~$u$
of~\eqref{spp:R-D:2d} have the same order of magnitude with respect to $\varepsilon$.

Balanced-norm error analyses have attracted much attention since the appearance of~\cite{LinStynes2012}. 
These analyses can be divided into two classes:  
\begin{enumerate}
\item[(i)] Use a standard numerical method but carry out the error analysis in 
a balanced norm, e.g., $C^0$-interior penalty method~\cite{RoosSchopf2015}, 
$hp$-finite element method (FEM) on spectral boundary layer meshes~\cite{Melenk2016}, 
streamline-diffusion FEM~\cite{Franz2014} and a finite volume method~\cite{MengStynes2023}. 
Note however that it is usually troublesome 
to coax a balanced-norm error bound from a method that fits more naturally with a
standard energy norm.

\item[(ii)] Construct a new numerical method that leads naturally to a balanced norm, 
e.g., mixed FEM~\cite{LinStynes2012}, dual FEM~\cite{Cai2020}, 
discontinuous Petrov-Galerkin method~\cite{HeuerKarkulik2017}, 
and an exponentially weighted FEM~\cite{Armentano2023,MaddenStynes2021}.

\end{enumerate}

\subsection{The LDG method}
The local discontinuous Galerkin (LDG) method \cite{Cockburn:Shu:LDG}
is a popular FEM that uses discontinuous piecewise polynomial spaces for the trial and test solutions. 
Its basic idea  is to rewrite the second-order differential equation as an equivalent first-order system 
to which one then applies the original DG method~\cite{Reed1973}. 
In the LDG  method one must specify a numerical flux on the boundary of each mesh element;
if this is done appropriately, then the auxiliary variables that approximate the derivatives of the solution
can be eliminated locally, which explains the name LDG. 
The LDG method is superior to many other numerical methods
because of its strong stability, high-order accuracy, flexibility for $hp$-adaptivity, and high parallelizability. 
It is well suited to problems whose solutions have boundary layers.
Furthermore, it can produce high-order accuracy for both the solution itself and its gradient; 
the latter quantity is important in fluid dynamics applications
where values of pressure gradients are desired.

The LDG method has been used to solve various 
singularly perturbed problems with boundary layers 
\cite{ChengJiangStynes2023,Cheng2021:Calcolo,ChengStynes2023,Lin2008SIAM,WCX2017,WangCheng2024,Xie2009JCM,Zhu2013,Zhu:2dMC} and 
several energy-norm convergence results have been derived; 
in these papers the numerical flux is usually either a penalty flux
at all element interfaces or an alternating flux at the interior element
interfaces combined with a penalty flux at the domain boundary~$\partial\Omega$.
Nevertheless, only a few   
balanced-norm error bounds are known for LDG solutions of 
singularly perturbed reaction-diffusion problems. 
We now describe these results.

In~\cite{Cheng2021JCAM} a local $L^2$~projection and an inverse inequality were used 
to derive a balanced-norm error estimate for a two-parameter singularly perturbed 
two-point boundary value problem
which is of reaction-diffusion type when the convection coefficient parameter is zero;
a suboptimal convergence rate of order $O((N^{-1}\ln N)^k)$ was obtained 
for the LDG method with a global penalty flux on element interfaces,
where $N$ is the number of mesh elements and 
$k$ is the degree of the piecewise polynomials in the FEM space. 
Subsequently, in \cite{CYM2022} an LDG method with interior alternating fluxes 
and penalty fluxes at the boundary~$\partial\Omega$ was considered.
Different projectors for the smooth and layer components of the solution 
were employed, improving the convergence rate on the elements adjacent to the transition points;
consequently the final convergence order was improved by one-half. 
Under the restrictive assumption that the regular component of the solution 
belongs to the FEM space, 
the convergence rate was proved to be the optimal rate $O((N^{-1}\ln N)^{k+1})$. 
But despite this progress, a gap still remained between the theoretical and numerical results.

\subsection{Choice of numerical flux in the LDG method}\label{sec:fluxchoice}
In~\cite{ChengWangStynes2024JSC} we recently applied the LDG method 
on a Shishkin mesh to solve a one-dimensional reaction-diffusion problem, 
using a central flux at interior element interfaces and a penalty flux at the boundary. 
By using the nodal superconvergence of the average of the local $L^2$ projections at the element interfaces
when adjacent elements are the same size,
an optimal-order error bound was derived in a balanced norm: 
we proved $O((N^{-1}\ln N)^{k+1})$ for even~$k$ and $O((N^{-1}\ln N)^{k})$ for odd~$k$. 
These convergence rates are sharp and agree with the numerical results 
for this method.

Despite this positive convergence result for a 1D problem, the LDG method with central numerical flux
at every interior element interface has some drawbacks.
First, it has a lower rate of convergence when $k$ is odd.
Second, it seems to be very difficult to extend the 1D error analysis
of~\cite{ChengWangStynes2024JSC}  to the 2D case,
because the obvious 2D choice of  projector is 
the tensor product of the 1D projectors from~\cite{ChengWangStynes2024JSC},
but in 2D these interact with each other in the vicinity of the mesh transition interfaces
in a complicated way that  seriously obstructs the error analysis.  
Third, the special nodal superconvergence property of the $L^2$ projection on locally uniform meshes,
which was used in~\cite{ChengWangStynes2024JSC} to analyse the method on a Shishkin mesh,
cannot be extended to other layer-adapted meshes such as Bakhvalov-type meshes \cite{Linss10,Roos2008} that
are not locally uniform inside the boundary layers.

To address these issues,  in the present paper we introduce and analyse a new design principle 
for the numerical flux in the LDG method when solving singularly perturbed problems with layers. 
We call it a \emph{layer-upwind} flux since the choice of numerical flux on each element 
of the fine Shishkin mesh depends on where the layer is weakest in that element. 
On the coarse part of the Shishkin mesh we use the old idea of a central flux.  
With these choices of numerical fluxes, the LDG method needs 
no penalty terms at the domain boundary, 
unlike~\cite{Castillo2002,ChengJiangStynes2023,ChengStynes2023,CYM2022,WangCheng2024,Zhu:2dMC}.

The new layer-upwind flux is natural and is symmetric across~$\Omega$ 
when the layers in the true solution of~\eqref{spp:R-D:2d} are symmetric,
unlike the fluxes used in~\cite{ChengWangStynes2024JSC}.
Furthermore, with this choice of flux we can derive 
a balanced-norm error estimate 
that achieves a nearly optimal convergence rate
for both even and odd values of $k$.
To be precise, we show $O((N^{-1}\ln N)^{k+1})$ convergence in the balanced norm 
for even values of~$k$ under the reasonable assumption that $\varepsilon^{1/2}\le N^{-1}$; see Remark~\ref{rmk:thm}.
For odd values of~$k$, we get the same result when $\varepsilon^{1/4}\le N^{-1}(\ln N)^{k+1}$,
and if this condition is violated (but one still has $\varepsilon^{1/2}\le N^{-1}$), 
then the balanced norm convergence rate becomes $O((N^{-1}\ln N)^{k+1}+N^{-(k+1/2)})$,
but in practice this is unlikely to happen and so
one usually sees  $O((N^{-1}\ln N)^{k+1})$ convergence; again see Remark~\ref{rmk:thm}.
To our knowledge, this is the best balanced-norm error estimate result that has been proved
for any LDG method applied to a 2D reaction-diffusion singularly perturbed problem.

\subsection{Complexity of the error analysis}
The bilinear form $B(\cdot;\cdot)$ of our LDG method satisfies 
$B(\vph;\vph) = \enorm{\vph}^2$ for all functions~$\vph$ 
in the finite element trial space~$\spc^3$, 
where $\enorm{\cdot}$ is an associated energy norm; see Section~\ref{sec:LDG:method}.
Thus, the method falls into category~(i) of Section~\ref{sec:balnorm}---its error analysis 
fits naturally in an energy norm framework but is challenging when 
the stronger balanced norm $\ba{\cdot}$ is used 
(see \eqref{balanced:norm:2d}),
as the only direct relationship between  $B(\cdot;\cdot)$
and  $\ba{\cdot}$
introduces a factor~$\sq^{-1/4}$ into the error estimates, which is unacceptable
since we want error bounds that prove the accuracy of the LDG solution 
for all values of $\sq\in (0,1]$.

To deal with this fundamental difficulty, 
our strategy is to prove a \emph{supercloseness} result
for the LDG solution $\wph\in \spc^3$: 
setting $\bm w := (u, \sq u_x, \sq u_y)$, 
we construct a projection $\prow \bm w\in \spc^3$ of~$\bm w$ 
such that (see \eqref{xisuperclose}) for some constant~$C$ one has
\begin{equation}\label{super}
\enorm{\prow \bm w-\wph} \le C\sq^{1/4}\zeta(N)
\ \text{ with }\zeta(N)\to 0 \text{ as }N\to\infty,
\end{equation}
where $N$ is the number of mesh intervals in each coordinate direction.
This implies immediately that $\ba{\prow \bm w-\wph} \le C\zeta(N)$
and the analysis is then reduced to establishing approximation theory estimates 
for $\ba{\bm w - \prow \bm w}$.

One should note that the estimate \eqref{super} is extraordinarily delicate
because of the factor $\sq^{1/4}$ (recall that $\sq\in (0,1]$ is arbitrary)
and the sharpness of the bound $\zeta(N)$, which agrees exactly with our numerical results. 
Thus, the construction of the complicated projector~$\prow$ requires great care.
Furthermore, to bound $\ba{\bm w - \prow \bm w}$ is far from routine---one must prove 
a variety of precise estimates for~$\prow$.

The projector $\prow$ is a combination of 
several composite Gauss-Radau and $L^2$ projectors 
that are tailored to the form of the numerical fluxes and have good approximation properties. 
They are local projectors and unlike~\cite{ChengWangStynes2024JSC} they do not 
need to be coupled near the transition interfaces in the mesh. 
The analyses of our earlier LDG papers 
\cite{ChengJiangStynes2023,Cheng2021:Calcolo,ChengStynes2023,CYM2022}
also used Gauss-Radau and $L^2$ projectors, but a comparison with our work 
here---note the elaborate constructions in equation~\eqref{flux:diffusion:2d} 
and in Sections~\ref{sec:projectorconstruction} and~\ref{sec:superapproximation}---shows that 
the present paper devises a remarkable new combination of them 
and presents some innovative approximation and superapproximation properties
that are able to yield a 2D optimal-order balanced-norm error analysis
in practical situations.
Where possible we reuse old results, but almost all our analysis here is new.

It is notable that our new error analysis on the fine Shishkin mesh
makes no use of the nodal superconvergence property of the $L^2$ projection on locally uniform meshes 
that was mentioned in Section~\ref{sec:fluxchoice}, thereby opening
the possibility of extending this work to 
other layer-adapted meshes such as Shishkin-type and Bakhvalov-type meshes.
In general,  our new layer-upwinded numerical flux appears to be a 
very promising tool in the application of  the LDG method to other problems whose solutions have boundary layers.

\subsection{Structure of the paper}
The paper is organised as follows. 
In Section~\ref{sec:scheme}, we define the Shishkin mesh 
and present the new LDG method with its layer-upwind flux. 
Section~\ref{sec:projectors} introduces several local projectors 
and derives their basic approximation properties.
Then in Section~\ref{sec:superapproximation}, various superapproximation properties
are established for these operators.
These results are then used in the balanced-norm error analysis of Section~\ref{section:error:analysis}. 
In Section~\ref{section:numerical:experiments} we present some numerical experiments to confirm 
the sharpness of our theoretical results. 
Finally, Section~\ref{sec:concluding:remarks} gives some concluding remarks.\\[1mm]

\emph{Notation.} We use $C$ to denote a generic positive constant that may depend on the data $b,f$ of~\eqref{spp:R-D:2d}, 
the parameter $\sigma$ of~\eqref{tau:2}, and the degree $k$ of the polynomials in our finite element space,  
but is independent of $\varepsilon$ and of $N$ (the number of mesh intervals in each coordinate direction); 
$C$ can take different values in different places.\\
\indent The usual Sobolev spaces $W^{m,p}(D)$, $H^{m}(D)$ and $L^{p}(D)$ will be used,
where $D$ is any measurable subset of~$\Omega$.
The $L^{2}(D)$ norm is denoted by $\norm{\cdot}_{D}$
and the $L^{\infty}(D)$ norm by $\norm{\cdot}_{L^\infty(D)}$,
and $\dual{\cdot}{\cdot}_D$ denotes the $L^2(D)$ inner product.
The subscript $D$ will be omitted when $D= \Omega$.

\section{The Shishkin mesh and the LDG method}\label{sec:scheme}
In this section we shall formulate the entire numerical method: the Shishkin mesh and 
the LDG discretisation of~\eqref{spp:R-D:2d}. But first we describe the essential features of
the solution~$u$ of~\eqref{spp:R-D:2d}.

\subsection{Problem and solution properties}
Typical solutions $u$ of \eqref{spp:R-D:2d} have an exponential boundary layer 
along each side of $\partial\Omega$.
When the data of~\eqref{spp:R-D:2d} is sufficiently smooth 
and compatible at the corners of~$\Omega$,
the next lemma presents a decomposition of~$u$ 
that includes pointwise bounds on the derivatives of each of its components.
The derivation of such results is lengthy and
detailed proofs are given in \cite{Andreev2006,Clavero2005,HanKellogg1990}; see also
\cite[p.257, Remark~III.1.27]{Roos2008} and \cite[Remark~4.18]{StySty18}.

For convenience we write $\partial_x^i\partial_y^j u(x,y)$ instead of the more usual
$\frac{\partial^{i+j}u(x,y)}{\partial x^i\partial y^j}$.

\begin{lemma}\label{lemma:prop:2d} 
	Let $m$ be a positive integer.
	Under suitable smoothness and compatibility conditions on the data, 
	the solution $u$ of \eqref{spp:R-D:2d}
	lies in the H\"older space $C^{m+2,\alpha}(\bar{\Omega})$ 
	for some $\alpha\in (0,1)$ and can be decomposed as 
	\[
	u=\bar{u}+\uep=\ureg+\sum_{i=1}^{4}u_{i}^{\rm{b}}+\sum_{i=1}^{4}u_{i}^{\rm{c}},
	\]
	where $\ureg$ is the regular/smooth component, 
	each $u_{i}^{\rm{b}}$ is a layer associated with the edge $\Gamma_i$ 
	and each $u_{i}^{\rm{c}}$ is a layer associated with the corner $c_i$ 
	(see Figure~\ref{domaindiv} for the edge and corner notation).
	The derivatives of these components satisfy the following bounds:
	\begin{subequations}\label{ubound}
		\begin{align}\label{reg:smooth}
		|\partdfij \ureg (x,y)| 
		&\leq C &&\text{for } 0\leq i+j\leq m+2,
		\\
		\label{reg:boundary}
		|\partdfij u_{1}^{\rm{b}}(x,y)|   
		&\leq C\sq^{-i/2}e^{-\beta x/\sqrt{\sq}} &&\text{for } 0\leq i+j\leq m+1,
		\\
		\label{reg:corner}
		|\partdfij u_{1}^{\rm{c}}(x,y)|   
		&\leq C\sq^{-(i+j)/2}e^{-\beta (x+y)/\sqrt{\sq}} &&\text{for } 0\leq i+j\leq m+1
		\end{align}
	\end{subequations}
	for all $(x,y)\in \bar{\Omega}$. 
	For $u_{2}^{\rm{b}}$, $u_{3}^{\rm{b}}$ and $u_{4}^{\rm{b}}$, bounds 
	analogous to \eqref{reg:boundary} can be derived; 
	for $u_{2}^{\rm{c}}$, $u_{3}^{\rm{c}}$ and $u_{4}^{\rm{c}}$, bounds 
	analogous to \eqref{reg:corner} can be derived.
\end{lemma}

If for instance one assumes that $f,b\in C^{4,\alpha}(\bar{\Omega})$
and the corner compatibility conditions $f(0,0)=f(1,0)=f(0,1)=f(1,1)=0$ are satisfied,
then Lemma~\ref{lemma:prop:2d} holds true for $m=2$; see 	\cite[Section~1.2]{Liu2009}.

\subsection{Shishkin mesh}
\label{sec:meshes}

We shall construct our LDG method on a standard piecewise-uniform Shishkin mesh 
that is refined near $\partial\Omega$. It is a tensor product of one-dimensional Shishkin meshes
as described in, e.g., \cite{Linss10,Roos2008,StySty18}.

Let $N\geq 4$ be a positive integer that is divisible by $4$.
Our mesh uses $N+1$ points in each coordinate direction.
Define the \emph{mesh transition parameter} $\tau$ by
\begin{equation}\label{tau}
\tau=\min\left\{\frac{\sigma\ssq}{\beta}\ln N,\frac14\right\},
\end{equation}
where $\sigma>0$ is a user-chosen parameter whose value will be discussed later.

The finite element space will be defined in Section~\ref{sec:LDG:method}; 
it uses piecewise polynomials of degree $k\ge 0$.
Define
\begin{equation}\label{khat}
\hat{k}=\begin{cases}
k+1  & \text{for $k$ even},\\
k  & \text{for $k$ odd}.
\end{cases}
\end{equation}

In~\eqref{tau} we choose $\sigma\geq \hat{k}+1\geq k+1$ to facilitate our error analysis.
We make the mild assumption that
\begin{align}\label{mild:condition:epsilon}
\sq\leq \left(\frac{\beta}{4\sigma\ln N}\right)^2,
\ \text{ i.e., }\
\frac{\sigma\ssq}{\beta}\ln N \leq \frac14,
\end{align}
as otherwise the mesh is sufficiently fine to resolve the boundary layers in the solution 
and the error in the numerical solution can be analysed in the classical way 
without any special treatment of the singularly 
perturbed nature of the problem. Consequently \eqref{tau} becomes
\begin{align}\label{tau:2}
\tau= \frac{\sigma\ssq}{\beta}\ln N \ \text{ with } \sigma\geq \hat{k}+1.
\end{align}

Define the mesh points $(x_i,y_j)$ for $i,j=0,1,\dots,N$ by
\begin{equation}\label{mesh:point}
x_i=y_i=
\begin{dcases}
4\tau iN^{-1}
& \textrm{for}\;  i=0,1,\dots,N/4,
\\
\tau+2(1-2\tau)(iN^{-1}-1/4)
&\textrm{for}\; i=N/4+1,\dots,3N/4,
\\
1-4\tau (1-iN^{-1})
& \textrm{for}\; i=3N/4+1,\dots,N.
\end{dcases}
\end{equation}
By drawing axiparallel lines through these mesh points 
we construct the layer-adapted mesh: set 
$\Omega_N:=\{K_{ij}\}_{i,j=1,2,\dots,N}$, where 
$K_{ij}:=I_i \times J_j:=(x_{i-1},x_{i})\times(y_{j-1},y_{j})$.

The right-hand picture of Figure \ref{domaindiv} displays a Shishkin mesh with $N=8$. 
It is uniform and coarse on the region $\dreg:=(\tau,1-\tau)\times(\tau,1-\tau)$, 
but is refined in the layer regions 
\begin{align*}
&\dws:=(0,\tau)\times(0,\tau),
&&\des:=(1-\tau,1)\times(0,\tau),
\\
&\dw:=(0,\tau)\times(\tau,1-\tau),
&&\de:=(1-\tau,1)\times(\tau,1-\tau),
\\
&\dwn:=(0,\tau)\times(1-\tau,1),
&&\den:=(1-\tau,1)\times(1-\tau,1),
\\
&\ds:=(\tau,1-\tau)\times(0,\tau),
&&\dn:=(\tau,1-\tau)\times(1-\tau,1).
\end{align*}
It is sometimes convenient to consider groups of these regions; define
\[
\Omega_i^x:=\bigcup_{j=1}^3\Omega_{ij}
\quad\text{ and } \quad  
\Omega_j^y:=\bigcup_{i=1}^3\Omega_{ij}
\ \text{ for }  i,j=1,2,3.
\]

Set $h_i=x_i-x_{i-1}=y_i-y_{i-1}$ for $i=1,2,\dots,N$.
Then 
\begin{equation*}
h_i=
\begin{dcases}
H: =2(1-2\tau)N^{-1}
= O(N^{-1})
& \text{ for } i=N/4+1,\dots,3N/4,
\\
h:=4\tau N^{-1} = O(\sq^{1/2} N^{-1}\ln N) 
& \text{ for } i=1,\dots,N/4 
\text{ and } i=3N/4+1,\dots,N.
\end{dcases}
\end{equation*}

\begin{figure}[!h]
\begin{minipage}{0.49\linewidth}
\centerline{\includegraphics[width=2.5in,height=2.5in]{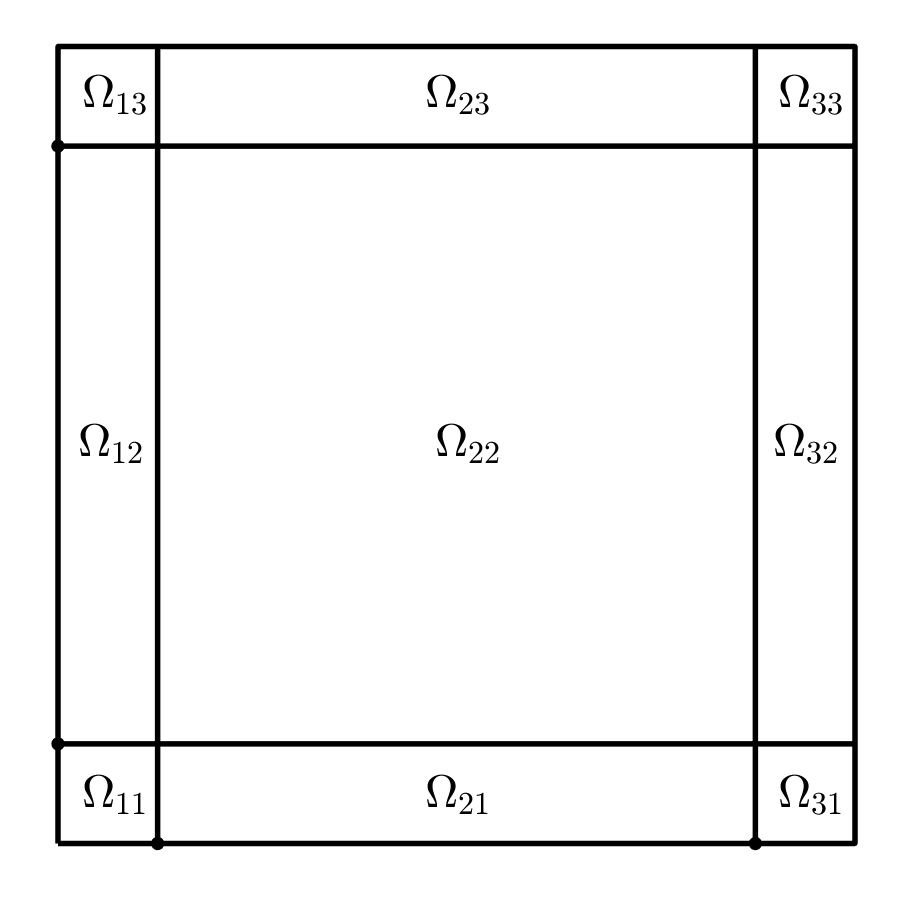}}
\end{minipage}
\begin{minipage}{0.49\linewidth}
\centerline{\includegraphics[width=2.5in,height=2.5in]{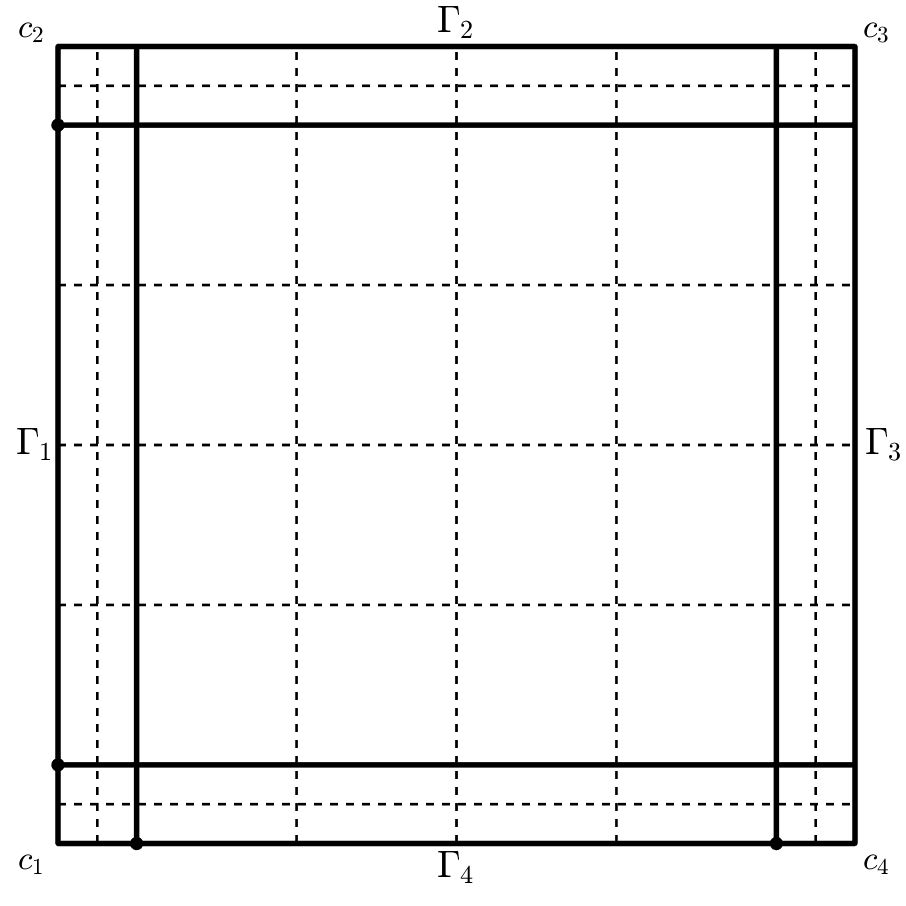}}
\end{minipage}
\caption{Division of $\Omega$ (left) and Shishkin mesh with $N=8$ (right).}
\label{domaindiv}
\end{figure}

\subsection{The LDG method with layer-upwind flux}
\label{sec:LDG:method}

Let $k\ge 0$ be a fixed integer, chosen by the user.
On each bounded interval $I= (a,b)\subset\mathbb{R}$, 
let $\mathcal{P}^{k} (I)$ be the space of polynomials of degree at most $k$ defined on~$I$.
For each $K=I_i\times J_j$, set $\mathcal{Q}^{k}(K)=\mathcal{P}^{k}(I_i)\otimes \mathcal{P}^{k}(J_j)$.
Then define the piecewise continuous finite element space 
\begin{equation*}
\spc= \{v\in L^2 (\Omega)\colon v|_{K} \in \mathcal{Q}^{k}(K)\, \forall K\in\Omega_N \}.
\end{equation*}
Functions in $\spc$ may be discontinuous across element interfaces.
For each $v\in \spc$ and $y\in J_j$, for $j=1,2,\dots,N$
we use $v^\pm_{i,y}(y)=\lim_{x\to x_{i}^\pm}v(x,y)$
and $v^\pm_{x,j}(x)=\lim_{y\to y_{j}^\pm}v(x,y)$
to express the traces on element edges.
Furthermore, for $i=1,2,\dots,N-1$ define the jump $\jump{v}_{i,y}(y)$ 
and the average $\average{v}_{i,y}(y)$ by
\begin{align*}
\jump{v}_{i,y}(y)&:=v_{i,y}^+(y) - v_{i,y}^-(y) 
 &\text{and }& \
\average{v}_{i,y}(y):=\frac12\big(v_{i,y}^+(y) + v_{i,y}^-(y)\big);\\
\text{with}\ 
\jump{v}_{0,y}(y)&:=v_{0,y}^+(y) =: \average{v}_{0,y}(y)
&\text{and }& \
\average{v}_{N,y}(y):=v_{N,y}^-(y) =:-\jump{v}_{N,y}(y).
\end{align*}
In a similar fashion, for each $x\in I_i$ and $i=1,2,\dots,N$,
define the traces $v_{x,j}^{\pm}(x)$,
the jumps $\jump{v}_{x,j}(x)$ and the averages $\average{v}_{x,j}(x)$
on the horizontal edges for $j=0,1,\dots,N$.
To simplify the notation, we omit the dependent variables of these functions and simply write them as $\jump{v}_{i,y}$, 
$\average{v}_{i,y}$, $\jump{v}_{x,j}$ and $\average{v}_{x,j}$ in what follows.

To define the LDG method, rewrite \eqref{R-D:model:1} as the equivalent first-order system
\begin{align}\label{para:rewrite:2d}
-p_x-q_y+bu=f,          \quad
\varepsilon^{-1}p=u_x   \;\text{ and }\;
\varepsilon^{-1}q=u_y    
 \quad \text{  in   }\Omega,
\end{align}
which are subject to the homogeneous Dirichlet boundary condition~\eqref{R-D:model:2}.
By applying the DG method of~\cite{Reed1973}
to~\eqref{para:rewrite:2d} we get the LDG method, 
which we now state precisely. 
Find $ \wph=(\uph,\pph,\qph)\in \spc^3 :=\spc\times \spc\times \spc$ such that in each element $K_{ij}\in\Omega_N$, 
the following variational equations hold true
for each $\vph=(\vphu,\vphp,\vphq)\in \mathcal{V}_N^3$:
\begin{subequations}\label{LDG:scheme:2d}
	\begin{alignat}{1}
	&\dual{b\uph}{\vphu}_{K_{ij}}
	+\dual{\pph}{\vphu_x}_{K_{ij}}
	-\dual{\widehat{\pph}_{i,y}}{\vphu^{-}_{i,y}}_{J_j}
	+\dual{\widehat{\pph}_{i-1,y}}{\vphu^{+}_{i-1,y}}_{J_j}
	\nonumber\\
	\label{vartional:equation:u}
	&+\dual{\qph}{\vphu_y}_{K_{ij}}
	-\dual{\widehat{\qph}_{x,j}}{\vphu^{-}_{x,j}}_{I_i}
	+\dual{\widehat{\qph}_{x,j-1}}{\vphu^{+}_{x,j-1}}_{I_i}
	=\dual{f}{\vphu}_{K_{ij}},
	\\
	\label{vartional:equation:p}
	&\varepsilon^{-1}\dual{\pph}{\vphp}_{K_{ij}}+
	\dual{ \uph}{\vphp_x}_{K_{ij}}
	-\dual{\widehat{\uph}_{i,y}}{\vphp^{-}_{i,y}}_{J_j}
	+\dual{\widehat{\uph}_{i-1,y}}{\vphp^{+}_{i-1,y}}_{J_j}
	=0,
	\\
	\label{vartional:equation:q}
	&\varepsilon^{-1}\dual{\qph}{\vphq}_{K_{ij}}+
	\dual{ \uph}{\vphq_y}_{K_{ij}}
	-\dual{\widehat{\uph}_{x,j}}{\vphq^{-}_{x,j}}_{I_i}
	+\dual{\widehat{\uph}_{x,j-1}}{\vphq^{+}_{x,j-1}}_{I_i}
	=0,	
	\end{alignat}
\end{subequations}
where for each $y\in J_j~(j=1,2,\dots,N)$ and $x\in I_i~(i=1,2,\dots,N)$ 
the numerical fluxes in~\eqref{LDG:scheme:2d} are defined by
\begin{subequations}\label{flux:diffusion:2d}
	\begin{align}
	\label{flux:diffusion:uy:2d}
	\widehat{\uph}_{i,y}
    &
     =\begin{cases}
     0,                  &i=0,\\
    \uph^{-}_{i,y},      &i=1,\dots,N/4,\\
    \average{\uph}_{i,y}, &i=N/4+1,\dots,3N/4-1,\\
    \uph^+_{i,y},         &i=3N/4,\dots,N-1,\\
     0,                  &i=N,
    \end{cases}
    \\
    \label{flux:diffusion:p:2d}
    \widehat{\pph}_{i,y}
    &=\begin{cases}
    \pph^{+}_{i,y},                   &i=0,1,\dots,N/4,\\
    \average{\pph}_{i,y},   &i=N/4+1,\dots,3N/4-1,\\
    \pph^-_{i,y}, &i=3N/4,\dots,N-1,N,\\
    \end{cases} 
    \\
    \label{flux:diffusion:ux:2d}
    \widehat{\uph}_{x,j}
    &=\begin{cases}
    0,                  &j=0,\\
    \uph^{-}_{x,j},      &j=1,\dots,N/4,\\
    \average{\uph}_{x,j}, &j=N/4+1,\dots,3N/4-1,\\
    \uph^+_{x,j},         & j=3N/4,\dots,N-1,\\
    0,                  &j=N,
    \end{cases}
    \\
     \label{flux:diffusion:q:2d}
    \widehat{\qph}_{x,j}
    &=\begin{cases}
    \qph^{+}_{x,j},                   &j=0,1,\dots,N/4,\\
    \average{\qph}_{x,j},             &j=N/4+1,\dots,3N/4-1,\\
    \qph^-_{x,j},                    &j=3N/4,\dots,N-1,N.\\
    \end{cases}
	\end{align}
\end{subequations}

\begin{remark}
In these definitions, note the unusual choice of $\widehat{\pph}_{i,y}$ 
(and similarly $\widehat{\qph}_{x,j}$) for the cases 
$i=1,\dots, N/4$ and $i=3N/4, \dots, N-1$: here we specify the value of $\widehat{\pph}_{i,y}$
using values of $\widehat{\pph}$ on adjacent elements 
that lie further from the boundary~$\partial\Omega$, i.e., that lie in the direction where the layer weakens.
This choice is reminiscent of the technique of ``upwinding" approximations of 1st-order derivatives when one
computes solutions of convection-dominated problems, and $\widehat{\pph}$ is an approximation
of the 1st-order derivative~$\varepsilon u_x$, so we call the numerical flux in~\eqref{flux:diffusion:2d} 
a \emph{layer-upwind} flux. On the coarse mesh ($i=N/4+1, \dots, 3N/4-1$) we use a standard central flux
with the aim of deriving a sharper convergence result on
the coarse mesh domain~$\dreg$; cf.~\cite{ChengWangStynes2024JSC}.
\end{remark}

Our layer-upwind flux works well from both the practical and theoretical points of view:
it often yields better accuracy than other choices such as a globally central flux or an alternating flux,
and it enables us to prove an optimal-order error bound in a balanced norm.

Note that our LDG method has a certain degree of symmetry; furthermore, our method
does not use any penalty terms when defining the flux 
for singularly perturbed problems, unlike many other LDG papers \cite{ChengJiangStynes2023,ChengStynes2023,ChengWangStynes2024JSC,CYM2022,Zhu:2dMC}.

To write the LDG method in a compact form, sum \eqref{LDG:scheme:2d} over all $K_{ij}\in \Omega_N$: Find $\wph=(\uph,\pph,\qph)\in \spc^3$  such that 
$B(\wph;\vph)=\dual{f}{\vphu}$ for all $\vph=(\vphu,\vphp,\vphq)\in \spc^3$,
where
\begin{equation}\label{B:def:2d}
\begin{aligned}
&B(\wph;\vph) :=\dual{b\uph}{\vphu}+\sq^{-1}\dual{\pph}{\vphp}+
\sq^{-1}\dual{\qph}{\vphq}
\\
&\quad
+\dual{\uph}{\vphp_x}
+\sum_{j=1}^{N}\sum_{i=1}^{N-1}
\dual{\widehat{\uph}_{i,y}}{\jump{\vphp}_{i,y}}_{J_j}
+\dual{\uph}{\vphq_y}+
\sum_{i=1}^{N}\sum_{j=1}^{N-1}\dual{\widehat{\uph}_{x,j}}{\jump{\vphq}_{x,j}}_{I_i}
\\
&\quad
+\dual{\pph}{\vphu_x}
+\sum_{j=1}^{N}\left[\sum_{i=1}^{N-1}
\dual{\widehat{\pph}_{i,y}}{\jump{\vphu}_{i,y}}_{J_j}
+\dual{\pph_{0,y}^+}{\vphu_{0,y}^+}_{J_j}
-\dual{\pph_{N,y}^-}{\vphu_{N,y}^-}_{J_j}
\right]
\\
&\quad
+\dual{\qph}{\vphu_y}
+\sum_{i=1}^{N}\left[\sum_{j=1}^{N-1}
\dual{\widehat{\qph}_{x,j}}{\jump{\vphu}_{x,j}}_{I_i}
+\dual{\qph_{x,0}^+}{\vphu_{x,0}^+}_{I_i}
-\dual{\qph_{x,N}^-}{\vphu_{x,N}^-}_{I_i}
\right]
\end{aligned}
\end{equation}
and the ``hat" terms are defined in~\eqref{flux:diffusion:2d}.

Using the notation of~\eqref{para:rewrite:2d}, we write the solution of~\eqref{spp:R-D:2d}
as $\bm w := (u,p,q)$, analogously to $\wph=(\uph,\pph,\qph)$.
One can then define $B(\bm w;\vph)$ for all  $\vph\in \spc^3$ analogously to~\eqref{B:def:2d}
and the consistency of the LDG scheme follows:
for the solution $\bm w\in C^{1}(\bar\Omega)\times C^{1}(\bar\Omega)\times C^{1}(\bar\Omega)$,
\[
B(\bm w;\vph)=\dual{f}{\vphu}\quad \text{   for  all }
\vph=(\vphu,\vphp,\vphq)\in \spc^3.
\]
Equivalently, we have the Galerkin orthogonality property
\begin{align}
\label{Galerkin:orthogonality}
B(\bm w-\wph;\vph)=0\quad  \text{   for  all }
\vph=(\vphu,\vphp,\vphq)\in \spc^3.
\end{align}

Define an energy norm $\enorm{\cdot}$ on $\spc^3$
by $\enorm{\vph}^2=B(\vph;\vph)$ for each $\vph=(\vphu,\vphp,\vphq)\in \spc^3$.
Using integration by parts and the simple identities 
\[
\jump{\vphu\vphq}
=\vphu^-\jump{\vphq}+\vphq^+\jump{\vphu}
=\vphu^+\jump{\vphq}+\vphq^-\jump{\vphu}
=\average{\vphu}\jump{\vphq}
+\average{\vphq}\jump{\vphu}
\]
at each interior element interface, one obtains 
\begin{equation}\label{energy:norm:2d}
\enorm{\vph}^2
=\sq^{-1}\left(\norm{\vphp}^2
+\norm{\vphq}^2\right)
+\norm{b^{1/2}\vphu}^2.
\end{equation}
One can apply this definition also to $\bm w = (u,p,q)$: 
\[
\enorm{\bm w}^2 =\sq^{-1}\left(\norm{p}^2 +\norm{q}^2\right) +\norm{b^{1/2}u}^2.
\]

The norm $\enorm{\cdot}$ seems natural for our LDG method, but 
(recall the decomposition $u = \bar u + u_\sq$ of Lemma~\ref{lemma:prop:2d}) 
the contribution to $\enorm{\bm w}$ of its layer component
${\bm w}_{\sq}:=(u_\sq,\sq \partial_x u_\sq,\sq \partial_y u_\sq)$
is essentially lost when $\sq$ is small
because a short calculation shows that
$\enorm{{\bm w}_{\sq}}=O(\sq^{1/4})$, which vanishes as $\varepsilon\rightarrow 0$.
To address this shortcoming of the energy norm, 
we change the scaling $\sq^{-1}$ of the auxiliary variable terms to $\sq^{-3/2}$
to define the \emph{balanced norm} $\ba{\cdot}$ by 
(cf.~\cite{CYM2022,HeuerKarkulik2017,LinStynes2012,RoosSchopf2015})
\begin{subequations}\label{balanced:norm:2d}
\begin{align}
\ba{\vph}^2&=\sq^{-3/2}\left(\norm{\vphp}^2
+\norm{\vphq}^2\right)+\norm{b^{1/2}\vphu}^2.
\\
\intertext{ and } 
\ba{\bm w}^2 &=\sq^{-3/2}\left(\norm{p}^2 +\norm{q}^2\right) +\norm{b^{1/2}u}^2. 
\end{align}
\end{subequations}
This norm treats both components of the true solution equally: a calculation shows that
$\ba{{\bm w}_{\sq}}=O(1)$ and also $\ba{\bar{\bm w}}=O(1)$ 
for the regular component
$\bar{\bm w}:=(\bar u,\sq \partial_x \bar u,\sq \partial_y \bar u)$
of~$\bm w$.

\section{Projectors: stability and approximation properties}
\label{sec:projectors}

We now begin our error analysis of the LDG method \eqref{LDG:scheme:2d} 
using the balanced norm \eqref{balanced:norm:2d}.
In Section~\ref{sec:projectors} we define various projectors that will be used in our error analysis,
then investigate their basic approximation properties.
(Superapproximation properties of the projectors will be described in Section~\ref{sec:superapproximation}.)
It takes some time to establish these results, but they are needed for the error analysis of 
Section~\ref{section:error:analysis}.

\subsection{Construction of projectors}\label{sec:projectorconstruction}
We begin by introducing some 1D projectors that are used later
to define various 2D projectors.
Let $i\in\{1,\dots,N\}$.   
Define the $L^2$-projector $\pi: L^2(I_i) \to \mathcal{P}^{k}(I_i)$  for each $z\in L^2(I_{i})$ by
$\dual{\pi z}{\vphu}_{I_i}= \dual{z}{\vphu}_{I_i}$ for all $\vphu\in \mathcal{P}^{k}(I_i)$.
Then define two Gauss-Radau projectors $\pi^{\pm}:C(\bar{I}_i) \to \mathcal{P}^{k}(I_i)$ by
\begin{align*}
\dual{\pi^{+} z}{\vphu}_{I_{i}} &= \dual {z}{\vphu}_{I_{i}} \ \forall \vphu\in \mathcal{P}^{k-1}(I_{i})
	\ \text{ and }\  (\pi^{+} z)^+_{i-1} = z^+_{i-1}; \\
\dual{\pi^{-} z}{\vphu}_{I_{i}} &= \dual {z}{\vphu}_{I_{i}} \ \forall \vphu\in \mathcal{P}^{k-1}(I_{i})
	\ \text{ and }\  (\pi^{-} z)^-_{i} = z^-_{i}.
\end{align*}

Let $K_{ij}=I_i\times J_j$ be a mesh rectangle.
The two-dimensional $L^2$-projector 
$\Pi: L^2(K_{ij}) \to \mathcal{Q}^{k}(K_{ij})$ is defined for each $z\in L^2(K_{ij})$ by
\begin{subequations}\label{2d:projectors}
\begin{align}\label{L2:prj:2d}
\dual{\Pi z}{\vphu}_{K_{ij}}=\dual{z}{\vphu}_{K_{ij}}
\quad \forall \vphu\in \mathcal{Q}^{k}(K_{ij}).
\end{align}
Define the following tensor-product projectors mapping from
$C(\bar K_{ij})$  to $\mathcal{Q}^{k}(K_{ij})$:
\begin{align}
\label{GR:prj:2d:1}
\Pi_x^{+}&:=\pi^+_x \otimes \pi_y, 
&&\Pi_x^{-}:=\pi^-_x \otimes \pi_y,
&&\Pi_y^{+}:=\pi_x \otimes \pi^+_y,
&&\Pi_y^{-}:=\pi_x \otimes \pi^-_y;
\\
\label{GR:prj:2d:2}
\Pi_{xy}^{+}&:=\pi^+_x \otimes \pi^+_y, 
&&\Pi_{xy}^{\pm}:=\pi^+_x \otimes \pi^-_y, 
&&\Pi_{xy}^{\mp}:=\pi^-_x \otimes \pi^+_y, 
&&\Pi_{xy}^{-}:=\pi^-_x \otimes \pi^-_y,
\end{align}
\end{subequations}
where a subscript $x$ or $y$ of the 1D projectors $\pi, \pi^-$ and $\pi^+$ 
refers to the spatial coordinate in which this projector is applied.
These 2D projectors can be described as local Gauss-Radau projectors;
several of them were discussed in
\cite{ChengJiangStynes2023,Cheng2021:Calcolo,ChengStynes2023,CYM2022,Zhu:2dMC}. 
To analyse them systematically, we regard them as
 ``vertex-edge-element approximation" operators
and divide them into three categories that are  
indicated by the symbols ``o", ``x" and ``$\bullet$" in Figure~\ref{Fig:projectors}, as we now explain.

(1) The symbol ``o" indicates ``element approximation",
which corresponds to 
the $L^2$ projector $\Pi$ satisfying the element orthogonality condition \eqref{L2:prj:2d}.

(2) The symbol ``x" indicates ``edge-element approximation",
which corresponds to the projectors
$\Pi_x^{+}, \Pi_x^{-}, \Pi_y^{+}$ and $\Pi_y^{-}$
satisfying respectively
\begin{subequations}\label{GR:single:direction}
\begin{align}
\label{Pi:x:+}
&
\begin{cases}
\dual{\Pi_x^{+}z}{\vphu}_{K_{ij}}
=\dual{z}{\vphu}_{K_{ij}}
&\forall \vphu\in \mathcal{P}^{k-1}(I_{i})\otimes \mathcal{P}^{k}(J_j),
\\
\dual{(\Pi_x^{+}z)_{i-1,y}^{+}}{\vphu}_{J_j}
=\dual{z^{+}_{i-1,y}}{\vphu}_{J_j}
&\forall \vphu\in \mathcal{P}^{k}(J_j);
\end{cases}
\\
\label{Pi:x:-}
&
\begin{cases}
\dual{\Pi_x^{-}z}{\vphu}_{K_{ij}}
=\dual{z}{\vphu}_{K_{ij}}
&\qquad\forall \vphu\in \mathcal{P}^{k-1}(I_{i})\otimes \mathcal{P}^{k}(J_j),
\\
\dual{(\Pi_x^{-}z)_{i,y}^{-}}{\vphu}_{J_j}
=\dual{z^{-}_{i,y}}{\vphu}_{J_j}
&\qquad\forall \vphu\in \mathcal{P}^{k}(J_j);
\end{cases}
\\
\label{Pi:y:+}
&
\begin{cases}
\dual{\Pi_y^{+}z}{\vphu}_{K_{ij}}
=\dual{z}{\vphu}_{K_{ij}}
&\forall \vphu\in \mathcal{P}^{k}(I_{i})\otimes \mathcal{P}^{k-1}(J_j),
\\
\dual{(\Pi_y^{+} z)_{x,j-1}^{+}}{\vphu}_{I_i}
=\dual{z_{x,j-1}^{+}}{\vphu}_{I_i}
&\forall \vphu\in \mathcal{P}^{k}(I_i);
\end{cases}
\\
\label{Pi:y:-}
&
\begin{cases}
\dual{\Pi_y^{-}z}{\vphu}_{K_{ij}}
=\dual{z}{\vphu}_{K_{ij}}
&\qquad\forall \vphu\in \mathcal{P}^{k}(I_{i})\otimes \mathcal{P}^{k-1}(J_j),
\\
\dual{(\Pi_y^{-} z)_{x,j}^{-}}{\vphu}_{I_i}
=\dual{z_{x,j}^{-}}{\vphu}_{I_i}
&\qquad\forall \vphu\in \mathcal{P}^{k}(I_i).
\end{cases}
\end{align}
\end{subequations}

(3) The symbol ``$\bullet$" indicates ``vertex-edge-element approximation", which corresponds to the projectors
$\Pi_{xy}^{+},\Pi_{xy}^{\pm},\Pi_{xy}^{\mp}$ and $\Pi_{xy}^{-}$
satisfying respectively
\begin{align*}
&
\begin{cases}
\dual{\Pi_{xy}^{+}z}{\vphu}_{K_{ij}}=\dual{z}{\vphu}_{K_{ij}}
&\forall \vphu\in \mathcal{Q}^{k-1}(K_{ij});
\\
\dual{(\Pi_{xy}^{+} z)_{i-1,y}^{+}}{\vphu}_{J_j}
=\dual{z_{i-1,y}^{+}}{\vphu}_{J_j}
&\forall \vphu\in \mathcal{P}^{k-1}(J_j);
\\
\dual{(\Pi_{xy}^{+} z)_{x,j-1}^{+}}{\vphu}_{I_i}
=\dual{z_{x,j-1}^{+}}{\vphu}_{I_i}
&\forall \vphu\in \mathcal{P}^{k-1}(I_i);
\\
(\Pi_{xy}^{+} z)(x_{i-1}^{+},y_{j-1}^{+})
=z(x_{i-1}^{+},y_{j-1}^{+});
\end{cases}
\\
&
\begin{cases}
\dual{\Pi_{xy}^{\pm}z}{\vphu}_{K_{ij}}
=\dual{z}{\vphu}_{K_{ij}}
&\quad\forall \vphu\in \mathcal{Q}^{k-1}(K_{ij});
\\
\dual{(\Pi_{xy}^{\pm} z)_{i-1,y}^{+}}{\vphu}_{J_j}
=\dual{z_{i-1,y}^{+}}{\vphu}_{J_j}
&\quad\forall \vphu\in \mathcal{P}^{k-1}(J_j);
\\
\dual{(\Pi_{xy}^{\pm} z)_{x,j}^{-}}{\vphu}_{I_i}
=\dual{z_{x,j}^{-}}{\vphu}_{I_i}
&\quad\forall \vphu\in \mathcal{P}^{k-1}(I_i);
\\
(\Pi_{xy}^{\pm} z)(x_{i-1}^{+},y_{j}^{-}) 
=z(x_{i-1}^{+},y_{j}^{-});
\end{cases}
\\
&
\begin{cases}
\dual{\Pi_{xy}^{\mp}z}{\vphu}_{K_{ij}}
=\dual{z}{\vphu}_{K_{ij}}
&\quad\forall \vphu\in \mathcal{Q}^{k-1}(K_{ij});
\\
\dual{(\Pi_{xy}^{\mp} z)_{i,y}^{-}}{\vphu}_{J_j}
=\dual{z_{i,y}^{-}}{\vphu}_{J_j}
&\quad\forall \vphu\in \mathcal{P}^{k-1}(J_j);
\\
\dual{(\Pi_{xy}^{\mp} z)_{x,j-1}^{+}}{\vphu}_{I_i}
=\dual{z_{x,j-1}^{+}}{\vphu}_{I_i}
&\quad\forall \vphu\in \mathcal{P}^{k-1}(I_i);
\\
(\Pi_{xy}^{\mp} z)(x_{i}^{-},y_{j-1}^{+}) 
=z(x_{i}^{-},y_{j-1}^{+});
\end{cases}
\\
&
\begin{cases}
\dual{\Pi_{xy}^{-}z}{\vphu}_{K_{ij}}
=\dual{z}{\vphu}_{K_{ij}}
&\qquad\quad\forall \vphu\in \mathcal{Q}^{k-1}(K_{ij});
\\
\dual{(\Pi_{xy}^{-} z)_{i,y}^{-}}{\vphu}_{J_j}
=\dual{z_{i,y}^{-}}{\vphu}_{J_j}
&\qquad\quad\forall \vphu\in \mathcal{P}^{k-1}(J_j);
\\
\dual{(\Pi_{xy}^{-} z)_{x,j}^{-}}{\vphu}_{I_i}
=\dual{z_{x,j}^{-}}{\vphu}_{I_i}
&\qquad\quad \forall \vphu\in \mathcal{P}^{k-1}(I_i);
\\
(\Pi_{xy}^{-} z)(x_{i}^{-},y_{j}^{-}) 
=z(x_{i}^{-},y_{j}^{-}).
\end{cases}
\end{align*}

For our error analysis of the LDG method, we make various combinations of the projectors
\eqref{2d:projectors} to give three composite projections of $u,p$ and $q$ into $\spc$:
\begin{align}\label{prj:u:2d}
\prou u&
=
\begin{cases}
\Pi_{xy}^{-}u   & \text{in } \dws,\\
\Pi_{y}^{-}u    & \text{in } \ds,\\
\Pi_{xy}^{\pm}u & \text{in } \des,\\
\Pi_{x}^{-}u    & \text{in } \dw,\\
\Pi u            & \text{in } \dreg,\\
\Pi_{x}^{+}u    & \text{in } \de,\\
\Pi_{xy}^{\mp}u & \text{in } \dwn,\\
\Pi_y^{+} u     & \text{in } \dn,\\
\Pi_{xy}^{+}u   & \text{in } \den,
\end{cases}
\quad
\prop p
=
\begin{cases}
\Pi_{x}^{+}p & \text{in } \dxa,\\
\Pi  p       & \text{in } \dxb,\\
\Pi_{x}^{-} p& \text{in } \dxc,
\end{cases}
\quad
\text{ and }\quad 
\proq q
=
\begin{cases}
\Pi_{y}^{+}q & \text{in } \dya,\\
\Pi q& \text{in }         \dyb,\\
\Pi_{y}^{-}q & \text{in } \dyc.
\end{cases}
\end{align}
Figure~\ref{Fig:projectors}  gives diagrammatic representations of 
$\prou, \prop$ and $\proq$.

\begin{figure}[h]
	\begin{minipage}{0.49\linewidth}
		\vspace{2pt}
		\centerline{\includegraphics[width=0.8\textwidth]{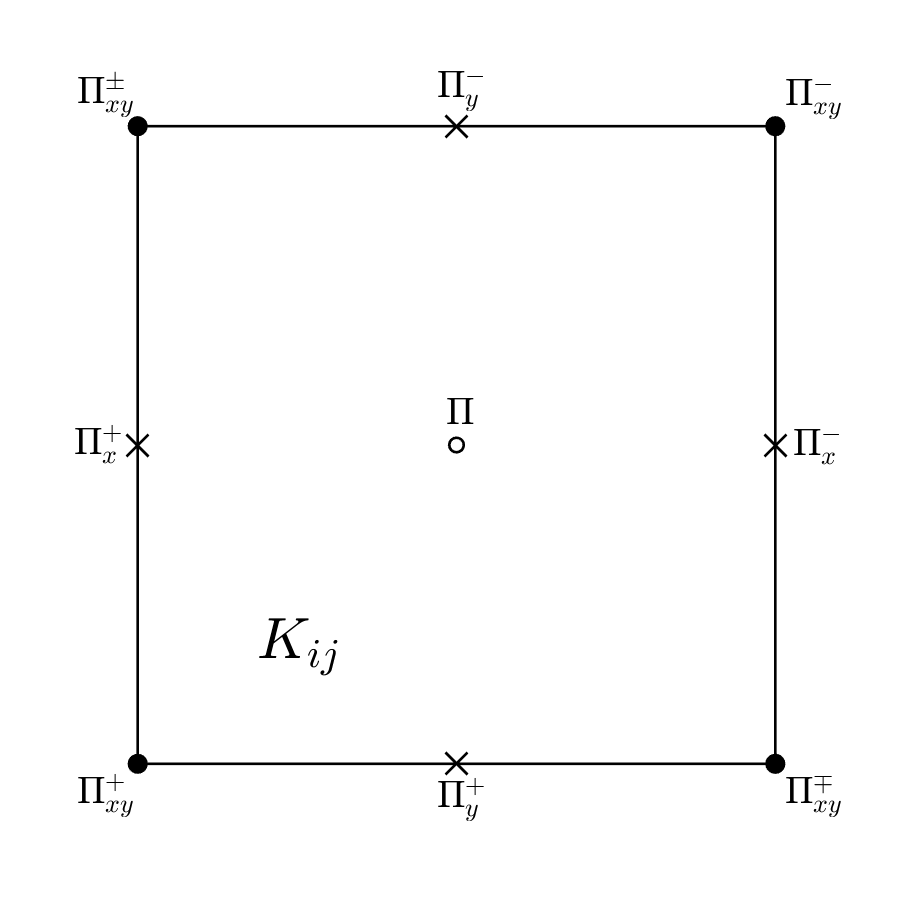}}
		\centerline{Vertex-edge-element approximation}
		\vspace{2pt}
		\centerline{\includegraphics[width=0.8\textwidth]{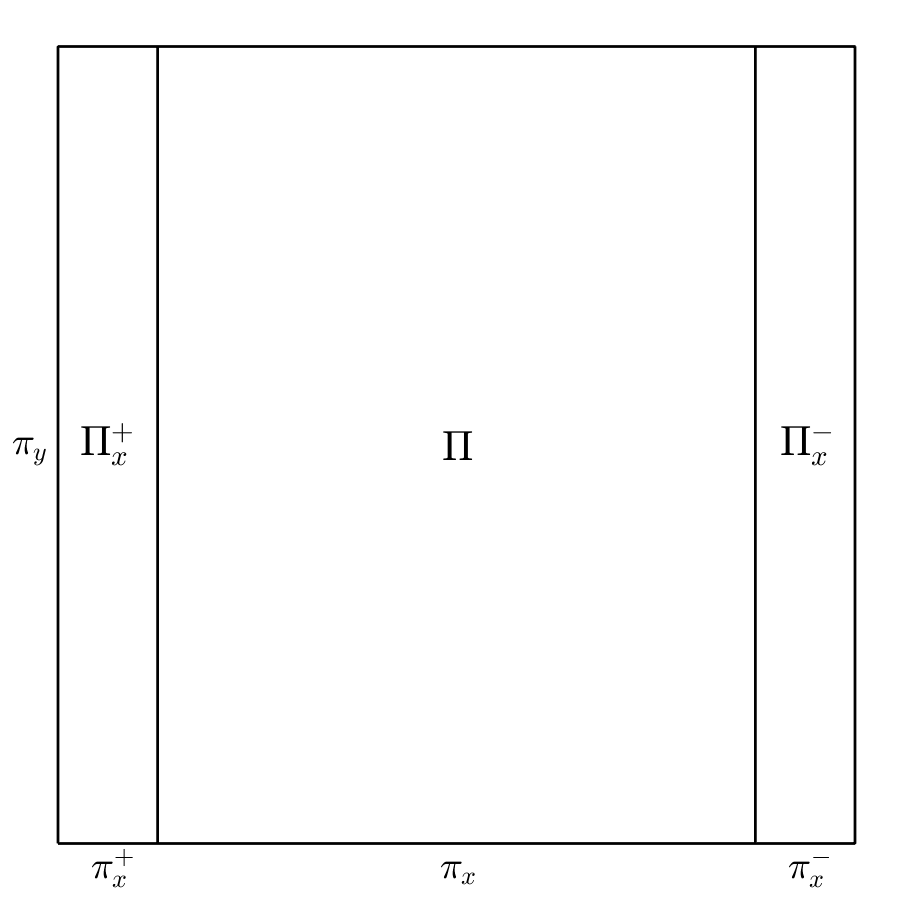}}
		\centerline{Projector $\prop$}
	\end{minipage}
	\begin{minipage}{0.49\linewidth}
		\vspace{2pt}
		\centerline{\includegraphics[width=0.8\textwidth]{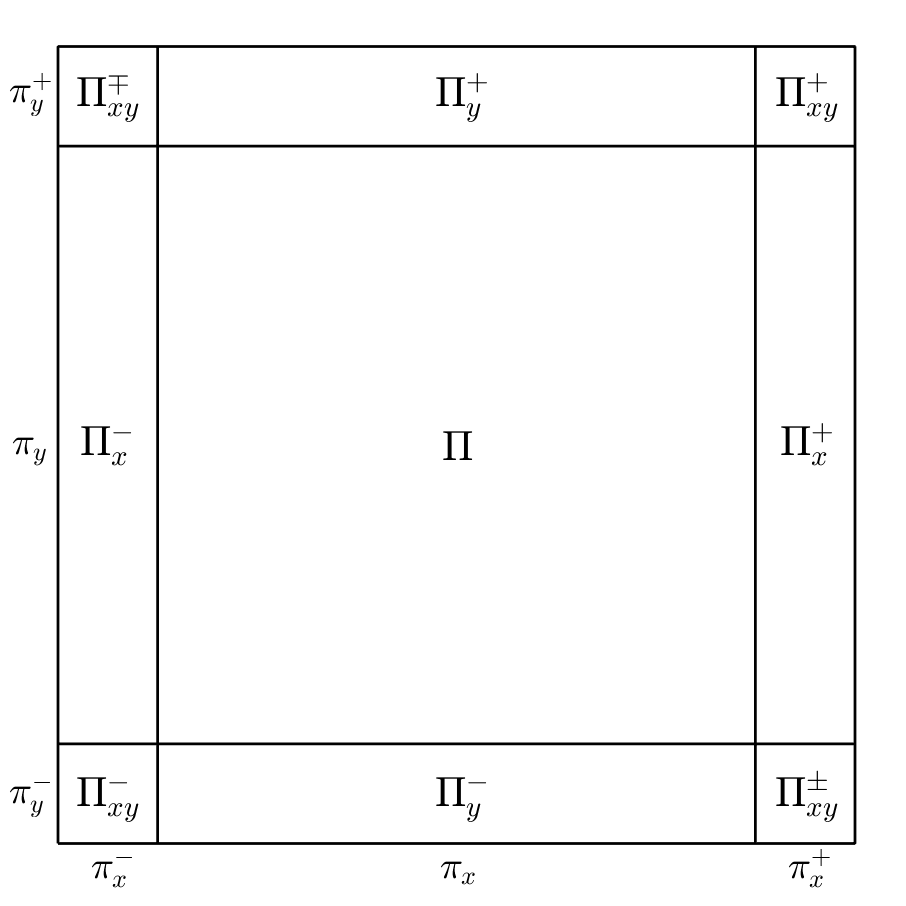}}
		\centerline{Projector $\prou$}
		\vspace{2pt}
		\centerline{\includegraphics[width=0.8\textwidth]{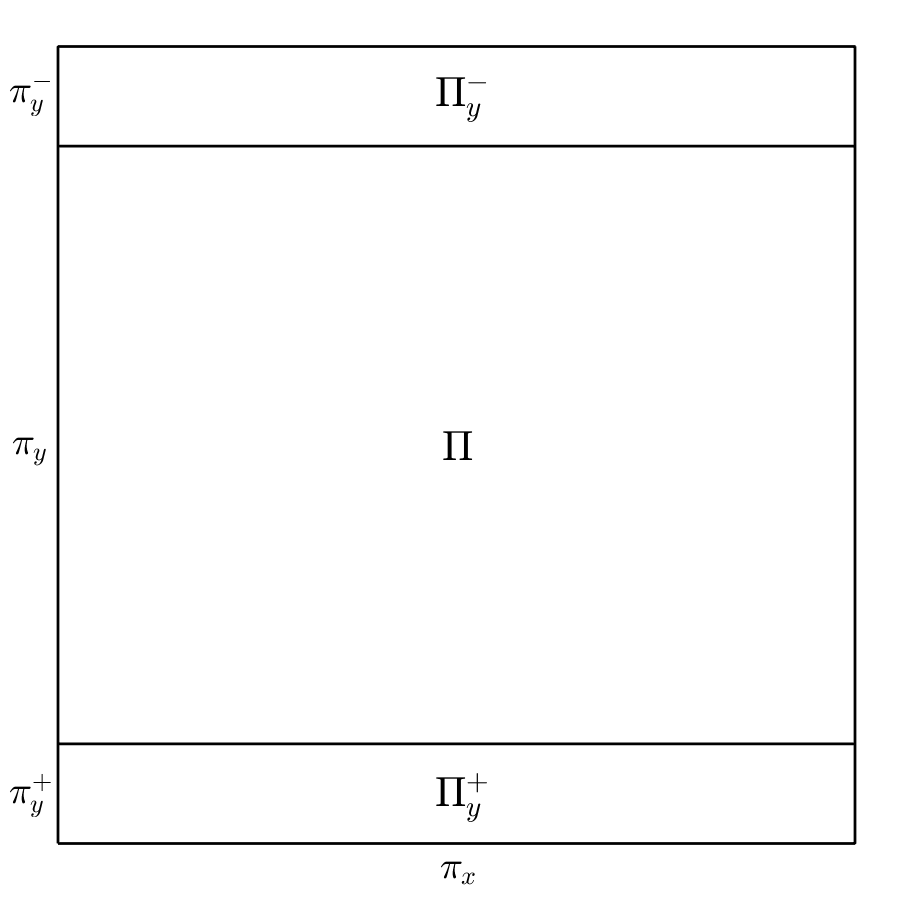}}
		\centerline{Projector $\proq$}
	\end{minipage}
	\caption{Vertex-edge-element approximation and three projectors $\prou$, $\prop$ and $\proq$.}
	\label{Fig:projectors}
\end{figure}

\begin{remark}
The projectors $\prou, \prop$ and $\proq$ are constructed to yield  
the energy-norm superclose property~\eqref{xisuperclose}, which shows
that the LDG solution $\wph=(\uph,\pph,\qph)\in \spc^3$ lies very close to the projection
$(\prou u,\prop p,\proq q)\in\spc^3$ of the true solution $(u,p,q)$ of~\eqref{spp:R-D:2d}.
\end{remark}  

\subsection{Stability and approximation properties of projectors}
Let $\prog\in\{\Pi,\Pi_x^{-},\Pi_x^{+},\Pi_y^{-},\Pi_y^{+},
\Pi_{xy}^{+},\Pi_{xy}^{\pm},\Pi_{xy}^{\mp}, \Pi_{xy}^{-}\}$.
	Similarly to \cite[Lemma 5]{Cheng2021:Calcolo},
	one  has the following stability properties in the $L^{\infty}$ and $L^2$ norms
	(here $C>0$ is some generic constant as described at the end of Section~\ref{sec:intro}):
	\begin{subequations}\label{2GR:stb:app}
		\begin{align}\label{L0:stb}
		\norm{\prog z}_{L^\infty(K_{ij})} 
		&\leq  C \norm{z}_{L^\infty(K_{ij})},
		\\
		\label{L2:stb:L2}
		\norm{\Pi z}_{K_{ij}}&\leq  \norm{z}_{K_{ij}},
		\\
		\label{L2:stb:GR1}
		\norm{\Pi^-_x z}_{K_{ij}}&\leq  C
		\Big[
		\norm{z}_{K_{ij}} 
		+h^{1/2}_{i}\norm{z^-_{i,y}}_{J_j}
		\Big],
		\\
		\label{L2:stb:GR2}
		\norm{\Pi^-_y z}_{K_{ij}}&\leq  C
		\Big[
		\norm{z}_{K_{ij}} + h^{1/2}_{j}\norm{z^-_{x,j}}_{I_i}
		\Big],
		\\
		\label{L2:stb:GR3}
		\norm{\Pi^-_{xy} z}_{K_{ij}}&\leq  C
		\Big[
		\norm{z}_{K_{ij}} + h^{1/2}_{j}\norm{z^-_{x,j}}_{I_i}
		+h^{1/2}_{i}\norm{z^-_{i,y}}_{J_j}
		+(h_{i}h_{j})^{1/2}|z^{-}_{ij}|
		\Big],
		\end{align}
	\end{subequations}
where $z^{-}_{ij}:=z(x_i^-,y_j^-)$.
Analogous bounds are valid for the remaining projectors 
$\Pi_x^{+}$, $\Pi_y^{+}$, $\Pi_{xy}^{\pm}$, $\Pi_{xy}^{\mp}$ and $\Pi_{xy}^{+}$. Furthermore, 
one has the following approximation properties 
(see, e.g., \cite[Lemma 3]{Cheng2021:Calcolo} and \cite[Lemma 4.3]{Zhu:2dMC}
for the case $k>0$,  and \cite[Lemma~3.8]{GeorgoulisThesis} and \cite[p.60]{Georgoulis06} 
for the case $k=0$):
	\begin{subequations}\label{2GR:stb:app:2}
		\begin{align}
		\label{GR:app}
		\norm{z-\prog z}_{L^\infty(K_{ij})}
		&\leq 
		C \left[h_{i}^{k+1}\norm{\partial_x^{k+1}z}_{L^\infty(K_{ij})}
		+h_{j}^{k+1}\norm{\partial_y^{k+1}z}_{L^\infty(K_{ij})}
		\right],
		\\
		\label{GR:L^2:app}
		\norm{z-\prog z}_{K_{ij}}
		&\leq 
		C \left[h_{i}^{k+1}\norm{\partial_x^{k+1}z}_{K_{ij}}
		+h_{j}^{k+1}\norm{\partial_y^{k+1}z}_{K_{ij}}
		\right].
		\end{align}
	\end{subequations}

Now we turn to $\prou, \prop$ and $\proq$.
Let $\eta_u:=u-\prou u, \, \eta_p:=p-\prop p$ and $\eta_q:=q-\proq q$.
These projection errors satisfy the following bounds.

\begin{lemma}\label{lemma:app:2d}
Assume that Lemma~\ref{lemma:prop:2d} is valid for $m=k$.
Then for the solution $(u,p,q)$ of~\eqref{para:rewrite:2d}, 
   there exists a constant $C>0$ such that
   \begin{subequations}
	\begin{align}
	\label{bound:eu:L2:2d}
	&\norm{\eta_u}_{\Omega\backslash\dreg}\leq C\sq^{1/4}(N^{-1}\ln N)^{k+1},
	 \norm{\eta_u}_{\dreg}\leq C\Big[N^{-(k+1)}+\sq^{1/4}N^{-\sigma}\Big], 
	\\
	\label{bound:eu:L0:2d}
	&\norm{\eta_u}_{L^\infty(\Omega\backslash\dreg)}\leq C(N^{-1}\ln N)^{k+1},
	\norm{\eta_u}_{L^\infty(\dreg)}\leq CN^{-(k+1)},
	\\
	\label{bound:ep:L2:2d}
	&\norm{\eta_p}_{\dxa\cup\dxc}\leq C\sq^{3/4}(N^{-1}\ln N)^{k+1},
	\norm{\eta_p}_{\dxb} \leq C\Big[\sq N^{-(k+1)}+\sq^{1/2}N^{-\sigma}\Big], 
	\\
	\label{bound:ep:L0:2d}
	&\norm{\eta_p}_{L^\infty(\dxa\cup\dxc)}\leq C\sq^{1/2}(N^{-1}\ln N)^{k+1},
	\; \norm{\eta_p}_{L^\infty(\dxb)}\leq C\sq^{1/2} N^{-(k+1)},
	\\
	\label{bound:eq:L2:2d}
	&\norm{\eta_q}_{\dya\cup\dyc}\leq C\sq^{3/4}(N^{-1}\ln N)^{k+1},\
	\norm{\eta_q}_{\dyb} \leq C\Big[\sq N^{-(k+1)}+\sq^{1/2}N^{-\sigma}\Big], 
	\\
	\label{bound:eq:L0:2d}
	&\norm{\eta_q}_{L^\infty(\dya\cup\dyc)}\leq C\sq^{1/2}(N^{-1}\ln N)^{k+1},
	\; \norm{\eta_q}_{L^\infty(\dyb)} \leq C\sq^{1/2} N^{-(k+1)}.
	\end{align}
	\end{subequations}
\end{lemma}

Similar inequalities were derived in \cite[Lemmas 3.1 and 3.2]{ChengStynes2023} 
for the projectors $\Pi^{-}_{xy}$, $\Pi^{+}_{x}$ and $\Pi^{+}_{y}$
applied to a problem whose solution exhibits exponential and characteristic layers. 
We nevertheless provide a proof of Lemma~\ref{lemma:app:2d} 
because our projections $\prou u, \, \prop p$ and $\proq q$ 
have different definitions in different subregions.

\emph{Proof of \eqref{bound:eu:L2:2d}:} For each component $u_z$ of $u$, set
$\eta_{u_z}=u_z-\prou u_z$. Then $\eta_{u}=\eta_{\bar{u}}
+\sum_{i=1}^4\eta_{u^{\mathrm{b}}_i}+\sum_{i=1}^4\eta_{u^{\mathrm{c}}_i}$.
In our calculations we consider only  $\eta_{\bar{u}}, \eta_{u^{\mathrm{b}}_1}$
and $\eta_{u^{\mathrm{c}}_1}$ as the other terms are handled similarly.
Using \eqref{GR:L^2:app}, \eqref{reg:smooth} and the measures of subregions
$\Omega\backslash\dreg$ and $\dreg$, one obtains easily 
\begin{align}\label{app:smooth:L2}
\norm{\eta_{\bar{u}}}_{\Omega\backslash\dreg}\leq C\tau^{1/2}N^{-(k+1)}\leq C\sq^{1/4}N^{-(k+1)}(\ln N)^{1/2} 
\ \text{ and }\
\norm{\eta_{\bar{u}}}_{\dreg}\leq CN^{-(k+1)}.
\end{align}
	
For the boundary layer component $u^{\mathrm{b}}_1$, if $K_{ij}\in\dxa\cup\dxc$ 
then the $L^2$-approximation property \eqref{GR:L^2:app} 
and the bound \eqref{reg:boundary} yield 
\begin{align*}
\norm{\eta_{u^{\mathrm{b}}_1}}_{K_{ij}}
&
\leq C\left[h_{i}^{k+1}\norm{\partial_x^{k+1} u^{\mathrm{b}}_1}_{K_{ij}}
+h_{j}^{k+1}\norm{\partial_y^{k+1} u^{\mathrm{b}}_1}_{K_{ij}}
\right]
\\
&\leq C\left\{h^{k+1}\Big[
\int_{K_{ij}} \sq^{-(k+1)}e^{-2\beta x/\sqrt{\varepsilon}}\,\mathrm{d}x\,\mathrm{d}y
\Big]^{1/2}
+H^{k+1}\Big[
\int_{K_{ij}} e^{-2\beta x/\sqrt{\varepsilon}}\,\mathrm{d}x\,\mathrm{d}y
\Big]^{1/2}
\right\}
\\
&\leq C(N^{-1}\ln N)^{k+1}\norm{e^{-\beta x/\sqrt{\varepsilon}}}_{K_{ij}}.
\end{align*}
Hence,
\begin{align}\label{boundary:1:layer}
\norm{\eta_{u^{\mathrm{b}}_1}}_{\dxa\cup\dxc}
\leq C(N^{-1}\ln N)^{k+1} \norm{e^{-\beta x/\sqrt{\varepsilon}}}
\leq C\sq^{1/4}(N^{-1}\ln N)^{k+1}.
\end{align}
If $K_{ij}\in\ds$, then $\eta_{u^{\mathrm{b}}_1}
=u^{\mathrm{b}}_1-\Pi_y^{-}u^{\mathrm{b}}_1$
and the $L^2$-stability property \eqref{L2:stb:GR2} 
and the bound \eqref{reg:boundary} give
\begin{align*}
\norm{\eta_{u^{\mathrm{b}}_1}}_{K_{ij}}
&\leq C\left[\norm{u^{\mathrm{b}}_1}_{K_{ij}}
+h_j^{1/2}\norm{(u^{\mathrm{b}}_1)^{-}_{x,j}}_{I_{i}}
\right]
\\
&\leq C\left\{\Big[
\int_{K_{ij}} e^{-2\beta x/\sqrt{\varepsilon}}\,\mathrm{d}x\,\mathrm{d}y
\Big]^{1/2}
+h_j^{1/2}\Big[
\int_{I_{i}} e^{-2\beta x/\sqrt{\varepsilon}}\,\mathrm{d}x\Big]^{1/2}
\right\} 
\\
&\leq C\norm{e^{-\beta x/\sqrt{\varepsilon}}}_{K_{ij}},
\end{align*}
which implies
\begin{align}\label{boundary:2:layer}
\norm{\eta_{u^{\mathrm{b}}_1}}_{\ds}
&\leq C\norm{e^{-\beta x/\sqrt{\varepsilon}}}_{\ds}
\leq C\tau^{1/2}
\left(\int_{\tau}^{1-\tau} e^{-2\beta x/\sqrt{\varepsilon}} 
\,\mathrm{d}x\right)^{1/2}
\nonumber\\
&=C\tau^{1/2}\Big[\frac{\ssq}{2\beta}e^{-2\beta \tau/\sqrt{\varepsilon}}
(1-e^{-2\beta (1-2\tau)/\sqrt{\varepsilon}} ) \Big]^{1/2}
\nonumber\\
&
\leq C\tau^{1/2}\sq^{1/4}N^{-\sigma}\leq C\sq^{1/2}N^{-\sigma}(\ln N)^{1/2},
\end{align}
where we used $\tau= \frac{\sigma\ssq}{\beta}\ln N \leq 1/4$ implying
$1-2\tau\geq 2\tau$ and $e^{-2\beta (1-2\tau)/\sqrt{\varepsilon}}\leq 
e^{-4\beta\tau/\ssq}=N^{-4\sigma}\leq 4^{-4\sigma}\leq 1/16$ for $N\geq 4$ and $\sigma\geq \hat{k}+1\geq 1$.

One has similarly
\begin{align}\label{boundary:3:layer}
\norm{\eta_{u^{\mathrm{b}}_1}}_{\dn}
\leq C\sq^{1/2}N^{-\sigma}(\ln N)^{1/2}.
\end{align}
If $K_{ij}\in\dreg$, then $\eta_{u^{\mathrm{b}}_1} =u^{\mathrm{b}}_1-\Pi u^{\mathrm{b}}_1$ 
and we use the $L^2$-stability property \eqref{L2:stb:L2} 
and the bound \eqref{reg:boundary} to get 	
\begin{align}\label{boundary:4:layer}
\norm{\eta_{u^{\mathrm{b}}_1}}_{\dreg}\leq 
2\norm{u^{\mathrm{b}}_1}_{\dreg}
\leq C\norm{e^{-\beta x/\sqrt{\varepsilon}}}_{\dreg}
\leq C\left(\int_{\tau}^{1-\tau} e^{-2\beta x/\sqrt{\varepsilon}} 
\,\mathrm{d}x\right)^{1/2}
\leq C\sq^{1/4}N^{-\sigma}.
\end{align}
But $\sigma\geq k+1$, so the bounds \eqref{boundary:1:layer}--\eqref{boundary:4:layer} imply that
\begin{align}\label{app:bry1:L2}
\norm{\eta_{u^{\mathrm{b}}_1}}_{\Omega\backslash\dreg}
\leq C\sq^{1/4}(N^{-1}\ln N)^{k+1} 
\quad\text{and}\quad
\norm{\eta_{u^{\mathrm{b}}_1}}_{\dreg}\leq C\sq^{1/4}N^{-\sigma}.
\end{align}

For the corner layer component $u^{\mathrm{c}}_1$, 
if $K_{ij}\in\dws\cup \dwn\cup\des\cup\den$ 
we use the $L^2$-approximation property
\eqref{GR:L^2:app} and the bound \eqref{reg:corner} to get 
\begin{align*}
\norm{\eta_{u^{\mathrm{c}}_1}}_{K_{ij}}
&
\leq C\left[h_{i}^{k+1}\norm{\partial_x^{k+1} u^{\mathrm{c}}_1}_{K_{ij}}
+h_{j}^{k+1}\norm{\partial_y^{k+1} u^{\mathrm{c}}_1}_{K_{ij}}
\right]
\\
&\leq Ch^{k+1}\left[
\int_{K_{ij}} \sq^{-(k+1)}e^{-2\beta (x+y)/\sqrt{\varepsilon}}\,\mathrm{d}x\,\mathrm{d}y
\right]^{1/2} \\
&\leq C(N^{-1}\ln N)^{k+1}\norm{e^{-\beta (x+y)/\sqrt{\varepsilon}}}_{K_{ij}}.
\end{align*}
Hence
\begin{align}\label{corner:1:layer}
\norm{\eta_{u^{\mathrm{c}}_1}}_{\dws\cup \dwn\cup\des\cup\den}
&\leq C(N^{-1}\ln N)^{k+1}\norm{e^{-\beta (x+y)/\sqrt{\varepsilon}}}_{\dws\cup \dwn\cup\des\cup\den} \notag\\
&\leq C\sq^{1/2}(N^{-1}\ln N)^{k+1}.
\end{align}
If $K_{ij}\in\dw$, then $\eta_{u^{\mathrm{c}}_1}
=u^{\mathrm{c}}_1-\Pi_x^{-}u^{\mathrm{c}}_1$; use the $L^2$-stability property \eqref{L2:stb:GR1} 
and the bound \eqref{reg:corner} to obtain
\begin{align*}
\norm{\eta_{u^{\mathrm{c}}_1}}_{K_{ij}}
&\leq C\left[\norm{u^{\mathrm{c}}_1}_{K_{ij}}
+h_i^{1/2}\norm{(u^{\mathrm{c}}_1)^{-}_{i,y}}_{J_{j}}
\right]
\\
&
\leq 
C\left[\norm{e^{-\beta (x+y)/\sqrt{\varepsilon}}}_{K_{ij}}
+h_i^{1/2}\norm{e^{-\beta (x_i+y)/\sqrt{\varepsilon}}}_{J_{j}}\right]
\\
&
\leq Ch_i^{1/2}\norm{e^{-\beta y/\sqrt{\varepsilon}}}_{J_{j}},
\end{align*}
which yields
\begin{align}\label{corner:2:layer}
\norm{\eta_{u^{\mathrm{c}}_1}}_{\dw}
&
\leq C\tau^{1/2}\left(\int_{\tau}^{1-\tau} e^{-2\beta y/\sqrt{\varepsilon}} \,\mathrm{d}y\right)^{1/2}
\leq C\tau^{1/2}\sq^{1/4}N^{-\sigma}\leq C\sq^{1/2}N^{-\sigma}(\ln N)^{1/2}.
\end{align}
In a similar manner, one gets the same bound
for $\norm{\eta_{u^{\mathrm{c}}_1}}_{\ds}$,
$\norm{\eta_{u^{\mathrm{c}}_1}}_{\dn}$ and
$\norm{\eta_{u^{\mathrm{c}}_1}}_{\de}$.
If $K_{ij}\in\dreg$,  then $\eta_{u^{\mathrm{c}}_1}
=u^{\mathrm{c}}_1-\Pi u^{\mathrm{c}}_1$ and we use the $L^2$-stability property \eqref{L2:stb:L2} 
and the bound \eqref{reg:corner} to get	
\begin{align}\label{corner:3:layer}
\norm{\eta_{u^{\mathrm{c}}_1}}_{\dreg}\leq 
2\norm{u^{\mathrm{c}}_1}_{\dreg}
\leq C\norm{e^{-\beta (x+y)/\sqrt{\varepsilon}}}_{\dreg}
\leq C\sq^{1/2}N^{-2\sigma}.
\end{align}
The bounds \eqref{corner:1:layer}, \eqref{corner:2:layer} and \eqref{corner:3:layer} imply that
\begin{align}\label{app:crn1:L2}
\norm{\eta_{u^{\mathrm{c}}_1}}_{\Omega\backslash\dreg}
\leq C\sq^{1/2}(N^{-1}\ln N)^{k+1} 
\quad\text{and}\quad
\norm{\eta_{u^{\mathrm{c}}_1}}_{\dreg}\leq C\sq^{1/2}N^{-2\sigma}.
\end{align}
Combining \eqref{app:smooth:L2}, \eqref{app:bry1:L2} and \eqref{app:crn1:L2},
we have \eqref{bound:eu:L2:2d}.

\emph{Proof of \eqref{bound:eu:L0:2d}:}
The $L^{\infty}$-approximation property \eqref{GR:app} and the bound~\eqref{reg:smooth} give 
\begin{align}\label{smooth:L0}
\norm{\eta_{\bar{u}}}_{L^\infty(K_{ij})}
&\leq 
C \left[h_{i}^{k+1}\norm{\partial_x^{k+1}\bar{u}}_{L^\infty(K_{ij})}
+h_{j}^{k+1}\norm{\partial_y^{k+1}\bar{u}}_{L^\infty(K_{ij})}
\right] \leq CN^{-(k+1)}.
\end{align}

Consider now the boundary layer component $u^{\mathrm{b}}_1$.
If $K_{ij}\in\Omega_{1}^x$, use the $L^{\infty}$-approximation property \eqref{GR:app} 
and the bound \eqref{reg:boundary} to get
\begin{align}
\norm{\eta_{u^{\mathrm{b}}_1}}_{L^\infty(K_{ij})}
&\leq 
C \left[h_i^{k+1}\norm{\partial_x^{k+1}u^{\mathrm{b}}_1}_{L^\infty(K_{ij})}
+h_j^{k+1}\norm{\partial_y^{k+1}u^{\mathrm{b}}_1}_{L^\infty(K_{ij})}
\right] 
\nonumber\\
&\leq 
C\left[h^{k+1}\norm{\sq^{-(k+1)/2}e^{-\beta x/\sqrt{\varepsilon}}}_{L^\infty(K_{ij})}
+H^{k+1}\norm{e^{-\beta x/\sqrt{\varepsilon}}}_{L^\infty(K_{ij})}
\right]
\nonumber\\
\label{bry:L0}
&\leq C(N^{-1}\ln N)^{k+1}.
\end{align}
If $K_{ij}\in\Omega\backslash\Omega_{1}^x$, 
the $L^{\infty}$-stability \eqref{L0:stb} 
and the bound \eqref{reg:boundary} yield
\begin{align}\label{boundary:L0:layer}
\norm{\eta_{u^{\mathrm{b}}_1}}_{L^{\infty}(K_{ij})}
\leq 
C\norm{u^{\mathrm{b}}_1}_{L^{\infty}(K_{ij})}
\leq C\norm{e^{-\beta x/\sqrt{\varepsilon}}}_{L^{\infty}(K_{ij})}
\leq Ce^{-\beta \tau/\sqrt{\varepsilon}}
\leq C N^{-\sigma}.
\end{align}

Next, consider the corner layer component $u^{\mathrm{c}}_1$. 
If $K_{ij}\in\dws$, we have
$\eta_{u^{\mathrm{c}}_1}
=u^{\mathrm{c}}_1-\Pi_{xy}^{-}u^{\mathrm{c}}_1$.
The $L^{\infty}$-approximation property \eqref{GR:app} and \eqref{reg:corner}
give 
\begin{align}
\norm{\eta_{u^{\mathrm{c}}_1}}_{L^\infty(K_{ij})}
&\leq 
C \left[h^{k+1}\norm{\partial_x^{k+1}u^{\mathrm{c}}_1}_{L^\infty(K_{ij})}
+h^{k+1}\norm{\partial_y^{k+1}u^{\mathrm{c}}_1}_{L^\infty(K_{ij})}
\right] 
\nonumber\\
\label{crn:L0:bry:layer}
&\leq C(N^{-1}\ln N)^{k+1}\norm{e^{-\beta (x+y)/\sqrt{\varepsilon}}}_{L^{\infty}(K_{ij})}
\leq C(N^{-1}\ln N)^{k+1}.
\end{align}
For $K_{ij}\in\Omega\backslash\dws$, 
we use the $L^{\infty}$-stability property \eqref{L0:stb}
and the bound \eqref{reg:corner} to get
\begin{align}\label{crn:L0:layer}
\norm{\eta_{u^{\mathrm{c}}_1}}_{L^{\infty}(K_{ij})}
\leq C
\norm{u^{\mathrm{c}}_1}_{L^{\infty}(K_{ij})}
\leq C\norm{e^{-\beta (x+y)/\sqrt{\varepsilon}}}_{L^{\infty}(K_{ij})}
\leq Ce^{-\beta \tau/\sqrt{\varepsilon}}
\leq C N^{-\sigma}.
\end{align}

The $L^{\infty}$-bounds for $\eta_{u^{\mathrm{b}}_i}$ and 
$\eta_{u^{\mathrm{c}}_i}$ ($i=2,3,4$) are obtained in a similar manner.
Then \eqref{bound:eu:L0:2d} follows from \eqref{smooth:L0}--\eqref{crn:L0:layer} and a triangle inequality.

\emph{Proof of \eqref{bound:ep:L2:2d}--\eqref{bound:eq:L0:2d}:}
These inequalities can be derived using analogous arguments.

\section{Superapproximation properties of projectors}\label{sec:superapproximation}

In Section~\ref{sec:superapproximation}  we present various superapproximation properties for the three types of vertex-edge-element approximations
appearing in~\eqref{prj:u:2d}.

%
%
\subsection{Superapproximation properties of element projector $\Pi$}\label{sec:superapproxPi}
First, consider the $L^2$-projector $\Pi$. Recall the definition of $\hat k$ in~\eqref{khat}.
\begin{lemma}\label{sup:L2:element}
Let $K_{ij}, K_{i+1,j}\in\Omega_N$ with $h_i=h_{i+1}$.
Let $z\in C^{k+2}(\bar  K_{ij}\cup \bar K_{i+1,j})$.
Then there exists a constant $C>0$ such that for all $\vphu\in\spc$ one has
\begin{align*}
\left|\int_{J_j}\average{z-\Pi z}_{i,y}\jump{\vphu}_{i,y} \,\mathrm{d}y\right|
	&\leq C\left[h_i^{\hat{k}}\norm{\partial_x^{\hat{k}+1}z}_{K_{ij}\cup K_{i+1,j}} \right. \\
	&\qquad\qquad \left. +h_i^{k+1}\norm{\partial_x^{k+2}z}_{K_{ij}\cup K_{i+1,j}}\right]
		\norm{\vphu}_{K_{ij}\cup K_{i+1,j}}.
\end{align*}
Analogously,  if  $z\in C^{k+2}(\bar K_{ij}\cup \bar K_{i,j+1})$ with $h_j=h_{j+1}$, then
\begin{align*}
\left|\int_{I_i}\average{z-\Pi z}_{x,j}\jump{\vphu}_{x,j} \,\mathrm{d}x\right|
	&\leq C\left[h_j^{\hat{k}}\norm{\partial_y^{\hat{k}+1}z}_{K_{ij}\cup K_{i,j+1}} \right.  \\
	&\qquad\qquad \left. +h_j^{k+1}\norm{\partial_y^{k+2}z}_{K_{ij}\cup K_{i,j+1}} \right]
		\norm{\vphu}_{K_{ij}\cup K_{i,j+1}}.
\end{align*}
\end{lemma}

\begin{proof}
We prove only the first inequality as the second is derived similarly.
For $m=0,1,\dots$, let $\ell_m(x)$ be the rescaled Legendre polynomial of degree $m$ 
defined via Rodrigues' formula for the standard Legendre polynomial $\hat{\ell}_m(s)$
on the interval $\hat{I}:=[-1,1]$, namely,
\begin{align}\label{lengdre}
\ell_m(x)=\hat{\ell}_m\left(\frac{2x-x_i-x_{i-1}}{h_i}\right)
=\hat{\ell}_m(s)=\frac{1}{2^m m!}\partial_s^m (s^2-1)^m.
\end{align}
Then one has
\begin{equation}\label{Legd:norm}
\norm{\hat{\ell}_m}_{\hat{I}}=\sqrt{\frac{2}{2m+1}}
\quad\text{ and  }\quad
\norm{\ell_m}_{I_i}=\sqrt{\frac{h_i}{2m+1}}.
\end{equation}
Since $z\in C^{k+2}(\bar  K_{ij})$, by \cite[Theorem~3.5]{WangJFAA23} 
an expansion in rescaled Legendre polynomials for each $y\in J_j$
yields $z(x,y) = \sum_{m=0}^\infty  \alpha_{y;m}^{ij}(y) \ell_m(x)$,
where for $m=0,1,\dots$ one defines
$\alpha_{y;m}^{ij}(y)  := \frac{1}{\norm{\ell_m}^2_{I_i}}\int_{I_i}z(x,y)\ell_m(x)\,\mathrm{d}x$.
But for each $x\in I_i$ one also has the similar expansion
 $z(x,y) = \sum_{n=0}^\infty  \alpha_{x;n}^{ij}(x) \ell_n(y)$,
where for each $n$ one sets
$\alpha_{x;n}^{ij}(x)  := \frac{1}{\norm{\ell_n}^2_{J_j}}\int_{J_j}z(x,y)\ell_n(y)\,\mathrm{d}y$.
In these expansions,
$|\alpha_{y;m}^{ij}(y)| = O(m^{-(k+5/2)})$ and  $|\alpha_{x;n}^{ij}(x)| = O(n^{-(k+5/2)})$ for $m,n\geq k+3$
follow from $z\in C^{k+2}(\bar  K_{ij})$ and~\cite[Theorem~3.1]{WangJFAA23}.
Hence for each $m$ one gets
\begin{align*}
\alpha_{y;m}^{ij}(y)  &= \frac{1}{\norm{\ell_m}^2_{I_i}}\int_{I_i}z(x,y)\ell_m(x)\,\mathrm{d}x \\
	&=  \frac{1}{\norm{\ell_m}^2_{I_i}}\int_{I_i} \sum_{n=0}^\infty 
		\alpha_{x;n}^{ij}(x) \ell_n(y) 			\ell_m(x)\,\mathrm{d}x \\
	&=  \frac{1}{\norm{\ell_m}^2_{I_i}} \sum_{n=0}^\infty 
	\left(\int_{I_i}\alpha_{x;n}^{ij}(x) \ell_m(x)\,\mathrm{d}x\right)
		\ell_n(y),		
\end{align*}
as the uniform convergence of the series allows us to interchange the integration and the sum.
Recalling the definition of $\alpha_{x;n}^{ij}(x)$, we have
\begin{align*}
\alpha_{y;m}^{ij}(y)=\frac{1}{\norm{\ell_m}^2_{I_i}} \sum_{n=0}^\infty
\frac{1}{\norm{\ell_n}^2_{J_j}}
\left(\int_{K_{ij}}z(x,y) \ell_m(x)\ell_n(y)\,\mathrm{d}x\,\mathrm{d}y\right)\ell_n(y).
\end{align*}
Substituting this identity into $z(x,y) = \sum_{m=0}^\infty  \alpha_{y;m}^{ij}(y) \ell_m(x)$, we obtain finally
\begin{equation}\label{expan:coef}
z(x,y) =\sum_{m=0}^{\infty}\sum_{n=0}^{\infty}
\alpha^{ij}_{mn}\ell_m(x)\ell_n(y) \ \text{ for }(x,y)\in K_{ij},
\end{equation} 
where the coefficients are 
\begin{equation}\label{legendre:coef}
\alpha^{ij}_{mn} :=\frac{1}{\norm{\ell_m}^2_{I_i}\norm{\ell_n}^2_{J_j}}
\int_{K_{ij}}z(x,y)\ell_m(x)\ell_n(y)\,\mathrm{d}x\,\mathrm{d}y
\ \text{ for }m,n=0,1,\dots
\end{equation}
Now the definition of~$\Pi$ implies that
\[
\Pi z(x,y)|_{K_{ij}}=\sum_{m=0}^{k}\sum_{n=0}^{k}
\alpha^{ij}_{mn}\ell_m(x)\ell_n(y).
\]
Hence 
\begin{align*}
(z-\Pi z)(x,y)|_{K_{ij}}
=\sum_{n=k+1}^{\infty}\sum_{m=0}^{k}
\alpha^{ij}_{mn}\ell_m(x)\ell_n(y)
+\sum_{n=0}^{\infty}\sum_{m=k+1}^{\infty}
\alpha^{ij}_{mn}\ell_m(x)\ell_n(y).
\end{align*}
Since $\hat{\ell}_m(\pm 1)=(\pm 1)^m$,  the average
\begin{align}
\average{z-\Pi z}_{i,y}
&=\frac12\Big[(z-\Pi z)^-_{i,y}+(z-\Pi z)^+_{i,y}\Big]
\notag\\
&\hspace{-2cm}=\frac12 \sum_{n=k+1}^{\infty}\sum_{m=0}^{k}
\big[\alpha^{ij}_{mn}+\alpha^{i+1,j}_{mn}(-1)^m\big]\ell_n(y)
+\frac12 \sum_{n=0}^{\infty}
\big[\alpha^{ij}_{k+1,n}+\alpha^{i+1,j}_{k+1,n}(-1)^{k+1}\big]\ell_n(y)
\notag\\
&\quad
+\frac12 \sum_{n=0}^{\infty}\sum_{m=k+2}^{\infty}
\big[\alpha^{ij}_{mn}+\alpha^{i+1,j}_{mn}(-1)^m\big]\ell_n(y)
\notag\\
&
:=A_1(y)+A_2(y)+A_3(y).  \label{A1A2A3}
\end{align}
The pairwise orthogonality property $\int_{J_j} \ell_n(y) \ell_r(y)\,dy = 0$ for $n\ne r$ implies that 
\begin{equation}\label{A1}
\int_{J_j} A_1(y) \jump{v}_{i,y}\,\mathrm{d}y=0 
\ \text{ for  each }\ \jump{v}_{i,y}\in \mathcal{P}^k(J_j).
\end{equation}

We turn next to $\int_{J_j} A_2(y) \jump{v}_{i,y}\,\,\mathrm{d}y$.
Let $\hat{K}=(-1,1)^2$ be  the reference element and set
$x_{i-1/2}=(x_i+x_{i-1})/2$, $y_{j-1/2}=(y_j+y_{j-1})/2$.
Consider the affine transformation
\[
z(x,y)|_{K_{ij}}=z\left(x_{i-1/2}+\frac12 h_i s,y_{j-1/2}+\frac12 h_j t\right)
	:=\hat{z}(s,t)|_{\hat{K}}.
\]
By \eqref{legendre:coef}, \eqref{lengdre}, a scaling argument,
integration by parts, Cauchy-Schwarz's inequality and \eqref{Legd:norm},
one has
\begin{align}
\left|\alpha^{ij}_{k+1,n}\right|
&=\left|
\frac{1}{\norm{\ell_{k+1}}^2_{I_i}\norm{\ell_n}^2_{J_j}}
\int_{K_{ij}}z(x,y)\ell_{k+1}(x)\ell_n(y)\,\mathrm{d}x\,\mathrm{d}y
\right|
\nonumber\\
&= 
\frac{1}{\norm{\hat{\ell}_{k+1}}^2_{\hat{I}}\norm{\hat{\ell}_n}^2_{\hat{I}}}
\cdot\frac1{2^{k+1}(k+1)!} 
\left|\int_{\hat{K}}\hat{z}(s,t)\partial_s^{k+1} (s^2-1)^{k+1}
\hat{\ell}_n(t)
\,\mathrm{d}s\,\mathrm{d}t\right|
\nonumber
\\
\label{IBP}
&= 
\frac{2k+3}{2^{k+2}(k+1)!} 
\cdot\frac{1}{\norm{\hat{\ell}_n}^2_{\hat{I}}} 
\left|\int_{\hat{K}}\partial_s^{k+1}\hat{z}(s,t) (s^2-1)^{k+1}
\hat{\ell}_n(t)
\,\mathrm{d}s\,\mathrm{d}t\right|
\\
&\leq \frac{1}{\norm{\hat{\ell}_n}_{\hat{I}}}
\left[\int_{\hat{K}}\left(\partial_s^{k+1}\hat{z}\right)^2\,\mathrm{d}s\,\mathrm{d}t\right]^{1/2}
\leq C\sqrt{2n+1}
\left[\frac{1}{|K_{ij}|}
\int_{K_{ij}}\left(h_i^{k+1}\partial_x^{k+1}z\right)^2\,\mathrm{d}x\,\mathrm{d}y
\right]^{1/2}
\nonumber\\
&
\leq C\sqrt{2n+1}
\sqrt{\frac{h_i}{h_j}}\, h_i^k\norm{\partial_x^{k+1}z}_{K_{ij}}. 
\nonumber
\end{align}
Here and in what follows $C>0$ is independent of $n,k,h_i,h_j$ and $z$.
Similarly,
\[
|\alpha^{i+1,j}_{k+1,n}|\leq C
\sqrt{2n+1}
\sqrt{\frac{h_{i+1}}{h_j}}\,h_{i+1}^k\norm{\partial_x^{k+1}z}_{K_{i+1,j}}.
\]
Our argument now depends on the parity of the nonnegative integer~$k$.
If $k$ is odd, then $\alpha^{ij}_{k+1,n}+\alpha^{i+1,j}_{k+1,n}(-1)^{k+1}
=\alpha^{ij}_{k+1,n}+\alpha^{i+1,j}_{k+1,n}$.
Thus, since $h_i=h_{i+1}$, the pairwise orthogonality of the rescaled Legendre polynomials, \eqref{Legd:norm},
Cauchy-Schwarz and inverse inequalities yield 
\begin{align}
\left|\int_{J_j} A_2(y) \jump{\vphu}_{i,y}\,\mathrm{d}y\right|
&=\left|\frac12\int_{J_j} \sum_{n=0}^{k} 
\big(\alpha^{ij}_{k+1,n}+\alpha^{i+1,j}_{k+1,n}\big)\ell_n(y)
\jump{\vphu}_{i,y}\,\mathrm{d}y\right|
\nonumber\\
&\leq C\left(\int_{J_j}\sum_{n=0}^{k} 
\big[|\alpha^{ij}_{k+1,n}|^2+|\alpha^{i+1,j}_{k+1,n}|^2\big]\ell^2_n(y)
\,\mathrm{d}y\right)^{1/2}
\left(\int_{J_j}\jump{\vphu}^2_{i,y}\,\mathrm{d}y\right)^{1/2}
\nonumber\\
&
\leq C\max_{0\leq n\leq k}\Big(|\alpha^{ij}_{k+1,n}|
+|\alpha^{i+1,j}_{k+1,n}|\Big)\norm{\ell_n}_{J_j}
h_i^{-1/2}\norm{\vphu}_{K_{ij}\cup K_{i+1,j}}
\nonumber\\
&
\leq C\sqrt{2n+1}\sqrt{\frac{h_i}{h_j}}\,h_i^k\norm{\partial_x^{k+1}z}_{K_{ij}\cup K_{i+1,j}}\sqrt{\frac{h_j}{2n+1}}h_i^{-1/2}\norm{\vphu}_{K_{ij}\cup K_{i+1,j}}
\nonumber\\
\label{A2:odd}
&
\leq Ch_i^k\norm{\partial_x^{k+1}z}_{K_{ij}\cup K_{i+1,j}}\norm{\vphu}_{K_{ij}\cup K_{i+1,j}}
\ \text{ for all } \vphu\in \spc.
\end{align}
On the other hand, if $k$ is even, then $\alpha^{ij}_{k+1,n}+\alpha^{i+1,j}_{k+1,n}(-1)^{k+1}
=\alpha^{ij}_{k+1,n}-\alpha^{i+1,j}_{k+1,n}$
and we can get a better bound on $\int_{J_j} A_2(y) \jump{\vphu}_{i,y}\,\mathrm{d}y$.
Since $h_i=h_{i+1}$, from \eqref{IBP} and some manipulations we obtain 
\begin{align*}
&\left|\alpha^{i+1,j}_{k+1,n}-\alpha^{ij}_{k+1,n}\right|
\\
&=\frac{2k+3}{2^{k+2}(k+1)!} 
\cdot\frac{1}{\norm{\hat{\ell}_n}^2_{\hat{I}}}
\left|
\int_{\hat{K}}
\partial_s^{k+1} \left[
z\Big(x_{i+1/2}+\frac{1}{2}h_{i+1}s,y_{j-1/2}+\frac{1}{2}h_j t \Big)  \right. \right.\\
&\qquad \left.  \left. -z\Big(x_{i-1/2}+\frac{1}{2}h_{i}s,y_{j-1/2}+\frac{1}{2}h_jt \Big)
\right] (s^2-1)^{k+1} \hat{\ell}_n(t)\,\mathrm{d}s\,\mathrm{d}t\right|
\\
&=\frac{2k+3}{2^{k+2}(k+1)!} 
\cdot\frac{1}{\norm{\hat{\ell}_n}^2_{\hat{I}}}
\Big(\frac12h_i\Big)^{k+1}
\left|
\int_{\hat{K}}\left[\int_{x_{i-1/2}}^{x_{i+1/2}}
\partial_x^{k+2}z\left(x+\frac{1}{2}h_is,y_{j-1/2}+\frac{1}{2}h_j t\right)\mathrm{d}x\right] \right.  \\
&\hspace{1cm}\left. (s^2-1)^{k+1}\hat{\ell}_n(t)\mathrm{d}s\,\mathrm{d}t
\right|
\\
&\leq  Ch_i^{k+1}
\frac{1}{\norm{\hat{\ell}_n}^2_{\hat{I}}}
\int_{-1}^1\int_{x_{i-1}}^{x_{i+1}}
\left|\partial_x^{k+2}z\left(x,y_{j-1/2}+\frac{1}{2}h_j t\right)\right|
|\hat{\ell}_n(t)|
\,\mathrm{d}x\,\mathrm{d}t
\\
&\leq  Ch_i^{k+1}
\frac{1}{\norm{\hat{\ell}_n}^2_{\hat{I}}}
\left[
\int_{-1}^1\int_{x_{i-1}}^{x_{i+1}}
\left|\partial_x^{k+2}z\left(x,y_{j-1/2}+\frac{1}{2}h_j t\right)\right|^2
\,\mathrm{d}x\,\mathrm{d}t\right]^{1/2}
\cdot\left(h_i^{1/2}\norm{\hat{\ell}_n}_{\hat{I}}\right)
\\
&= C\sqrt{2n+1}\cdot h_i^{k+3/2}h_j^{-1/2}
\left[
\int_{y_{j-1}}^{y_j}\int_{x_{i-1}}^{x_{i+1}}
\left|\partial_x^{k+2}z\left(x,y_{j-1/2}+\frac{1}{2}h_j t\right)\right|^2
\,\mathrm{d}x\,\mathrm{d}y\right]^{1/2}
\\
&=C\sqrt{2n+1}
\sqrt{\frac{h_i}{h_j}}\,h_i^{k+1}\norm{\partial_x^{k+2}z}_{K_{ij}\cup K_{i+1,j}}.
\end{align*} 
Hence Cauchy-Schwarz and inverse inequalities and \eqref{Legd:norm} yield
\begin{align}
\left|\int_{J_j} A_2(y) \jump{\vphu}_{i,y}\,\mathrm{d}y\right|
&=\left|\frac12\int_{J_j} \sum_{n=0}^{k} 
\big(\alpha^{ij}_{k+1,n}-\alpha^{i+1,j}_{k+1,n}\big)\ell_n(y)
\jump{\vphu}_{i,y}\,\mathrm{d}y\right|
\nonumber\\
&\leq C\left(\int_{J_j}\sum_{n=0}^{k} 
|\alpha^{ij}_{k+1,n}-\alpha^{i+1,j}_{k+1,n}|^2\ell^2_n(y)
\,\mathrm{d}y\right)^{1/2}
\left(\int_{J_j}\jump{\vphu}^2_{i,y}\,\mathrm{d}y\right)^{1/2}
\nonumber\\
&
\leq C\max_{0\leq n\leq k}\Big(|\alpha^{ij}_{k+1,n}
-\alpha^{i+1,j}_{k+1,n}|\Big)\norm{\ell_n}_{J_j}h_i^{-1/2}\norm{\vphu}_{K_{ij}\cup K_{i+1,j}}
\nonumber\\
&
\leq C\sqrt{2n+1}\sqrt{\frac{h_i}{h_j}}h_i^{k+1}\norm{\partial_x^{k+2}z}_{K_{ij}\cup K_{i+1,j}}
\sqrt{\frac{h_j}{2n+1}}
h_i^{-1/2}\norm{\vphu}_{K_{ij}\cup K_{i+1,j}}
\nonumber\\
\label{A2:even}
&
\leq Ch_i^{k+1}\norm{\partial_x^{k+2}z}_{K_{ij}\cup K_{i+1,j}}
\norm{\vphu}_{K_{ij}\cup K_{i+1,j}}
\end{align}
for any $\vphu\in \spc$.
From \eqref{A2:odd} and \eqref{A2:even},
for all nonnegative integers~$k$ one has
\begin{align}\label{A2}
\left|\int_{J_j} A_2(y) \jump{\vphu}_{i,y}\,\mathrm{d}y\right|
\leq Ch_i^{\hat{k}}\norm{\partial_x^{\hat{k}+1}z}_{K_{ij}\cup K_{i+1,j}}
\norm{\vphu}_{K_{ij}\cup K_{i+1,j}}.
\end{align}

It remains to bound $\int_{J_j} A_3(y) \jump{\vphu}_{i,y}\,\mathrm{d}y$. 
Substituting \eqref{expan:coef} into the definition of $\alpha_{x;n}^{ij}(x)$ 
then appealing to uniform convergence to interchange integration and summation,
the pairwise orthogonality of the rescaled Legendre polynomials yields
\begin{align*}
\alpha_{x;n}^{ij}(x) &= \frac{1}{\norm{\ell_n}^2_{J_j}} \int_{J_{j}}
	\left(\sum_{m=0}^{\infty}\sum_{r=0}^{\infty}\alpha^{ij}_{mr}\ell_m(x)\ell_r(y)\right)\ell_n(y)\,\mathrm{d}y \\
&=\sum_{m=0}^{\infty} \alpha^{ij}_{mn}\ell_m(x)\ \text{ for } n=0,1,2,\dots
\end{align*}
Thus, writing $\pi^{k+1}$ for the $L^2$-projector into $\mathcal{P}^{k+1}(I_i)$,
one has $\pi^{k+1} \alpha_{x;n}^{ij}(x) = \sum_{m=0}^{k+1} \alpha^{ij}_{mn}\ell_m(x)$
and the approximation error bound
\cite[p.32, Lemma 1.59]{Pietro2012}
\[
\norm{\alpha_{x;n}^{ij} - \pi^{k+1} \alpha_{x;n}^{ij}}_{L^{2}(\partial I_i)}\leq
Ch_i^{k+3/2}\norm{\partial_x^{k+2} \alpha_{x;n}^{ij}}_{I_i},
\]
where $\norm{z}_{L^{2}(\partial I_i)} := [z^2(x_i^{-})+z^2(x_{i-1}^{+})]^{1/2}$.
Using this inequality, $\ell_m(x_i)=\hat{\ell}_m(1)=1$, the definition of $\alpha_{x;m}^{ij}(x)$,
a Cauchy-Schwarz inequality and \eqref{Legd:norm}, we obtain
\begin{align}
\left|\sum_{m=k+2}^{\infty} \alpha^{ij}_{mn}\right|
&= \left|\sum_{m=k+2}^{\infty} \alpha^{ij}_{mn}\ell_m(x_i)\right| 
= \left|(\alpha_{x;n}^{ij} - \pi^{k+1} \alpha_{x;n}^{ij})(x_i)\right| \notag\\
&
\leq \norm{\alpha_{x;n}^{ij} - \pi^{k+1} \alpha_{x;n}^{ij}}_{L^{2}(\partial I_i)} 
\leq Ch_i^{k+3/2}\norm{\partial_x^{k+2}\alpha_{x;n}^{ij} }_{I_i} 
\notag\\
&
=Ch_i^{k+3/2}\left(\int_{I_i}[\partial_x^{k+2}\alpha_{x;n}^{ij}(x)]^2 \,\mathrm{d}x\right)^{1/2}
\nonumber\\
&=Ch_i^{k+3/2}\left(\int_{I_i}\left[
	\frac{1}{\norm{\ell_n}^2_{J_j}}\int_{J_{j}}\partial_x^{k+2}z(x,y)\ell_n(y)\,\mathrm{d}y
	\right]^2\,\mathrm{d}x\right)^{1/2}\nonumber\\ \label{sum:1}
&\leq Ch_i^{k+3/2}  \left(\int_{I_i}\frac{1}{\norm{\ell_n}^2_{J_j}}
	\int_{J_{j}}[\partial_x^{k+2}z]^2\,\mathrm{d}y\,\mathrm{d}x\right)^{1/2}
\notag\\
&
= C\sqrt{2n+1}
\sqrt{\frac{h_i}{h_j}}h_i^{k+1}\norm{\partial_x^{k+2}z}_{K_{ij}}.
\end{align}
Similarly,
\begin{align}\label{sum:2}
\left|\sum_{m=k+2}^{\infty} \alpha^{i+1,j}_{mn}(-1)^m\right|
&\leq C\sqrt{2n+1}
\sqrt{\frac{h_{i+1}}{h_j}}\,h_{i+1}^{k+1}\norm{\partial_x^{k+2}z}_{K_{i+1,j}}  
\notag \\
&= C\sqrt{2n+1}
\sqrt{\frac{h_{i}}{h_j}}\,h_{i}^{k+1}\norm{\partial_x^{k+2}z}_{K_{i+1,j}}. 
\end{align}
Now the pairwise orthogonality of the rescaled Legendre polynomials,
\eqref{Legd:norm} 
and Cauchy-Schwarz and inverse inequalities yield
\begin{align}
&\left|\int_{J_j} A_3(y) \jump{\vphu}_{i,y}\,\mathrm{d}y\right|
\nonumber\\
&=\frac12\left|\int_{J_j}
\sum_{n=0}^{k}\sum_{m=k+2}^{\infty}
\big[\alpha^{ij}_{mn}+\alpha^{i+1,j}_{mn}(-1)^m\big]\ell_n(y)
\jump{\vphu}_{i,y}\,\mathrm{d}y\right|
\nonumber\\
&\leq C\int_{J_j}\sum_{n=0}^{k}
\left[\left|\sum_{m=k+2}^{\infty}\alpha^{ij}_{m,n}\right|
+\left|\sum_{m=k+2}^{\infty}\alpha^{i+1,j}_{m,n}(-1)^m\right|
\right]|\ell_n(y)| \, \left|\jump{\vphu}_{i,y}\right|\,\mathrm{d}y
\nonumber\\
&\leq C\left(\int_{J_j}\sum_{n=0}^{k}
\left[\left|\sum_{m=k+2}^{\infty}\alpha^{ij}_{m,n}\right|^2
+\left|\sum_{m=k+2}^{\infty}\alpha^{i+1,j}_{m,n}(-1)^m\right|^2\right]\ell^2_n(y)
\,\mathrm{d}y\right)^{1/2}
\left(\int_{J_j}\jump{\vphu}^2_{i,y}\,\mathrm{d}y\right)^{1/2}
\nonumber\\
&\leq C\sqrt{2n+1}
\sqrt{\frac{h_i}{h_j}}h_i^{k+1}\norm{\partial_x^{k+2}z}_{K_{ij}\cup K_{i+1,j}}
\sqrt{\frac{h_j}{2n+1}}
h_i^{-1/2}\norm{\vphu}_{K_{ij}\cup K_{i+1,j}}
\nonumber\\
\label{A3}
&
\leq Ch_i^{k+1}\norm{\partial_x^{k+2}z}_{K_{ij}\cup K_{i+1,j}}
\norm{\vphu}_{K_{ij}\cup K_{i+1,j}},
\end{align}
where we used \eqref{sum:1}, \eqref{sum:2} and an inverse inequality.

Combine \eqref{A1A2A3},  \eqref{A1}, \eqref{A2} and \eqref{A3} to complete the proof.
\end{proof}

Lemma~\ref{sup:L2:element}
enables us to prove the following bounds 
for the projector $\Pi$ in the coarse-mesh domain~$\dreg$
and two extensions of~$\dreg$.

\begin{lemma}\label{lemma:sup:L2:2d}
Let $\sigma\geq \hat{k}+1$.
Assume that Lemma \ref{lemma:prop:2d} is valid when $m=k+1$.
Then there exists a constant $C>0$ such that for any $\vphu\in\spc$ one has
\begin{subequations}\label{sup:L2:2d}
\begin{align}
\label{sup:L2:2d:u:x}
\left|\sum_{j=N/4+1}^{3N/4}\sum_{i=N/4+1}^{3N/4-1}\dual{\average{u-\Pi u}_{i,y}}{\jump{\vphu}_{i,y}}_{J_j}
\right|&\leq CN^{-\hat{k}}\norm{\vphu},
\\
\label{sup:L2:2d:u:y}
\left|\sum_{j=N/4+1}^{3N/4}\sum_{i=N/4+1}^{3N/4-1}\dual{\average{u-\Pi u}_{x,j}}{\jump{\vphu}_{x,j}}_{I_i}
\right|&\leq CN^{-\hat{k}}\norm{\vphu},
\\
\label{sup:L2:2d:p}
\left|\sum_{j=1}^N\sum_{i=N/4+1}^{3N/4-1}\dual{\average{p-\Pi p}_{i,y}}{\jump{\vphu}_{i,y}}_{J_j}
\right|&\leq C\sq^{1/2}N^{-\hat{k}}\norm{\vphu},
\\
\label{sup:L2:2d:q}
\left|\sum_{i=1}^N\sum_{j=N/4+1}^{3N/4-1}\dual{\average{q-\Pi q}_{x,j}}{\jump{\vphu}_{x,j}}_{I_i}
\right|&\leq C\sq^{1/2}N^{-\hat{k}}\norm{\vphu}.
\end{align}
\end{subequations}
\end{lemma}

\begin{proof}
All four inequalities are proved similarly so we present a detailed proof 
only for~\eqref{sup:L2:2d:u:x}.
For the smooth component $\bar{u}$ of~$u$, Lemma~\ref{sup:L2:element}, 
a Cauchy-Schwarz inequality and Lemma~\ref{lemma:prop:2d} give us 
\begin{align}
&\left|\sum_{j=N/4+1}^{3N/4}\sum_{i=N/4+1}^{3N/4-1}
\dual{\average{\bar{u}-\Pi \bar{u}}_{i,y}}{\jump{\vphu}_{i,y}}_{J_j}
\right|\nonumber\\
&\quad\leq 
C\sum_{j=N/4+1}^{3N/4}\sum_{i=N/4+1}^{3N/4-1} 
\left[
H^{\hat{k}}\norm{\partial_x^{\hat{k}+1} \bar{u}}_{K_{ij}\cup K_{i+1,j}}
+H^{k+1}\norm{\partial_x^{k+2}\bar{u}}_{K_{ij}\cup K_{i+1,j}}
\right]	
\norm{\vphu}_{K_{ij}\cup K_{i+1,j}}
\nonumber\\
\label{avg:regular}
&\quad\leq CH^{\hat{k}}\Big(\norm{\partial_x^{\hat{k}+1} \bar{u}}_{\dreg}
+\norm{\partial_x^{k+2}\bar{u}}_{\dreg}
\Big)\norm{\vphu}_{\dreg}
\leq CN^{-\hat{k}}\norm{\vphu}_{\dreg}.
\end{align}

For the layer component $u_{\sq}$ of Lemma~\ref{lemma:prop:2d}, we use a Cauchy-Schwarz inequality,
the $L^{\infty}$-stability property~\eqref{L0:stb}, 
an inverse inequality,  and $\sigma\geq \hat{k}+1$ to get 
\begin{align}
&\left|\sum_{j=N/4+1}^{3N/4}\sum_{i=N/4+1}^{3N/4-1}
\dual{\average{u_{\sq}-\Pi u_{\sq}}_{i,y}}{\jump{\vphu}_{i,y}}_{J_j}
\right|
\nonumber\\
&\quad\leq 
\sum_{j=N/4+1}^{3N/4} \sum_{i=N/4+1}^{3N/4-1}
\norm{\average{u_{\sq}-\Pi u_{\sq}}_{i,y}}_{J_j}
\norm{\jump{\vphu}_{i,y}}_{J_j}
\nonumber\\
&\quad\leq 
C\sum_{j=N/4+1}^{3N/4}\sum_{i=N/4+1}^{3N/4-1}
H^{1/2}\norm{u_{\sq}-\Pi u_{\sq}}_{L^{\infty}(K_{ij}\cup K_{i+1,j})}
H^{-1/2}\norm{\vphu}_{K_{ij}\cup K_{i+1,j}}
\nonumber\\
&\quad\leq 
C\sum_{j=N/4+1}^{3N/4}\sum_{i=N/4+1}^{3N/4-1}
\norm{u_{\sq}}_{L^{\infty}(K_{ij}\cup K_{i+1,j})}
\norm{\vphu}_{K_{ij}\cup K_{i+1,j}}
\nonumber\\
\label{avg:layer}
&\quad\leq 
C\sum_{j=N/4+1}^{3N/4}\sum_{i=N/4+1}^{3N/4-1}
N^{-\sigma}\norm{\vphu}_{K_{ij}\cup K_{i+1,j}}
\leq CN^{-\sigma+1}\norm{\vphu}_{\dreg}
\leq CN^{-\hat{k}}\norm{\vphu}_{\dreg}.
\end{align}

Then \eqref{sup:L2:2d:u:x} follows from \eqref{avg:regular}, \eqref{avg:layer} and a triangle inequality.
\end{proof}

\subsection{Superapproximation properties of edge-element projectors}
\label{sec:superapprox_edge_element}

We now turn our attention to the edge-element approximation operators
$\Pi_x^{+}, \Pi_x^{-}, \Pi_y^{+}$ and $\Pi_y^{-}$.

\begin{lemma}\label{sup:edge:element}
Let $K_{ij}, K_{i+1,j}\in\Omega_N$ with $h_i=h_{i+1}$.
Let $z\in C^{k+1}(\bar  K_{ij}\cup \bar K_{i+1,j})$.
Then there exists a constant $C>0$ such that for all $\vphu\in\spc$ one has
\begin{align*} 
\left|\int_{J_j}\average{z-\Pi_y^{\pm} z}_{i,y}\jump{\vphu}_{i,y} \,\mathrm{d}y\right|
	&\leq Ch_i^{-1}\left[h_i^{k+1}\norm{\partial_x^{k+1}z}_{K_{ij}\cup K_{i+1,j}} \right.  \\
	&\qquad \left.\qquad+h_j^{k+1}\norm{\partial_y^{k+1}z}_{K_{ij}\cup K_{i+1,j}}
	\right]
	\norm{\vphu}_{K_{ij}\cup K_{i+1,j}}.
\end{align*}
Analogously,  if  $z\in C^{k+1}(\bar K_{ij}\cup \bar K_{i,j+1})$ with $h_j=h_{j+1}$, then
\begin{align*} 
	\left|\int_{I_i}\average{z-\Pi_x^{\pm} z}_{x,j}\jump{\vphu}_{x,j} \,\mathrm{d}x\right|
	&\leq Ch_j^{-1}\left[h_i^{k+1}\norm{\partial_x^{k+1}z}_{K_{ij}\cup K_{i,j+1}} \right. \\
	&\qquad \left. \qquad+h_j^{k+1}\norm{\partial_y^{k+1}z}_{K_{ij}\cup K_{i,j+1}}
	\right]
	\norm{\vphu}_{K_{ij}\cup K_{i,j+1}}.
\end{align*}
\end{lemma}

\begin{proof}
We present the proof of the first inequality for the projector $\Pi_y^{-}$;
the remaining cases are handled in a similar way.
Recall the Legendre expansion \eqref{expan:coef}: 
\begin{align*}
z(x,y)|_{K_{ij}}=\sum_{m=0}^{\infty}\sum_{n=0}^{\infty}
\alpha^{ij}_{mn}\ell_m(x)\ell_n(y).
\end{align*}
Similarly to \eqref{sum:1}, one can show that
\begin{subequations}\label{sum:y}
\begin{align}\label{sum:y:1}
\left|\sum_{n=k+1}^{\infty} \alpha^{ij}_{mn}\right|
&\leq C\sqrt{2m+1}
\sqrt{\frac{h_{j}}{h_i}}h_{j}^{k}\norm{\partial_y^{k+1}z}_{K_{ij}}.
\end{align} 
and
\begin{align}\label{sum:y:2}
\left|\sum_{n=k+1}^{\infty} \alpha^{i+1,j}_{mn}(-1)^m\right|
&=\left|\sum_{n=k+1}^{\infty} \alpha^{i+1,j}_{mn}\right|
\leq C\sqrt{2m+1}
\sqrt{\frac{h_{j}}{h_{i+1}}}h_{j}^{k}\norm{\partial_y^{k+1}z}_{K_{i+1,j}}.
\end{align}
\end{subequations}

One can write the Legendre expansion for the projection $\Pi_y^{-}z$ as
\[ 
\Pi_y^{-} z(x,y)|_{K_{ij}}=\sum_{m=0}^{k}\sum_{n=0}^{k}
b^{ij}_{mn}\ell_m(x)\ell_n(y)
\]
for some coefficients $b^{ij}_{mn}$.
The pairwise orthogonality of the Legendre polynomials and the first condition
of \eqref{Pi:y:-} yield
$b^{ij}_{mn}=\alpha^{ij}_{mn}$ for $m=0,1,\dots,k$ and  $n=0,1,\dots,k-1$.
Then the pairwise orthogonality of the Legendre polynomials, the second condition of \eqref{Pi:y:-} 
and $\ell_n(y_j)=1$ imply that the remaining coefficients $b^{ij}_{mk}$ satisfy
\begin{align*}
b^{ij}_{mk}&=
\frac{1}{\norm{\ell_m}^2_{I_i}}
\int_{I_{i}}\left(z^{-}_{x,j}(x) - \sum_{s=0}^{k}\sum_{n=0}^{k-1}
	\alpha^{ij}_{sn}\ell_s(x) \right)\ell_m(x)\,\mathrm{d}x
\\
&=\frac{1}{\norm{\ell_m}^2_{I_i}}
\int_{I_{i}}\Big(\sum_{s=0}^{\infty}\sum_{n=0}^{\infty}
\alpha^{ij}_{sn}\ell_s(x)\Big)\ell_m(x)\,\mathrm{d}x-\sum_{n=0}^{k-1}\alpha^{ij}_{mn}
\\
&=\sum_{n=0}^{\infty}\alpha^{ij}_{mn}
-\sum_{n=0}^{k-1}\alpha^{ij}_{mn}
=\sum_{n=k}^{\infty}\alpha^{ij}_{mn} \ \text{ for } m=0,1,\dots,k.
\end{align*}
Hence,
\begin{align*}
\Pi_y^{-} z(x,y)|_{K_{ij}}&
=\sum_{m=0}^{k}
\left[\sum_{n=0}^{k-1}
\alpha^{ij}_{mn}\ell_m(x)\ell_n(y)+
\Big(\sum_{n=k}^{\infty}\alpha^{ij}_{mn}\Big)\ell_m(x)\ell_k(y)\right],
\\
(z-\Pi_y^{-} z)(x,y)|_{K_{ij}}
&=\sum_{n=k+1}^{\infty}\sum_{m=0}^{k}
\alpha^{ij}_{mn}\ell_m(x)\ell_n(y)
+\sum_{n=0}^{\infty}\sum_{m=k+1}^{\infty}
\alpha^{ij}_{mn}\ell_m(x)\ell_n(y)
\\
&\quad 
+\sum_{m=0}^{k}\Big(\sum_{n=k+1}^{\infty}\alpha^{ij}_{mn}\Big)
\ell_m(x)\ell_k(y).
\end{align*}
Since $\hat{\ell}_m(\pm 1)=(\pm 1)^m$, one has the average
\begin{align*}
\average{z-\Pi_y^{-} z}_{i,y}
&=\frac12\Big[(z-\Pi_y^{-} z)^-_{i,y}+(z-\Pi_y^{-} z)^+_{i,y}\Big]
\\
=\frac12 \sum_{n=k+1}^{\infty}\sum_{m=0}^{k}
&\big[\alpha^{ij}_{mn}+\alpha^{i+1,j}_{mn}(-1)^m\big]\ell_n(y)
+\frac12 \sum_{n=0}^{\infty}\sum_{m=k+1}^{\infty}
\big[\alpha^{ij}_{mn}+\alpha^{i+1,j}_{mn}(-1)^m\big]\ell_n(y)
\\
&\quad
+\frac12 \sum_{m=0}^{k}
\left[\sum_{n=k+1}^{\infty}\alpha^{ij}_{mn}
+\sum_{n=k+1}^{\infty}\alpha^{i+1,j}_{mn}(-1)^{m}\right]\ell_k(y)
\\
&
:=B_1(y)+B_2(y)+B_3(y).
\end{align*}
The pairwise orthogonality of the rescaled Legendre polynomials forces
\begin{equation}\label{B1}
\int_{J_j} B_1(y) \jump{\vphu}_{i,y}\,\mathrm{d}y=0 \ 
\text{   for  any } \jump{\vphu}_{i,y}\in \mathcal{P}^k(J_j).
\end{equation}
Similarly to \eqref{A3}, since $h_i=h_{i+1}$ we have
\begin{align}
\left|\int_{J_j} B_2(y) \jump{\vphu}_{i,y}\,\mathrm{d}y\right|
&=\frac12\left|\int_{J_j}
\sum_{n=0}^{k}\sum_{m=k+1}^{\infty}
\big[\alpha^{ij}_{mn}+\alpha^{i+1,j}_{mn}(-1)^m\big]\ell_n(y)
\jump{\vphu}_{i,y}\,\mathrm{d}y\right|
\nonumber\\
\label{B2}
&
\leq Ch_i^{k}\norm{\partial_x^{k+1}z}_{K_{ij}\cup K_{i+1,j}}
\norm{\vphu}_{K_{ij}\cup K_{i+1,j}}.
\end{align}
Furthermore, imitating the derivation of~\eqref{A3}, \eqref{Legd:norm} and 
recalling the inequalities~\eqref{sum:y}, we obtain
\begin{align}
&\left|\int_{J_j} B_3(y) \jump{\vphu}_{i,y}\,\mathrm{d}y\right|
\nonumber\\
&=\frac12\left|\int_{J_j}
\sum_{m=0}^{k}
\left[\sum_{n=k+1}^{\infty}\alpha^{ij}_{mn}
+\sum_{n=k+1}^{\infty}\alpha^{i+1,j}_{mn}(-1)^{m}\right] \ell_k(y)
\jump{\vphu}_{i,y}\,\mathrm{d}y\right|
\nonumber\\
&\leq C\int_{J_j}\sum_{m=0}^{k}
\left[\left|\sum_{n=k+1}^{\infty}\alpha^{ij}_{m,n}\right|
+\left|\sum_{n=k+1}^{\infty}\alpha^{i+1,j}_{m,n}(-1)^m\right|
\right]|\ell_k(y)||\jump{\vphu}_{i,y}|\,\mathrm{d}y
\nonumber\\
&\leq C\left(\int_{J_j}\sum_{m=0}^{k}
\left[\left|\sum_{n=k+1}^{\infty}\alpha^{ij}_{m,n}\right|^2
+\left|\sum_{n=k+1}^{\infty}\alpha^{i+1,j}_{m,n}(-1)^m\right|^2\right]\ell^2_k(y)
\,\mathrm{d}y\right)^{1/2}
\left(\int_{J_j}\jump{\vphu}^2_{i,y}\,\mathrm{d}y\right)^{1/2}
\nonumber\\
&\leq
C\max_{0\leq m\leq k}\left[
\left|\sum_{n=k+1}^{\infty} \alpha^{ij}_{mn}\right|
+\left|\sum_{n=k+1}^{\infty} \alpha^{i+1,j}_{mn}(-1)^m\right|
\right]
\norm{\ell_k}_{J_j}h_i^{-1/2}\norm{\vphu}_{K_{ij}\cup K_{i+1,j}}
\nonumber\\
&\leq C\sqrt{2k+1}
\sqrt{\frac{h_j}{h_i}}h_j^{k}\norm{\partial_y^{k+1}z}_{K_{ij}\cup K_{i+1,j}}
\sqrt{\frac{h_j}{2k+1}}
h_i^{-1/2}\norm{\vphu}_{K_{ij}\cup K_{i+1,j}}
\nonumber\\
\label{B3}
&
\leq Ch_i^{-1}h_j^{k+1}\norm{\partial_y^{k+1}z}_{K_{ij}\cup K_{i+1,j}}
\norm{\vphu}_{K_{ij}\cup K_{i+1,j}}.
\end{align}
Combining \eqref{B1}, \eqref{B2} and \eqref{B3} gives
\begin{align*}
\left|\int_{J_j}\average{z-\Pi_y^{-} z}_{i,y}\jump{\vphu}_{i,y} \,\mathrm{d}y\right|
&\leq\sum_{i=1}^3\left|\int_{J_j}B_i(y)\jump{\vphu}_{i,y} \,\mathrm{d}y\right|
\\
&\hspace{-3cm}\leq C \left[h_i^{k}\norm{\partial_x^{k+1}z}_{K_{ij}\cup K_{i+1,j}}+h_i^{-1}h_j^{k+1}\norm{\partial_y^{k+1}z}_{K_{ij}\cup K_{i+1,j}}
\right]
\norm{\vphu}_{K_{ij}\cup K_{i+1,j}}.
\end{align*}
This completes the proof.
\end{proof}

Lemma~\ref{sup:edge:element} enables us to prove the following bounds for $\Pi_y^{\pm}$ 
and $\Pi_x^{\pm}$ in the  fine-mesh regions $\ds, \dn,
\dw$ and $\de$.

\begin{lemma}\label{lemma:sup:edge:2d}
Assume that Lemma \ref{lemma:prop:2d} is valid when $m=k$.
	Then there exists a constant $C>0$ such that for any $\vphu\in\spc$ one has
	\begin{subequations}\label{sup:edge:2d}
		\begin{align}
		\label{sup:Pi:2d:y:-}
		\left|\sum_{j=1}^{N/4}\sum_{i=N/4+1}^{3N/4-1}\dual{\average{u-\Pi_y^{-} u}_{i,y}}{\jump{\vphu}_{i,y}}_{J_j}
		\right|&\leq C\sq^{1/4}N^{-k}(\ln N)^{k+1}\norm{\vphu}_{\ds},
		\\
		\label{sup:Pi:2d:y:+}
		\left|\sum_{j=3N/4+1}^{N}\sum_{i=N/4+1}^{3N/4-1}\dual{\average{u-\Pi_y^{+} u}_{i,y}}{\jump{\vphu}_{i,y}}_{J_j}
		\right|&\leq C\sq^{1/4}N^{-k}(\ln N)^{k+1}\norm{\vphu}_{\dn},
		\\
		\label{sup:Pi:2d:x:-}
		\left|\sum_{i=1}^{N/4}\sum_{j=N/4+1}^{3N/4-1}\dual{\average{u-\Pi_x^{-} u}_{x,j}}{\jump{\vphu}_{x,j}}_{I_i}
		\right|&\leq C\sq^{1/4}N^{-k}(\ln N)^{k+1}\norm{\vphu}_{\dw},
		\\
		\label{sup:Pi:2d:x:+}
		\left|\sum_{i=3N/4+1}^{N}\sum_{j=N/4+1}^{3N/4-1}\dual{\average{u-\Pi_x^{+} u}_{x,j}}{\jump{\vphu}_{x,j}}_{I_i}
		\right|&\leq C\sq^{1/4}N^{-k}(\ln N)^{k+1}\norm{\vphu}_{\de}.
		\end{align}
	\end{subequations}
\end{lemma}

\begin{proof}
We prove \eqref{sup:Pi:2d:y:-}; the other inequalities
can be proved in a similar manner.
For the smooth component $\bar{u}$, use Lemma~\ref{sup:edge:element}, 
a Cauchy-Schwarz inequality, Lemma~\ref{lemma:prop:2d} 
and the $O(\tau)$ measure of the layer region $\ds$
to obtain 
\begin{align}
&\hspace{-2cm}\left|\sum_{j=1}^{N/4}\sum_{i=N/4+1}^{3N/4-1}
\dual{\average{\bar{u}-\Pi_y^{-} \bar{u}}_{i,y}}{\jump{\vphu}_{i,y}}_{J_j}
\right|\nonumber\\
&\leq 
C\sum_{j=1}^{N/4}\sum_{i=N/4+1}^{3N/4-1} 
H^{-1}\left[
H^{k+1}\norm{\partial_x^{k+1} \bar{u}}_{K_{ij}\cup K_{i+1,j}}  
\right.  \notag\\
&\qquad \left. 
+ h^{k+1}\norm{\partial_y^{k+1}\bar{u}}_{K_{ij}\cup K_{i+1,j}}
\right]	
\norm{\vphu}_{K_{ij}\cup K_{i+1,j}}
\nonumber\\
&\leq CN^{-k}\Big(\norm{\partial_x^{k+1} \bar{u}}_{\ds}
+\norm{\partial_y^{k+1}\bar{u}}_{\ds}
\Big)\norm{\vphu}_{\ds}
\nonumber\\
&\leq C\tau^{1/2}N^{-k}\norm{\vphu}_{\ds}
\leq C\sq^{1/4}N^{-k}(\ln N)^{1/2}\norm{\vphu}_{\ds}. \label{avg:regular:1b}
\end{align}

For the layer component $u^{\mathrm{b}}_{4}$, 
use Lemma~\ref{sup:edge:element}, a Cauchy-Schwarz inequality
and the bounds on derivatives of $u^{\mathrm{b}}_{4}$ from Lemma~\ref{lemma:prop:2d}, 
obtaining
\begin{align}
&\left|\sum_{j=1}^{N/4}\sum_{i=N/4+1}^{3N/4-1}
\dual{\average{u^{\mathrm{b}}_{4}-\Pi_y^{-} u^{\mathrm{b}}_{4}}_{i,y}}{\jump{\vphu}_{i,y}}_{J_j}
\right|\nonumber\\
&\leq 
C\sum_{j=1}^{N/4}\sum_{i=N/4+1}^{3N/4-1} 
H^{-1}\left[
H^{k+1}\norm{\partial_x^{k+1} u^{\mathrm{b}}_{4}}_{K_{ij}\cup K_{i+1,j}}
+ h^{k+1}\norm{\partial_y^{k+1}u^{\mathrm{b}}_{4}}_{K_{ij}\cup K_{i+1,j}}
\right]	
\norm{\vphu}_{K_{ij}\cup K_{i+1,j}}
\nonumber\\
&\leq 
C\sum_{j=1}^{N/4}\sum_{i=N/4+1}^{3N/4-1} 
\left[N^{-k}+N(N^{-1}\ln N)^{k+1}\right]	
\norm{e^{-\beta y/\sqrt{\sq}}}_{K_{ij}\cup K_{i+1,j}}
\norm{\vphu}_{K_{ij}\cup K_{i+1,j}}
\nonumber\\
\label{avg:bry:1}
&\leq CN^{-k}(\ln N)^{k+1}\norm{e^{-\beta y/\sqrt{\sq}}}_{\ds}
\norm{\vphu}_{\ds}
\leq C\sq^{1/4}N^{-k}(\ln N)^{k+1}\norm{\vphu}_{\ds}.
\end{align}

For the remaining components $z=u^{\mathrm{b}}_{i} (i=1,2,3)$
and $z=u^{\mathrm{c}}_{i} (i=1,2,3,4)$, one has
$\norm{z}_{L^{\infty}(\ds)}\leq CN^{-\sigma}$
using Lemma~\ref{lemma:prop:2d}.
Recall that  $\sigma\geq \hat{k}+1$. 
Thus, a Cauchy-Schwarz inequality,
the $L^{\infty}$-stability property~\eqref{L0:stb}  
and an inverse inequality yield
\begin{align}
&\hspace{-1cm}\left|\sum_{j=1}^{N/4}\sum_{i=N/4+1}^{3N/4-1}
\dual{\average{z-\Pi_y^{-} z}_{i,y}}{\jump{\vphu}_{i,y}}_{J_j}
\right|
\nonumber\\
&\leq 
\sum_{j=1}^{N/4} \sum_{i=N/4+1}^{3N/4-1}
\norm{\average{z-\Pi_y^{-} z}_{i,y}}_{J_j}
\norm{\jump{\vphu}_{i,y}}_{J_j}
\nonumber\\
&\leq 
C\sum_{j=1}^{N/4}\sum_{i=N/4+1}^{3N/4-1}
h_j^{1/2}\norm{z-\Pi_y^{-} z}_{L^{\infty}(K_{ij}\cup K_{i+1,j})}
h_i^{-1/2}\norm{\vphu}_{K_{ij}\cup K_{i+1,j}}
\nonumber\\
&\leq 
C\sum_{j=1}^{N/4}\sum_{i=N/4+1}^{3N/4-1}
N^{1/2}h_j^{1/2}\norm{z}_{L^{\infty}(K_{ij}\cup K_{i+1,j})}
\norm{\vphu}_{K_{ij}\cup K_{i+1,j}}
\nonumber\\
&\leq 
C\sum_{j=1}^{N/4}\sum_{i=N/4+1}^{3N/4-1}
h_j^{1/2}N^{-\sigma+1/2}\norm{\vphu}_{K_{ij}\cup K_{i+1,j}}
\nonumber\\
\label{avg:crn:layer}
&
\leq C\tau^{1/2} N^{-\sigma+1}\norm{\vphu}_{\ds}
\leq C\sq^{1/4}N^{-\hat{k}}(\ln N)^{1/2}\norm{\vphu}_{\ds}.
\end{align}

Finally, \eqref{sup:Pi:2d:y:-} follows from \eqref{avg:regular:1b}, \eqref{avg:bry:1},
 \eqref{avg:crn:layer} and a triangle inequality.
\end{proof}

\subsection{Superapproximation properties of vertex-edge-element projectors}
\label{sec:superapprox_vertex_edge_element}

Now we move on to the vertex-edge-element approximation operators
$\Pi_{xy}^{+},\Pi_{xy}^{\pm},\Pi_{xy}^{\mp}$ and $\Pi_{xy}^{-}$. 
For each element $K_{ij}\in \Omega_N$ and each $\vphu\in \spc$,
define the bilinear forms 
\begin{align*}
\mathcal{D}^-_{ij,x}(u-\prog u,\vphu)
&:=
\dual{u-\prog u}{\vphu_x}_{K_{ij}}
-\dual{(u-\prog u)^{-}_{i,y}}{\vphu^{-}_{i,y}}_{J_{j}}
+\dual{(u-\prog u)^{-}_{i-1,y}}{\vphu^{+}_{i-1,y}}_{J_{j}},
\\
\mathcal{D}^+_{ij,x}(u-\prog u,\vphu)
&:=
\dual{u-\prog u}{\vphu_x}_{K_{ij}}
-\dual{(u-\prog u)^{+}_{i,y}}{\vphu^{-}_{i,y}}_{J_{j}}
+\dual{(u-\prog u)^{+}_{i-1,y}}{\vphu^{+}_{i-1,y}}_{J_{j}},
\\
\mathcal{D}^-_{ij,y}(u-\prog u,\vphu)
&:=
\dual{u-\prog u}{\vphu_y}_{K_{ij}}
-\dual{(u-\prog u)^{-}_{x,j}}{\vphu^{-}_{x,j}}_{I_{i}}
+\dual{(u-\prog u)^{-}_{x,j-1}}{\vphu^{+}_{x,j-1}}_{I_{i}},
\\
\mathcal{D}^+_{ij,y}(u-\prog u,\vphu)
&:=
\dual{u-\prog u}{\vphu_y}_{K_{ij}}
-\dual{(u-\prog u)^{+}_{x,j}}{\vphu^{-}_{x,j}}_{I_{i}}
+\dual{(u-\prog u)^{+}_{x,j-1}}{\vphu^{+}_{x,j-1}}_{I_{i}}
\end{align*}
for any projector $\prog$, where we set
$(u-\prog u)^{-}_{0,y}=(u-\prog u)^{-}_{x,0}=(u-\prog u)^{+}_{N,y}=(u-\prog u)^{+}_{x,N}=0$.
The next lemma presents superapproximation results for these bilinear forms.

\begin{lemma}
	\label{superapproximation:element}
	Assume that  Lemma~\ref{lemma:prop:2d} is valid when $m=k$.
		Let $K_{ij}\in\Omega_N$.
		Let $z\in C^{k+2}(\bar  K_{ij})$.
		Then there exists a constant $C>0$ such that for all $\vphu\in\spc$ one has
	\begin{subequations}\label{sup:GR:element:1}
	\begin{align}
	\label{sup:P-:L2}
	|\mathcal{D}^-_{ij,x}(z-\prog ^{-} z,\vphu)|
	&\leq
	Ch_i^{-1}\left[h_{i}^{k+2}\norm{\partial_x^{k+2}z}_{K_{ij}}
	+h_{j}^{k+2}\norm{\partial_y^{k+2}z}_{K_{ij}}
	\right]\norm{\vphu}_{K_{ij}}
	\end{align}
	for $\prog ^{-}=\Pi_{xy}^{-}, \Pi_{xy}^{\mp}$ and
	\begin{align}
	|\mathcal{D}^+_{ij,x}(z-\prog^+ z,\vphu)|
	&\leq
	Ch_i^{-1}\left[h_{i}^{k+2}\norm{\partial_x^{k+2}z}_{K_{ij}}
	+h_{j}^{k+2}\norm{\partial_y^{k+2}z}_{K_{ij}}
	\right]\norm{\vphu}_{K_{ij}}
	\end{align}
	\end{subequations}
	for $\prog^+=\Pi_{xy}^{+}, \Pi_{xy}^{\pm}$.
	Furthermore, one has superapproximation properties for 
	 $\prog ^{-}=\Pi_{xy}^{-}, \Pi_{xy}^{\mp}$ and 
	 $\prog^+=\Pi_{xy}^{+}, \Pi_{xy}^{\pm}$ in the corner layer regions:
	\begin{subequations}\label{sup:GR:element:2}
	\begin{align}
	\label{sup:L2:bilinear:P-}
	\sup_{\vphu\in\spc}\sum_{K_{ij}\in \dws\cup \dwn}
	&\Bigg(\frac{|\mathcal{D}^-_{ij,x}(u-\prog^{-} u,\vphu)|}{\norm{\vphu}_{K_{ij}}}
	\Bigg)^2
	\leq C\varepsilon^{-1/2}(N^{-1}\ln N)^{2(k+1)},
	\\
	\label{sup:L2:bilinear:P+}
	\sup_{\vphu\in\spc}	\sum_{K_{ij}\in \des\cup\den}
	&\Bigg(\frac{|\mathcal{D}^+_{ij,x}(u-\prog^{+} u,\vphu)|}{\norm{v}_{K_{ij}}}
	\Bigg)^2
	\leq C\varepsilon^{-1/2}(N^{-1}\ln N)^{2(k+1)}.
	\end{align}
   \end{subequations}
	Analogous bounds hold true for 
	$\mathcal{D}^-_{ij,y}(u-\prog^{-} u,\vphu)$
	and $\mathcal{D}^+_{ij,y}(u-\prog^{+} u,\vphu)$.
\end{lemma}

\begin{proof}
The inequalities \eqref{sup:GR:element:1} are proved similarly to \cite[Lemma~4.8]{Zhu:2dMC}. 

To prove \eqref{sup:L2:bilinear:P-}, we imitate \cite[Lemma~3.3]{ChengStynes2023}.
Let $\vphu\in\spc$.
Given $K_{ij}\in \dws\cup \dwn$ so $h_i = h_j = h$,
use \eqref{sup:P-:L2} and Lemma~\ref{lemma:prop:2d} to get 
\begin{align*}
|\mathcal{D}^-_{ij,x}(z-\prog ^{-} z,\vphu)|\leq
\begin{cases}
CN^{-(k+2)} \norm{\vphu}_{K_{ij}}
&\text{if } z=\bar{u},
\\
Ch_i^{-1}(N^{-1}\ln N)^{k+2}
\norm{e^{-\beta x/\sqrt{\sq}}}_{K_{ij}} \norm{\vphu}_{K_{ij}}
&\text{if }  z=u_1^{\mathrm{b}},
\\
Ch_i^{-1}(N^{-1}\ln N)^{k+2}
\norm{e^{-\beta x/\sqrt{\sq}}e^{-\beta y/\sqrt{\sq}}}_{K_{ij}}
\norm{\vphu}_{K_{ij}}
&\text{if }  z=u_1^{\mathrm{c}}.
\end{cases}
\end{align*}
After some straightforward calculations, one obtains
\begin{align*}
\sum_{K_{ij}\in \dws\cup \dwn}
\Bigg(\frac{|\mathcal{D}^-_{ij,x}(z-\prog^{-} z,\vphu)|}{\norm{\vphu}_{K_{ij}}}
\Bigg)^2
\leq 
\begin{cases}
CN^{-2(k+1)} &\text{if }  z=\bar{u},
\\
C\varepsilon^{-1/2}(N^{-1}\ln N)^{2(k+1)} &\text{if }  z=u_1^{\mathrm{b}},
\\
C (N^{-1}\ln N)^{2(k+1)} &\text{if }  z=u_1^{\mathrm{c}}.
\end{cases}
\end{align*}
The remaining components $u_i^{\mathrm{b}}$ and $u_i^{\mathrm{c}}$ for $i=2,3,4$ can be bounded similarly, so we get~\eqref{sup:L2:bilinear:P-}.
One can likewise prove \eqref{sup:L2:bilinear:P+}.
\end{proof}

\subsection{Two bounds on $\uph-\prou u$}
Recall from Section~\ref{sec:LDG:method} that $\bm w = (u,p,q)$.
Set $\prow \bm w:=(\prou u,\prop p,\proq q)\in\spc^3$, where $\prou, \prop$ and $\proq$ were defined in \eqref{prj:u:2d}. 
Then the LDG error is 
$\bm e:=(e_u,e_p,e_q) :=(u-\uph,p-\pph,q-\qph)$,  which can be written as
\begin{equation}\label{error:decomposition:2d}
\bm e = \bm w - \bm W = (\bm w-\prow \bm w)-(\wph-\prow\bm w)  = \bm \eta-\bm \xi,
\end{equation}
where we define 
\begin{equation*}
\begin{split}
\bm \eta  &:= (\eta_u,\eta_p,\eta_q)= (u-\prou u,p-\prop p,q-\proq q),
\\
\bm \xi &:= (\xi_u,\xi_p,\xi_q)= (\uph-\prou u,\pph-\prop p,\qph-\proq q)\in \spc^3.
\end{split}
\end{equation*}

In this subsection we bound $\xi_u=\uph-\prou u$
along the fine/coarse mesh interfaces.

\begin{lemma}\label{lemma:sup:bry}
There exists a constant $C>0$ such that
\begin{subequations}
\begin{align}
\label{bry:L}
\left(\sum_{j=1}^{N}\norm{(\xi_u)_{N/4,y}^-}^2_{J_j}\right)^{1/2}
\leq C\sqrt{\frac{\tau}{\sq}}\,\Big[\enorm{\bm\xi}+\sq^{1/4}(N^{-1}\ln N)^{k+1}\Big],
\\
\label{bry:R}
\left(\sum_{j=1}^{N}\norm{(\xi_u)_{3N/4,y}^+}^2_{J_j}\right)^{1/2}
\leq C\sqrt{\frac{\tau}{\sq}}\,\Big[\enorm{\bm\xi}+\sq^{1/4}(N^{-1}\ln N)^{k+1}\Big].
\end{align}
\end{subequations}
\end{lemma}

	
\begin{proof}
The two inequalities can be proved in a similar manner; we shall prove~\eqref{bry:L}.
Similarly to the proof of \cite[Lemma~3.4]{ChengJiangStynes2023}, one has the element relationships
\begin{align*}
\norm{(\xi_u)_x}_{K_{ij}}
& \leq 
C\varepsilon^{-1}(\norm{\eta_p}_{K_{ij}}+\norm{\xi_p}_{K_{ij}})
+C\sup_{\vphu\in\spc}
\frac{|\mathcal{D}^{-}_{ij,x}(\eta_u,\vphu)|}{\norm{\vphu}_{K_{ij}}}
\intertext{and}
h_i^{-1}\norm{\jump{\xi_u}_{i-1,y}}^2_{J_j}
&\leq 
C\Bigg[
\varepsilon^{-2}\left(\norm{\eta_p}_{K_{ij}}^2+\norm{\xi_p}_{K_{ij}}^2\right)
+\norm{(\xi_u)_x}_{K_{ij}}^2  \\
&\qquad  +\sup_{\vphu\in\spc}\left(\frac{|\mathcal{D}^{-}_{ij,x}(\eta_u,\vphu)|}{\norm{\vphu}_{K_{ij}}}\right)^2
\Bigg]
\end{align*}
for any $K_{ij}\in\Omega_{1}^x$, which yield 
\begin{align}
\norm{(\xi_u)_x}_{K_{ij}}^2
&+h_i^{-1}\norm{\jump{\xi_u}_{i-1,y}}^2_{J_j}
\nonumber\\
&\leq 
C\left[
\varepsilon^{-2}\left(\norm{\eta_p}_{K_{ij}}^2+\norm{\xi_p}_{K_{ij}}^2\right)
+\sup_{\vphu\in\spc}\left(\frac{|\mathcal{D}^{-}_{ij,x}(\eta_u,\vphu)|}{\norm{\vphu}_{K_{ij}}}\right)^2
\right]. \label{relation:xiu:xiq}
\end{align}

Starting from the elementary identity
\[
(\xi_u)_{N/4,y}^-
=\sum_{i=1}^{N/4}\int_{I_i}(\xi_{u})_{x} \,dx + \sum_{i=1}^{N/4}\jump{\xi_u}_{i-1,y}
\]
with $\jump{\xi_u}_{0,y}=(\xi_{u})^+_{0,y}$,
we use the Cauchy-Schwarz inequality to obtain
\begin{align}
|(\xi_u)_{N/4,y}^-|
&\leq \sum_{i=1}^{N/4}\left[h_i^{1/2} \norm{(\xi_{u})_{x}}_{I_i}+|\jump{\xi_u}_{i-1,y}|\right]
\nonumber\\
&\leq \left(\sum_{i=1}^{N/4} h_i\right)^{1/2}
\left(
\sum_{i=1}^{N/4}\Big[\norm{(\xi_{u})_{x}}^2_{I_i}
+h_i^{-1}|\jump{\xi_u}_{i-1,y}|^2\Big]
\right)^{1/2}
\nonumber\\
\label{ineq}
&= \tau^{1/2}
\left(
\sum_{i=1}^{N/4}\Big[\norm{(\xi_{u})_{x}}^2_{I_i}
+h_i^{-1}|\jump{\xi_u}_{i-1,y}|^2\Big]
\right)^{1/2}.
\end{align}
Then \eqref{ineq}, \eqref{relation:xiu:xiq},
Lemmas \ref{lemma:app:2d} and \ref{superapproximation:element} give
\begin{align*}
\sum_{j=1}^{N} & \norm{(\xi_u)_{N/4,y}^-}^2_{J_j}
=\sum_{j=1}^{N}\int_{J_j}|(\xi_u)_{N/4,y}^-|^2\,\mathrm{d}y
\\
&\leq \tau \sum_{j=1}^{N}\int_{J_j} \sum_{i=1}^{N/4}
\Big(\norm{(\xi_u)_x}_{I_{i}}^2
+h_i^{-1}|\jump{\xi_u}_{i-1,y}|^2\Big) \,\mathrm{d}y
\\
&=\tau \left[\sum_{j=1}^{N}\sum_{i=1}^{N/4}
\left(\norm{(\xi_u)_x}_{K_{ij}}^2
+h_i^{-1}\norm{\jump{\xi_u}_{i-1,y}}^2_{J_j}
\right)
\right]
\\
&\leq C\tau\left[
\sq^{-2}\left(\norm{\eta_p}_{\Omega_{1}^x}^2+\norm{\xi_p}_{\Omega_{1}^x}^2
\right)
+\sup_{\vphu\in\spc}\sum_{K_{ij}\in \Omega_{1}^x}
\Bigg(\frac{|\mathcal{D}^-_{ij,x}(u-\prou u,\vphu)|}{\norm{\vphu}_{K_{ij}}}
\Bigg)^2
\right]
\\
&
\leq C\tau\sq^{-1}\left[
\sq^{-1}\norm{\xi_p}^2+\sq^{-1}\norm{\eta_p}_{\Omega_{1}^x}^2
+\sq\sup_{\vphu\in\spc}\sum_{K_{ij}\in \dws\cup\dwn}
\Bigg(\frac{|\mathcal{D}^-_{ij,x}(u-\prou u,\vphu)|}{\norm{\vphu}_{K_{ij}}}
\Bigg)^2
\right]
\\
&
\leq C\tau\sq^{-1}
\left[\enorm{\bm\xi}^2+\sq^{1/2}(N^{-1}\ln N)^{2(k+1)}\right],
\end{align*}
where we used the property $\mathcal{D}^-_{ij,x}(u-\prou u,\vphu)
=\mathcal{D}^-_{ij,x}(u-\Pi_x^{-} u,\vphu)=0$ for each $K_{ij}\in\dw$.
The proof is complete.
\end{proof}

\section{Energy-norm and balanced-norm error estimates}\label{section:error:analysis}

In this section we establish the main result of the paper, Theorem~\ref{thm:2d}. 
To make the main arguments in its proof easier to follow, we assume that $b$ is a positive constant.
If $b\geq 2\beta^2>0$ is not constant
and $b\in W^{1,\infty}(\bar{\Omega})$, one can still obtain the same error bound by modifying slightly the projector~$\prou$;  see Remark~\ref{remark:b:nonconstant}.

The proof of the balanced-norm error estimate in Theorem~\ref{thm:2d} hinges on the superclose energy-norm bound \eqref{xisuperclose}, so we include an energy-norm error estimate in the theorem 
also since it requires only a little extra work.

\begin{theorem}\label{thm:2d}
Assume that 
$b$ is a positive constant and that Lemma~\ref{lemma:prop:2d} is valid for $m=k+1$.
	Let $\bm w=(u,p,q)=(u,\sq u_x,\sq u_y)$ be the true solution of the reaction-diffusion problem \eqref{spp:R-D:2d}.
	Let $\wph=(\uph,\pph,\qph)\in \spc^3$ be the numerical solution of
	the LDG scheme \eqref{LDG:scheme:2d} with layer-upwind flux \eqref{flux:diffusion:2d}
	on the Shishkin mesh \eqref{mesh:point}	
	with $\sigma\geq \hat{k}+1$, where $\hat{k}=k+1$ for even $k$ and $\hat{k}=k$ for odd $k$.
	Then there exists a constant $C>0$, which is independent of $\varepsilon$ and $N$, such that
\begin{align*}
\enorm{\bm w -\wph}
& \leq C\Big[\sq^{1/4}(N^{-1}\ln N)^{k+1}+\sq^{1/2}N^{-(k+1/2)}+\sq^{1/2} N^{-\hat{k}} \\
&\qquad	+\sq^{3/4}N^{-k}(\ln N)^{k+1}+N^{-(k+1)}\Big]  
\intertext{and}
\ba{\bm w -\wph}
&\leq C\Big[(N^{-1}\ln N)^{k+1}+\sq^{1/4}N^{-(k+1/2)}+\sq^{1/4} N^{-\hat{k}}
	+\sq^{1/2}N^{-k}(\ln N)^{k+1}\Big],   
\end{align*}
where $\enorm{\cdot}$ and $\ba{\cdot}$ were defined in~\eqref{energy:norm:2d} and~\eqref{balanced:norm:2d}.
\end{theorem}

\begin{proof}
From the Galerkin orthogonality property 
\eqref{Galerkin:orthogonality} 
and the decomposition \eqref{error:decomposition:2d}, one has 
$B(\bm \xi;\vph)=B(\bm \eta;\vph)$ for all $\vph=(\vphu,\vphp,\vphq)\in \spc^3$.
Taking $\vph=\bm \xi$ and recalling \eqref{B:def:2d} and \eqref{energy:norm:2d}, one gets 
\begin{align}
\label{energy:equation:2d}
\enorm{\bm \xi}^2
&\;=B(\bm \xi;\bm \xi)= B(\bm \eta;\bm \xi):= \sum_{i=1}^8 T_i,
\end{align}
where
\begin{align*}
T_1&=\dual{b\eta_u}{\xi_u},\quad
T_2=\sq^{-1}\dual{\eta_p}{\xi_p},\quad
T_3=\sq^{-1}\dual{\eta_q}{\xi_q},\\
T_4&=\dual{\eta_{u}}{(\xi_p)_x}
+\sum_{j=1}^{N}\sum_{i=1}^{N-1}
\dual{(\widehat{\eta_u})_{i,y}}{\jump{\xi_p}_{i,y}}_{J_j},
\;
T_5=\dual{\eta_{u}}{(\xi_q)_y}+
\sum_{i=1}^{N}\sum_{j=1}^{N-1}\dual{(\widehat{\eta_u})_{x,j}}{\jump{\xi_q}_{x,j}}_{I_i},\\
T_6&=\dual{\eta_p}{(\xi_u)_x}
+\sum_{j=1}^{N}\sum_{i=0}^{N}
\dual{(\widehat{\eta_p})_{i,y}}{\jump{\xi_u}_{i,y}}_{J_j},
\;
T_7=\dual{\eta_q}{(\xi_u)_y}
+\sum_{i=1}^{N}\sum_{j=0}^{N}
\dual{(\widehat{\eta_q})_{x,j}}{\jump{\xi_u}_{x,j}}_{I_i}.
\end{align*}
Here the ``hat" terms are specified in~\eqref{flux:diffusion:2d}.

By hypothesis $b$ is constant, so 
$\dual{b\eta_{u}}{\xi_{u}}_{\dreg}=b\dual{(u-\Pi u)}{\xi_{u}}_{\dreg}=0$ from~\eqref{L2:prj:2d}.
Thus, a Cauchy-Schwarz inequality and inequality~\eqref{bound:eu:L2:2d} 
of Lemma~\ref{lemma:app:2d} yield
\begin{align}\label{T1}
|T_1|
&=|\dual{b\eta_{u}}{\xi_{u}}_{\Omega\backslash\dreg}|
\leq C\norm{\eta_{u}}_{\Omega\backslash\dreg}\norm{\xi_{u}}
\leq C\sq^{1/4} (N^{-1}\ln N)^{k+1}\enorm{\bm\xi}.
\end{align} 
Similarly, from inequalities \eqref{bound:ep:L2:2d} and \eqref{bound:eq:L2:2d} 
of Lemma~\ref{lemma:app:2d} one has
\begin{align}\label{T2}
|T_2|
&=|\sq^{-1}\dual{\eta_{p}}{\xi_{p}}_{\dxa\cup\dxc}|
\leq \sq^{-1}\norm{\eta_{p}}_{\dxa\cup\dxc}\norm{\xi_{p}}
\leq C\sq^{1/4} (N^{-1}\ln N)^{k+1}\enorm{\bm\xi}
\end{align} 
and
\begin{align}\label{T3}
|T_3|
&=|\sq^{-1}\dual{\eta_{q}}{\xi_{q}}_{\dya\cup\dyc}|
\leq \sq^{-1}\norm{\eta_{q}}_{\dya\cup\dyc}\norm{\xi_{q}}
\leq C\sq^{1/4} (N^{-1}\ln N)^{k+1}\enorm{\bm\xi}.
\end{align} 

In Lemma~\ref{lem:T4} we shall show that 
\begin{equation}\label{T4}
|T_4| \leq C\Big[\sq^{1/4}(N^{-1}\ln N)^{k+1}+\sq^{1/2}N^{-\hat{k}}
		+\sq^{3/4}N^{-k}(\ln N)^{k+1}\Big]\enorm{\bm\xi}.
\end{equation}
A similar argument yields the same bound for $T_5$.

The definitions \eqref{L2:prj:2d}, \eqref{Pi:x:+}, \eqref{Pi:x:-}
and \eqref{flux:diffusion:p:2d}  give
\begin{align*}
\dual{p-\Pi_x^{-} p}{\vphu_x}_{K_{ij}}&=\dual{p-\Pi_x^{+} p}{\vphu_x}_{K_{ij}}=\dual{p-\Pi p}{\vphu_x}_{K_{ij}}=0 
\quad \forall K_{ij}\in \Omega_N;
\\
\dual{(\widehat{\eta_p})_{i,y}}{\jump{\vphu}_{i,y}}_{J_{j}}
&=\dual{(p-\Pi_x^+ p)^{+}_{i,y}}{\jump{\vphu}_{i,y}}_{J_{j}}=0,\;
i=0,1,\dots,N/4-1 \text{ and } j=1,2,\dots,N;
\\
\dual{(\widehat{\eta_p})_{i,y}}{\jump{\vphu}_{i,y}}_{J_{j}}
&=\dual{(p-\Pi_x^- p)^{-}_{i,y}}{\jump{\vphu}_{i,y}}_{J_{j}}=0,\;
i=3N/4+1,\dots,N \text{ and } j=1,2,\dots,N
\end{align*}
for any $\vphu\in \mathcal{Q}^{k}(K_{ij})$. But $\eta_p=p-\prop p$,
so recalling \eqref{prj:u:2d} and \eqref{flux:diffusion:p:2d} we see that
\begin{align*}
T_6 &= \sum_{j=1}^{N}
\dual{(\eta_p)_{N/4,y}^+}{\jump{\xi_u}_{N/4,y}}_{J_j}
+\sum_{j=1}^{N}
\dual{(\eta_p)_{3N/4,y}^-}{\jump{\xi_u}_{3N/4,y}}_{J_j} \\
&\qquad +\sum_{j=1}^{N}\sum_{i=N/4+1}^{3N/4-1}\dual{\average{\eta_p}_{i,y}}{\jump{\xi_u}_{i,y}}_{J_j}.
\end{align*}
For the first term here, one has
\begin{align}\label{Two:terms}
\sum_{j=1}^{N}\dual{(\eta_p)_{N/4,y}^+}{\jump{\xi_u}_{N/4,y}}_{J_j}
=\sum_{j=1}^{N}\dual{(\eta_p)_{N/4,y}^+}{(\xi_u)^+_{N/4,y}}_{J_j}
-\sum_{j=1}^{N}\dual{(\eta_p)_{N/4,y}^+}{(\xi_u)^-_{N/4,y}}_{J_j}.
\end{align}
A  Cauchy-Schwarz inequality, an inverse inequality and \eqref{bound:ep:L0:2d} give
\begin{align}
\left|\sum_{j=1}^{N}\dual{(\eta_p)_{N/4,y}^+}{(\xi_u)^+_{N/4,y}}_{J_j}\right|
&\leq  \left(\sum_{j=1}^{N}\norm{(\eta_p)_{N/4,y}^+}^2_{J_j}\right)^{1/2}
       \left(\sum_{j=1}^{N}\norm{(\xi_u)_{N/4,y}^+}^2_{J_j}\right)^{1/2}
\nonumber\\
\label{T61}
&\leq C\norm{\eta_p}_{L^{\infty}(\dxb)}N^{1/2}\norm{\xi_u}
 \leq C\sq^{1/2}N^{-(k+1/2)}\enorm{\bm\xi}.
\end{align}
Another Cauchy-Schwarz inequality, inequality \eqref{bry:L} of Lemma~\ref{lemma:sup:bry} and \eqref{bound:ep:L0:2d} 
show that
\begin{align}
\left|\sum_{j=1}^{N}\dual{(\eta_p)_{N/4,y}^+}{(\xi_u)^-_{N/4,y}}_{J_j}\right| 
&\leq
\left(\sum_{j=1}^{N}\norm{(\eta_p)_{N/4,y}^+}^2_{J_j}\right)^{1/2}
\left(\sum_{j=1}^{N}\norm{(\xi_u)_{N/4,y}^-}^2_{J_j}\right)^{1/2}
\nonumber\\
&\leq C\norm{\eta_p}_{L^{\infty}(\dxb)}\sqrt{\frac{\tau}{\sq}}
\Big[\enorm{\bm\xi}+\sq^{1/4}(N^{-1}\ln N)^{k+1}\Big]
\nonumber\\
\label{T62}
&\leq C\sq^{1/4}N^{-(k+1)}(\ln N)^{1/2}\enorm{\bm\xi}
+C\sq^{1/2}N^{-2(k+1)}(\ln N)^{k+3/2}.
\end{align}
From \eqref{Two:terms}, \eqref{T61}, \eqref{T62} and a triangle inequality, we obtain
\begin{align}
\left|\sum_{j=1}^{N}\dual{(\eta_p)_{N/4,y}^+}{\jump{\xi_u}_{N/4,y}}_{J_j}\right|
&\leq
C\Big[\sq^{1/2}N^{-(k+1/2)}+\sq^{1/4}N^{-(k+1)}(\ln N)^{1/2}\Big]
\enorm{\bm\xi}
\nonumber\\
\label{inner:product:L}
&\qquad\qquad +C\sq^{1/2}N^{-2(k+1)}(\ln N)^{k+3/2}.
\end{align}
Analogously, one has
\begin{align}
\left|\sum_{j=1}^{N}\dual{(\eta_p)_{3N/4,y}^-}{\jump{\xi_u}_{3N/4,y}}_{J_j}\right|
&\leq
C\Big[\sq^{1/2}N^{-(k+1/2)}+\sq^{1/4}N^{-(k+1)}(\ln N)^{1/2}\Big]
\enorm{\bm\xi}
\nonumber\\
\label{inner:product:R}
&\qquad\qquad
+C\sq^{1/2}N^{-2(k+1)}(\ln N)^{k+3/2}.
\end{align}
The inequality \eqref{sup:L2:2d:p} of Lemma \ref{lemma:sup:L2:2d} gives
\begin{align}\label{T63}
\left|\sum_{j=1}^{N}\sum_{i=N/4+1}^{3N/4-1}\dual{\average{\eta_p}_{i,y}}{\jump{\xi_u}_{i,y}}_{J_j}\right|
\leq  C\sq^{1/2}N^{-\hat{k}} \norm{\xi_u}
\leq  C\sq^{1/2}N^{-\hat{k}}\enorm{\bm\xi}. 
\end{align}
Adding \eqref{inner:product:L}, \eqref{inner:product:R} and \eqref{T63}, we get
\begin{align}\label{T6}
|T_6|&\leq
C\Big[\sq^{1/2}N^{-(k+1/2)}+\sq^{1/4}N^{-(k+1)}(\ln N)^{1/2}
+\sq^{1/2}N^{-\hat{k}}\Big]\enorm{\bm\xi}
\nonumber\\
&\qquad +C\sq^{1/2}N^{-2(k+1)}(\ln N)^{k+3/2}.
\end{align}
In a similar fashion, one can derive the same bound as~\eqref{T6}  for $T_7$.

Substituting \eqref{T1}--\eqref{T4} and~\eqref{T6} into \eqref{energy:equation:2d}
and using Young's inequality, we see that
\begin{align*}
\enorm{\bm\xi}^2\leq \frac12 \enorm{\bm\xi}^2
	&+ C\Big[\sq^{1/2}(N^{-1}\ln N)^{2(k+1)}
	 + \sq N^{-(2k+1)} +\sq N^{-2\hat{k}}  \\
	&\qquad +\sq^{3/2}N^{-2k}(\ln N)^{2(k+1)}\Big],
\end{align*}
so 
\begin{equation}\label{xisuperclose}
	\enorm{\bm\xi} \leq 
	C\sq^{1/4}\Big[(N^{-1}\ln N)^{k+1}+ \sq^{1/4} N^{-(k+1/2)}
	+\sq^{1/4} N^{-\hat{k}}+\sq^{1/2}N^{-k}(\ln N)^{k+1}\Big].
\end{equation}
Then the trivial inequality 	$\ba{\bm\xi}\leq C\sq^{-1/4}\enorm{\bm\xi}$ yields 
\[
\ba{\bm\xi} \leq C\Big[(N^{-1}\ln N)^{k+1}+ \sq^{1/4} N^{-(k+1/2)}
	+\sq^{1/4} N^{-\hat{k}}+\sq^{1/2}N^{-k}(\ln N)^{k+1}\Big].
\]

Lemma~\ref{lemma:app:2d} yields
$\enorm{\bm \eta}
\leq C\big[\sq^{-1/2}(\norm{\eta_p}+\norm{\eta_q})+\norm{\eta_u}\big]
\leq C\big[\sq^{1/4}(N^{-1}\ln N)^{k+1}+N^{-(k+1)}\big]$
and 
$\ba{\bm \eta}
\leq C\big[\sq^{-3/4}(\norm{\eta_p}+\norm{\eta_q})+\norm{\eta_u}\big]
\leq C(N^{-1}\ln N)^{k+1}$
as $\sigma\geq \hat{k}+1\geq k+1$.
Thus, the decomposition $\bm e=\bm \eta-\bm \xi$ and the above bounds on $\bm\xi$ and $\bm\eta$ yield
\begin{align}
\enorm{\bm e}
& \leq C\Big[\sq^{1/4}(N^{-1}\ln N)^{k+1}+\sq^{1/2}N^{-(k+1/2)}+\sq^{1/2} N^{-\hat{k}} \notag\\
&\qquad +\sq^{3/4}N^{-k}(\ln N)^{k+1}+N^{-(k+1)}\Big]  \notag 
\intertext{and}
\ba{\bm e} 
&\leq C\Big[(N^{-1}\ln N)^{k+1}+\sq^{1/4}N^{-(k+1/2)}+\sq^{1/4} N^{-\hat{k}}
	+\sq^{1/2}N^{-k}(\ln N)^{k+1}\Big],   \notag 
\end{align}
completing the proof of Theorem~\ref{thm:2d}.
\end{proof}

The following estimate was used above.

\begin{lemma}\label{lem:T4}
In the proof of Theorem~\ref{thm:2d}, one has
\[
|T_4| \leq C\Big[\sq^{1/4}(N^{-1}\ln N)^{k+1}+\sq^{1/2}N^{-\hat{k}}
		+\sq^{3/4}N^{-k}(\ln N)^{k+1}\Big]\enorm{\bm\xi}.
\]
\end{lemma}
\begin{proof}
We write 
\begin{align*}
T_4 =\sum_{j=1}^{N}\sum_{i=1}^{N} 
\Big[
\dual{\eta_u}{(\xi_p)_x}_{K_{ij}}
-\dual{(\widehat{\eta_u})_{i,y}}{(\xi_p)^{-}_{i,y}}_{J_{j}}  
+\dual{(\widehat{\eta_u})_{i-1,y}}{(\xi_p)^{+}_{i-1,y}}_{J_{j}}
\Big]
\end{align*}
where $(\widehat{\eta_u})_{i,y} (i=0,1,\dots,N)$ is specified by \eqref{flux:diffusion:2d} with the special definitions
$(\widehat{\eta_u})_{0,y}=(\widehat{\eta_u})_{N,y}=0$.
Divide this sum into nine separate sums
according to the different definitions 
of~$\prou$ in~\eqref{prj:u:2d} over nine subregions, viz.,
$T_4:=\sum_{i=1}^9 T_{4i}$,
where 
\begin{align*}
T_{41}&=\sum_{j=1}^{N/4}\sum_{i=1}^{N/4}\mathcal{D}^-_{ij,x}(u-\Pi_{xy}^{-} u,\xi_p), \quad
T_{42}=\sum_{j=N/4+1}^{3N/4}\sum_{i=1}^{N/4}\mathcal{D}^-_{ij,x}(u-\Pi_{x}^{-} u,\xi_p), \\
T_{43}&=\sum_{j=3N/4+1}^{N}\sum_{i=1}^{N/4}\mathcal{D}^{-}_{ij,x}(u-\Pi_{xy}^{\mp} u,\xi_p),
\\
T_{44}&=\dual{u-\Pi_y^{-} u}{(\xi_{p})_{x}}_{\ds}
+\sum_{j=1}^{N/4}\sum_{i=N/4+1}^{3N/4-1}\dual{\average{u-\Pi_y^{-} u}_{i,y}}{\jump{\xi_p}_{i,y}}_{J_j}
\\
&\quad 
+\sum_{j=1}^{N/4}\dual{(u-\Pi^-_{xy}u)^-_{N/4,y}}{(\xi_p)^+_{N/4,y}}_{J_j}
-\sum_{j=1}^{N/4}\dual{(u-\Pi^{\pm}_{xy}u)^+_{3N/4,y}}{(\xi_p)^-_{3N/4+1,y}}_{J_j},
\\
T_{45}&=\dual{u-\Pi u}{(\xi_{p})_{x}}_{\dreg}
+\sum_{j=N/4+1}^{3N/4}\sum_{i=N/4+1}^{3N/4-1}
\dual{\average{u-\Pi u}_{i,y}}{\jump{\xi_p}_{i,y}}_{J_j}
\\
&\quad 
+\sum_{j=N/4+1}^{3N/4}\dual{(u-\Pi^-_{x}u)^-_{N/4,y}}{(\xi_p)^+_{N/4,y}}_{J_j}
-\sum_{j=N/4+1}^{3N/4}\dual{(u-\Pi^{+}_{x}u)^+_{3N/4,y}}{(\xi_p)^-_{3N/4+1,y}}_{J_j},
\\
T_{46}&=\dual{u-\Pi_y^{+} u}{(\xi_{p})_{x}}_{\dn}
+\sum_{j=3N/4+1}^{N}\sum_{i=N/4+1}^{3N/4-1}\dual{\average{u-\Pi_y^+ u}_{i,y}}{\jump{\xi_p}_{i,y}}_{J_j}
\\
&\quad 
+\sum_{j=3N/4+1}^{N}\dual{(u-\Pi^{\mp}_{xy}u)^-_{N/4,y}}{(\xi_p)^+_{N/4,y}}_{J_j}
-\sum_{j=3N/4+1}^{N}\dual{(u-\Pi^{+}_{xy}u)^+_{3N/4,y}}{(\xi_p)^-_{3N/4+1,y}}_{J_j}
,
\\
T_{47}&=\sum_{j=1}^{N/4}\sum_{i=3N/4+1}^{N}\mathcal{D}^+_{ij,x}(u-\Pi_{xy}^{\pm} u,\xi_p), \quad
T_{48}=
\sum_{j=N/4+1}^{3N/4}\sum_{i=3N/4+1}^{N}\mathcal{D}^+_{ij,x}(u-\Pi_{x}^{+} u,\xi_p),
\\
T_{49}&=\sum_{j=3N/4+1}^{N}\sum_{i=3N/4+1}^{N}\mathcal{D}^{+}_{ij,x}(u-\Pi_{xy}^{+} u,\xi_p).
\end{align*}

A Cauchy-Schwarz inequality and Lemma~\ref{superapproximation:element} yield
\begin{align*}
|T_{41}|&\leq \left[\sum_{K_{ij}\in \dws}
\Bigg(\frac{|\mathcal{D}^-_{ij,x}(u-\Pi_{xy}^{-} u,\xi_p)|}{\norm{\xi_p}_{K_{ij}}}
\Bigg)^2\right]^{1/2}\norm{\xi_p}_{\dws}
\leq C\sq^{-1/4}(N^{-1}\ln N)^{k+1}\norm{\xi_p},\\
|T_{43}|&\leq \left[\sum_{K_{ij}\in \dwn}
\Bigg(\frac{|\mathcal{D}^-_{ij,x}(u-\Pi_{xy}^{\mp} u,\xi_p)|}{\norm{\xi_p}_{K_{ij}}}
\Bigg)^2\right]^{1/2}\norm{\xi_p}_{\dwn}
\leq C\sq^{-1/4}(N^{-1}\ln N)^{k+1}\norm{\xi_p},\\
|T_{47}|&\leq \left[\sum_{K_{ij}\in \des}
\Bigg(\frac{|\mathcal{D}^+_{ij,x}(u-\Pi_{xy}^{\pm} u,\xi_p)|}{\norm{\xi_p}_{K_{ij}}}
\Bigg)^2\right]^{1/2}\norm{\xi_p}_{\des}
\leq C\sq^{-1/4}(N^{-1}\ln N)^{k+1}\norm{\xi_p},\\
|T_{49}|&\leq \left[\sum_{K_{ij}\in \den}
\Bigg(\frac{|\mathcal{D}^+_{ij,x}(u-\Pi_{xy}^{+} u,\xi_p)|}{\norm{\xi_p}_{K_{ij}}}
\Bigg)^2\right]^{1/2}\norm{\xi_p}_{\den}
\leq C\sq^{-1/4}(N^{-1}\ln N)^{k+1}\norm{\xi_p}.
\end{align*}

The definition \eqref{Pi:x:-} gives immediately
\[
\dual{u-\Pi_x^{-} u}{\vphu_x}_{K_{ij}}=0
\text{ and  } \dual{(u-\Pi_x^{-} u)^{-}_{i,y}}{\vphu_{i,y}}_{J_{j}}=0
\]
for any $\vphu\in \mathcal{Q}^{k}(K_{ij})$ and $i,j=1,2,\dots,N$.
Hence
\begin{align*}
T_{42}&=\sum_{j=N/4+1}^{3N/4}\sum_{i=1}^{N/4}
\Big[
\dual{u-\Pi_x^{-} u}{(\xi_p)_x}_{K_{ij}}
-\dual{(u-\Pi_x^{-} u)^{-}_{i,y}}{(\xi_p)^{-}_{i,y}}_{J_{j}}
\nonumber\\
&\qquad
+\dual{(u-\Pi_x^{-} u)^{-}_{i-1,y}}{(\xi_p)^{+}_{i-1,y}}_{J_{j}}
\Big]=0,
\end{align*}
where $(u-\Pi_x^{-} u)^{-}_{0,y} := \widehat{(u-\Pi_x^{-} u)}_{0,y} =0$.
Similarly, $T_{48}=0$.

A Cauchy-Schwarz inequality, an inverse inequality and \eqref{bound:eu:L2:2d} 
show that
\begin{align}
|\dual{u-\Pi_y^{-} u}{(\xi_{p})_{x}}_{\ds}|
&\leq CN\norm{u-\Pi_y^- u}_{\ds} \norm{\xi_p}_{\ds}\nonumber\\
\label{ineq:1}
&\leq C\sq^{1/4}N^{-k}(\ln N)^{k+1}\norm{\xi_p}_{\ds}.
\end{align}
The inequality \eqref{sup:Pi:2d:y:-} gives
\begin{align}\label{ineq:2}
\left|\sum_{j=1}^{N/4}\sum_{i=N/4+1}^{3N/4-1}
\dual{\average{u-\Pi_y^{-} u}_{i,y}}{\jump{\xi_p}_{i,y}}_{J_j}\right|
\leq  C\sq^{1/4}N^{-k}(\ln N)^{k+1}\norm{\xi_p}_{\ds}.
\end{align}
Using the Cauchy-Schwarz and inverse inequalities and \eqref{bound:eu:L0:2d} yields
\begin{align}
\left|\sum_{j=1}^{N/4}\dual{(u-\Pi^-_{xy}u)^-_{N/4,y}}{(\xi_p)^+_{N/4,y}}_{J_j}\right|
&\leq \sum_{j=1}^{N/4}\norm{(u-\Pi^-_{xy}u)^-_{N/4,y}}_{J_j}
\norm{(\xi_p)^+_{N/4,y}}_{J_j}
\nonumber\\
&\leq C\sum_{j=1}^{N/4} h_j^{1/2}\norm{\eta_u}_{L^{\infty}(K_{N/4,j})} N^{1/2}
\norm{\xi_p}_{K_{N/4+1,j}}
\nonumber\\
&\leq C\tau^{1/2}N^{1/2}\norm{\eta_u}_{L^{\infty}(\dws)}\norm{\xi_p}_{\ds}
\nonumber\\
\label{ineq:3}
&\leq
C\sq^{1/4}N^{-k}(\ln N)^{k+1}\norm{\xi_p}_{\ds},
\end{align}
where $N^{-1/2}(\ln N)^{1/2}\leq C$ was used for $N\geq 4$. Similarly, 
\begin{align}\label{ineq:4}
\left|\sum_{j=1}^{N/4}\dual{(u-\Pi^{\pm}_{xy}u)^+_{3N/4,y}}{(\xi_p)^-_{3N/4+1,y}}_{J_j}\right|
\leq C\sq^{1/4}N^{-k}(\ln N)^{k+1}\norm{\xi_p}_{\ds}.
\end{align}
Combining \eqref{ineq:1}--\eqref{ineq:4} yields
\begin{align}
|T_{44}|\leq C\sq^{1/4}N^{-k}(\ln N)^{k+1}\norm{\xi_p}_{\ds}.\label{T44}\nonumber
\end{align}

Similarly, 
\begin{align*}
|T_{46}|&\leq C\sq^{1/4}N^{-k}(\ln N)^{k+1}\norm{\xi_p}_{\dn}.
\end{align*}

Appealing first to the definition \eqref{2d:projectors} of the $L^2$-projector $\Pi$ and the Gauss-Radau projectors $\Pi^{-}_{x}$ and $\Pi^{+}_{x}$, 
we use \eqref{sup:L2:2d:u:x} to get
\begin{align*}
|T_{45}|=\left|\sum_{j=N/4+1}^{3N/4}\sum_{i=N/4+1}^{3N/4-1}
\dual{\average{u-\Pi u}_{i,y}}{\jump{\xi_p}_{i,y}}_{J_j}\right|
\leq C N^{-\hat{k}} \norm{\xi_p}.
\end{align*}

Gathering the estimates of the $T_{4i}$ for $i=1,2,\dots,9$, 
the lemma follows from
$\norm{\xi_p}\leq \sq^{1/2}\enorm{\bm\xi}$.
\end{proof}

\begin{remark}\label{rmk:thm}
When analysing numerical methods for solving \eqref{spp:R-D:2d}, many authors 
(e.g., \cite{Armentano2023,LiNavon1998,Liu2009}) assume that $\sq^{1/2}\leq N^{-1}$ 
since this is usually true in practice.
With this assumption, the balanced-norm bound of Theorem~\ref{thm:2d} simplifies to
\begin{align*}
&\ba{\bm w -\wph}
\leq C
\begin{cases}
(N^{-1}\ln N)^{k+1} &\text{for $k$ even};
\\
(N^{-1}\ln N)^{k+1} &\text{for $k$ odd and }
\\ 
&\text{$\sq^{1/4}\leq N^{-1}(\ln N)^{k+1}$};
\\
(N^{-1}\ln N)^{k+1}+N^{-(k+1/2)}
&\text{for $k$ odd and}
\\
&\text{$\sq^{1/4}\geq N^{-1}(\ln N)^{k+1}$}.
\end{cases}
\end{align*}  
That is, our LDG method achieves  $O((N^{-1}\ln N)^{k+1})$ convergence in the balanced norm
except when $k$ is odd and $\sq^{1/4}\geq N^{-1}(\ln N)^{k+1}$.
But this inequality, combined with our earlier mild assumption~\eqref{mild:condition:epsilon}, 
says that $N^{-1}(\ln N)^{k+1}\leq \sq^{1/4}\leq \beta^{1/2}(4\sigma\ln N)^{-1/2}$.
In Table~\ref{table:N-ep-N}, which lists some typical values of these quantities, 
only for $k=1$ are there entries (in bold font) with 
$N^{-1}(\ln N)^{k+1}\leq \beta^{1/2}(4\sigma\ln N)^{-1/2}$. 
Similarly, the errors and observed rates of convergence in Table~\ref{table:ex1:balanced}
depart from $O((N^{-1}\ln N)^{k+1})$ convergence only in two entries (in bold font);
in these entries $\sq = 10^{-4}$ and $N= 256, 512$ so one has $N^{-1} < \sq^{1/2}$,
which is generally regarded as not lying in  the singularly perturbed regime when solving 
\eqref{spp:R-D:2d} numerically. In other words, in practice 
the exception mentioned above is unlikely to occur and
one usually expects to see a 
balanced-norm $O((N^{-1}\ln N)^{k+1})$ convergence rate 
when using our LDG method to solve singularly perturbed reaction-diffusion problems. 
\end{remark}

\begin{table}[ht]
	\setlength{\tabcolsep}{2.5pt}
	\footnotesize
	\centering
	\caption{Two quantities $N^{-1}(\ln N)^{k+1}$ and $\beta^{1/2}(4\sigma\ln N)^{-1/2}$, where $\beta=1, \sigma=k+1$.}
	\label{table:N-ep-N}
	\begin{tabular}{ccccc ccccc ccc}
		\toprule
		& &\multicolumn{3}{c}{$k=1$} & &\multicolumn{3}{c}{$k=3$} & & \multicolumn{3}{c}{$k=5$} \\
		\cmidrule(r){3-6} \cmidrule(r){7-10}  \cmidrule(r){11-13} 
		$N$ & & $N^{-1}(\ln N)^2$ & & $(8 \ln N)^{-1/2}$ & & $N^{-1}(\ln N)^4$ & & $(16 \ln N)^{-1/2}$ & & $N^{-1}(\ln N)^6$ & & $(24 \ln N)^{-1/2}$ \\
		\midrule
		8&  & 5.4051E-01 &  & 2.4518E-01 &  & 2.3372E+00 &  & 1.7337E-01 &  & 1.0106E+01 &  & 1.4155E-01 \\
		16&  & 4.8045E-01 &  & 2.1233E-01 &  & 3.6934E+00 &  & 1.5014E-01 &  & 2.8392E+01 &  & 1.2259E-01 \\
		32&  & 3.7535E-01 &  & 1.8991E-01 &  & 4.5085E+00 &  & 1.3429E-01 &  & 5.4153E+01 &  & 1.0965E-01 \\
		64&  & 2.7025E-01 &  & 1.7337E-01 &  & 4.6744E+00 &  & 1.2259E-01 &  & 8.0850E+01 &  & 1.0009E-01 \\
		128&  & 1.8392E-01 &  & 1.6051E-01 &  & 4.3300E+00 &  & 1.1350E-01 &  & 1.0194E+02 &  & 9.2669E-02 \\
		256&  & \bf{1.2011E-01} &  & \bf{1.5014E-01} &  & 3.6934E+00 &  & 1.0617E-01 &  & 1.1357E+02 &  & 8.6684E-02 \\
		512&  & \bf{7.6009E-02} &  & \bf{1.4155E-01} &  & 2.9580E+00 &  & 1.0009E-01 &  & 1.1512E+02 &  & 8.1726E-02
		\\
		\bottomrule
	\end{tabular}
\end{table}

\begin{remark}\label{rmk:energy:error}
With the commonly-used assumption $\sq^{1/2}\leq N^{-1}$, 
the energy-norm bound of Theorem~\ref{thm:2d} reduces to
\begin{align*}
\enorm{\bm w -\wph}
& \leq C\Big[\sq^{1/4}(N^{-1}\ln N)^{k+1}+N^{-(k+1)}\Big].
\end{align*}
If $k$ and $\sq$ are so small that $\sq^{1/4}(\ln N)^{k+1}\ll 1$,
we would expect a convergence rate of order $O(N^{-(k+1)})$
for the energy-norm error; otherwise, the error is likely to converge
at a rate of $O(\sq^{1/4}(N^{-1}\ln N)^{k+1})$.
These observations concur with our numerical results in Section~\ref{section:numerical:experiments}.
\end{remark}

\begin{remark}
In $T_6$ no integral terms arise from the domain boundary, which allows us to dispense with 
penalty terms in the definition~\eqref{flux:diffusion:2d} of the numerical flux, unlike
traditional LDG methods
\cite{Castillo2002,ChengJiangStynes2023,CYM2022,Zhu:2dMC}.
\end{remark}

\begin{remark}
	Our LDG method does not employ any penalty terms.
	If one uses a numerical flux with penalty in the LDG equations~(\ref{LDG:scheme:2d}), e.g., 
	$\widehat{\pph}^{\star}_{i,y}=\widehat{\pph}_{i,y}+\lambda_i\jump{\uph}_{i,y}$
	and $\widehat{\qph}^{\star}_{x,j}=\widehat{\qph}_{x,j}+\mu_j\jump{\uph}_{x,j}$
	for $i,j=0,1,\dots,N$,
	where  $\widehat{\pph}_{i,y}$ and $\widehat{\qph}_{x,j}$
	are defined in (\ref{flux:diffusion:2d}) and the penalty parameters $\lambda_i=\mu_j=\sq^{1/2}$ for $i,j=0,1,\dots,N$,
	then modify our error analysis accordingly,  one obtains an optimal-order
	error estimate in the corresponding energy and balanced norms.
	But the convergence rate is heavily influenced by the error from the boundary penalty terms; 
	consequently, the theoretical result for the $H^1$-error
	is only a suboptimal convergence rate  $O((N^{-1}\ln N)^{k+1/2})$, 
	which is a half-order inferior to the numerical results. 
\end{remark}

\begin{remark}
	In our analysis we considered the standard Shishkin mesh, 
	whose coarse and fine components are each piecewise uniform.
	If instead a randomly perturbed quasi-uniform mesh is used in $\dreg$, 
	then a slightly worse bound can be derived in Lemmas 4.1 and 4.3,
	and finally the term $\sq^{1/4} N^{-\hat{k}}$ that appears in 
	the balanced-norm error bound will be changed to  $\sq^{1/4} N^{-k}$.
	Because of the positive power of the small parameter,
	the final convergence rate is only slightly affected.
	Furthermore, our analysis  
	is applicable to other layer-adapted meshes such as 
	Bakhvalov-type meshes which are not locally uniform inside the layer regions,
	but the convergence result will be significantly different.
	The extension of our methodology to other types of layer-adapted meshes deserves further investigation 
	and we shall explore this in future work, as we state in Section~\ref{sec:concluding:remarks}.
\end{remark}

\begin{remark}\label{remark:b:nonconstant}[non-constant $b$] 
If $b\geq 2\beta^2>0$ is non-constant 
and $b\in W^{1,\infty}(\bar{\Omega})$, 
then Theorem~\ref{thm:2d} is still valid: one alters
the error analysis by using a modification $\prou^{\star} $ 
of the projector $\prou$ in $\dreg$:
\[ 
\prou^{\star} u=\begin{cases}
\prou u             & \text{in } \Omega\backslash\dreg,\\
\Pi_b u            & \text{in } \dreg,
\end{cases}
\]
	where for each $z\in L^2(\Omega)$,
	the weighted $L^2$ projection $\Pi_b z\in \spc$ is defined by 
	\begin{align}\label{weighted:L2}
	\dual{b\Pi_b z}{\vphu}_{K_{ij}}= \dual{bz}{\vphu}_{K_{ij}}
	\quad \forall \vphu\in \mathcal{Q}^{k}(K_{ij}) \quad 
	\forall K_{ij}\in\Omega_N.
	\end{align}
Then the bounds on the terms in~\eqref{energy:equation:2d} remain unchanged except for $T_4$ and $T_5$.
In the subdivision of~$T_4$, we only have a different expression
for the former two terms included in  $T_{45}$, denoted by
\[
T_{45}^1:= \dual{u-\Pi_b u}{(\xi_p)_x}_{\dreg}
\text{ and }
T_{45}^2:=\sum_{j=N/4+1}^{3N/4}\sum_{i=N/4+1}^{3N/4-1}\dual{\average{u-\Pi_b u}_{i,y}}{\jump{\xi_p}_{i,y}}_{J_j}.
\]
We now show how to bound $T_{45}^1$ and $T_{45}^2$ satisfactorily.
	
For $T_{45}^1$,
on each $K_{ij}$ set $b_{\mathrm{avg}}|_{K_{ij}}=(h_ih_j)^{-1}\int_{K_{ij}} b(x,y)\,dx\,dy$,
so $b_{\mathrm{avg}}$~is a piecewise constant approximation of~$b$.
The definition \eqref{weighted:L2}, a Cauchy-Schwarz inequality and an inverse inequality
give us
\begin{align}
\left|T_{45}^1\right| &= \left|\dual{u-\Pi_b u}{(\xi_p)_x}_{\dreg}\right|
= \left|\sum_{K_{ij}\in \dreg}\dual{b_{\mathrm{avg}}^{-1}(b_{\mathrm{avg}}-b)
	(u-\Pi_b u)}{(\xi_p)_x}_{K_{ij}}\right|
\nonumber\\
&\leq C\sum_{K_{ij}\in \dreg}\norm{b_{\mathrm{avg}}-b}_{L^{\infty}(K_{ij})}
\norm{u-\Pi_b u}_{K_{ij}}H^{-1}\norm{\xi_p}_{K_{ij}}
\nonumber\\
\label{T45:1}
&\leq C \norm{b}_{W^{1,\infty}(\Omega)} \norm{u-\Pi_b u}_{\dreg}\norm{\xi_p}_{\dreg}
\leq CN^{-(k+1)}\norm{\xi_p}_{\dreg},
\end{align}
where we used an approximation property for the projector $\Pi_b$
that is an analogue of~\eqref{bound:eu:L2:2d}.

For $T_{45}^2$, we shall use the supercloseness result
\begin{equation}\label{supclose}
\norm{\Pi_b z-\Pi z}_{K_{ij}}\leq CN^{-(k+2)}\norm{z}_{H^{k+1}(K_{ij})}
\quad \forall K_{ij}\in\dreg.
\end{equation}
To verify \eqref{supclose}, use the definitions \eqref{weighted:L2}
and \eqref{L2:prj:2d} to obtain
\[
\dual{b(\Pi_b z-\Pi z)}{\vphu}_{K_{ij}}=\dual{b(z-\Pi z)}{\vphu}_{K_{ij}}
=\dual{(b-b_{\mathrm{avg}})(z-\Pi z)}{\vphu}_{K_{ij}}.
\]
Take $\vphu=\Pi_b z-\Pi z$ and then a Cauchy-Schwarz inequality yields
\begin{align*}
\norm{\Pi_b z-\Pi z}_{K_{ij}}^2&\leq (2\beta^2)^{-1}
\dual{b(\Pi_b z-\Pi z)}{\Pi_b z-\Pi z}_{K_{ij}}
\\
&=(2\beta^2)^{-1}
\dual{(b-b_{\mathrm{avg}})(z-\Pi z)}{\Pi_b z-\Pi z}_{K_{ij}}
\\
&\leq C\norm{b_{\mathrm{avg}}-b}_{L^{\infty}(K_{ij})}
\norm{z-\Pi z}_{K_{ij}}\norm{\Pi_b z-\Pi z}_{K_{ij}}
\\
&\leq C \norm{b}_{W^{1,\infty}(\Omega)} N^{-(k+2)}\norm{z}_{H^{k+1}(K_{ij})}
\norm{\Pi_b z-\Pi z}_{K_{ij}}
\end{align*}
for $K_{ij}\in \dreg$,
which leads to \eqref{supclose}.
Now an inverse inequality gives
\begin{align}\label{supclose:L0}
\norm{\Pi_b z-\Pi z}_{L^{\infty}(K_{ij})}
\leq CN\norm{\Pi_b z-\Pi z}_{K_{ij}}
\leq CN^{-(k+1)}\norm{z}_{H^{k+1}(K_{ij})},\quad
K_{ij}\in \dreg.
\end{align}
For $K_{ij}, K_{i+1,j}\in \dreg$	and the smooth component $\bar{u}$,
by a Cauchy-Schwarz inequality, an inverse inequality, \eqref{supclose:L0}
and \eqref{reg:smooth}  one has
\begin{align*}
&\left|\sum_{j=N/4+1}^{3N/4}\sum_{i=N/4+1}^{3N/4-1}
\dual{\average{\Pi_b \bar{u}-\Pi \bar{u}}_{i,y}}{\jump{\xi_p}_{i,y}}_{J_j}
\right|\nonumber\\
&\leq
C\sum_{j=N/4+1}^{3N/4}\sum_{i=N/4+1}^{3N/4-1}N^{-1/2}\norm{\Pi_b \bar{u}-\Pi \bar{u}}_{L^{\infty}(K_{ij}\cup K_{i+1,j} )}
N^{1/2}\norm{\xi_p}_{K_{ij}\cup K_{i+1,j}}
\nonumber\\
&\leq CN^{-(k+1)}\sum_{j=N/4+1}^{3N/4}\sum_{i=N/4+1}^{3N/4-1}
\norm{\bar{u}}_{H^{k+1}(K_{ij}\cup K_{i+1,j})}\norm{\xi_p}_{K_{ij}\cup K_{i+1,j}}
\leq
CN^{-k+1}\norm{\xi_p}_{\dreg}.
\end{align*}
Combining this inequality with \eqref{avg:regular}, 
and observing that \eqref{avg:layer} will hold when $\Pi$ is replaced by~$\Pi_b$, we obtain  
\begin{align}\label{T45:2}
\left|T_{45}^2\right|=
\left|\sum_{j=N/4+1}^{3N/4}\sum_{i=N/4+1}^{3N/4-1}
\dual{\average{u-\Pi_b u}_{i,y}}{\jump{\xi_p}_{i,y}}_{J_j}
\right|\leq
CN^{-\hat{k}}\norm{\xi_p}_{\dreg}.
\end{align}	
Adding \eqref{T45:1}, \eqref{T45:2} 
and the existing bounds for the remaining terms of $T_{45}$
gives 
$T_{45}\leq CN^{-\hat{k}}\norm{\xi_p}_{\dreg}$
and thus the same bound for $T_4$. 
We can bound $T_5$ in a similar way. 
Thus, we again obtain the error bound of Theorem~\ref{thm:2d}.
\end{remark}

\section{Numerical experiments}
\label{section:numerical:experiments}

In this section, we present numerical results for the LDG method
on the Shishkin mesh \eqref{mesh:point} for two singularly perturbed 
reaction-diffusion problems---one with a known solution and one with an unknown solution. 
All the calculations were carried out in MATLAB R2017a. 
The resulting discrete system of the FEM is solved directly using the Matlab backslash operator 
after diagonal preconditioning of the linear system.  All integrals were evaluated using the $5\times5$ Gauss-Legendre quadrature rule.

Convergence rates in the energy
and balanced norms are computed by the formula
\[
r_S :=
	\frac{ \ln(\bm e_N/\bm e_{2N})}{\ln (2\ln N/\ln (2N)) },
\]
where $\bm e_N$ is the observed error when $N$ elements are used.
This quantity $r_S$ measures the convergence rate in the form $O((N^{-1}\ln N)^{r_S})$,
which is natural for convergence on Shishkin meshes.

\begin{example}\label{exa:1}
	Consider the linear constant-coefficient problem 
\[
-\sq\Delta u+2u=f  \text{ in } \Omega=(0,1)\times(0,1),
\quad 	u =0   \text{ on }\partial\Omega,
\]	
with $f$ chosen such that the exact solution is
	\[
	u(x,y)=\left[
	\frac{e^{-x/\sqrt\sq}-e^{-(1-x)/\sqrt\sq}}{1-e^{-1/\sqrt\varepsilon}}-\cos(\pi x)\right]
	\left[
	\frac{e^{-y/\sqrt\sq}-e^{-(1-y)/\sqrt\sq}}{1-e^{-1/\sqrt\varepsilon}}-\cos(\pi y)\right].
	\]
For any nonnegative integer $m$, this exact solution has precisely the layer behaviour that one assumes in  Lemma~\ref{lemma:prop:2d}.
\end{example}

Fix $k=2$ and $N=64$. Two numerical solutions $\uph$ and their pointwise errors $|u-\uph|$ 
are displayed in Figure \ref{ex1:Fig:nmr:err} for $\sq=10^{-4}$ and $\sq=10^{-8}$.
One can see that the LDG method produces an accurate solution with 
sharp boundary layers and no oscillations near these layers, 
demonstrating the ability of the LDG method to capture boundary layers.

In Tables~\ref{table:ex1:energy} and~\ref{table:ex1:balanced}
we present the values of 
$\enorm{\bm w -\wph}$ and $\ba{\bm w -\wph}$
and their convergence rates.
The convergence rate of the energy-norm error is
$O(\sq^{1/4}(N^{-1}\ln N)^{k+1})$ in most cases, except for the data in boldface font, where $\sq^{1/4}(\ln N)^{k+1}<1$ and the convergence rate is $O(N^{-(k+1)})$. 
This agrees with the predictions of
Theorem~\ref{thm:2d} and Remark~\ref{rmk:energy:error}.
On the other hand, the balanced-norm errors attain the optimal order 
$O((N^{-1}\ln N)^{k+1})$ in all cases, which agrees with 
the \emph{expected orders of convergence} (EOC in our tables)
that are predicted by Theorem~\ref{thm:2d}.

\begin{example}\label{exa:2}
	Consider the variable-coefficient problem 
	\begin{equation*}
	\begin{aligned}
	-\sq\Delta u+\big(1+x^2y^2 e^{xy/2}\big)u&=\tanh\big[(x+1)(y+1)\big]
	\ \text{ in } \Omega=(0,1)\times(0,1),\\
	u &=0 \ \text{ on }\partial\Omega,
	\end{aligned}
	\end{equation*}
	whose exact solution is unknown.
\end{example}

Fix $k=2$ and $N=64$. In Figure \ref{ex2:Fig:nmr:err} we plot
two numerical solutions  computed by the LDG method on the Shishkin mesh. 
We use $|\uph_{64}-\widetilde{\uph}_{128}|$ to represent
its pointwise error, where $\widetilde{\uph}_{128}$ is the numerical solution computed by the LDG method using the two-mesh principle \cite{Farrell2000}. To be specific,
if $\uph_{N}$ is the computed solution on the Shishkin mesh $\{(x_i,y_j):i,j=0, 1, ..., N\}$, then
 $\widetilde{\uph}_{2N}$ is computed on the mesh $\{(x_{i/2},y_{j/2}):i,j=0,1,...,2N\}$,
where $x_{i+1/2}:=(x_{i}+x_{i+1})/2$ and $y_{j+1/2}:=(y_{j}+y_{j+1})/2$. 
Again we see that the LDG method produces solutions
with no visible oscillations in the solution.

Tables \ref{table:ex2:energy} and \ref{table:ex2:balanced} display 
the energy-norm errors $\enorm{\wph_{N}-\widetilde{\wph}_{2N}}$, 
balanced-norm errors $\ba{\wph_N-\widetilde{\wph}_{2N}}$
and their convergence rates. 
One can see that these convergence rates 
are, respectively, $O(\sq^{1/4}(N^{-1}\ln N)^{k+1}+N^{-(k+1)})$
and $O((N^{-1}\ln N)^{k+1})$ in general, 
once again agreeing with Theorem~\ref{thm:2d}.

\begin{table}[h]
	\setlength{\tabcolsep}{3pt}
	\footnotesize
	\centering
	\caption{Energy-norm errors and convergence rates in Example \ref{exa:1}.}
	\label{table:ex1:energy}
	\begin{tabular}{ccccc ccccc}
		\toprule
		& & \multicolumn{2}{c}{$k=0$} & \multicolumn{2}{c}{$k=1$} & \multicolumn{2}{c}{$k=2$}  & \multicolumn{2}{c}{$k=3$}\\
		\cmidrule(r){3-4} \cmidrule(r){5-6} \cmidrule(r){7-8} \cmidrule(r){9-10}
		$\sq$ & $N$ & Error & $r_S$ & Error& $r_S$  & Error & $r_S$ & Error & $r_S$\\
		\midrule
		$10^{-4}$  & 8   & 2.1412E-01 & ---      & 3.5465E-02 & ---      & 1.3681E-02 & ---      & 5.9709E-03 & ---      \\
		& 16  & 1.1264E-01 & \bf{1.5842} & 1.6812E-02 & 1.8410 & 5.3725E-03 & 2.3053 & 1.8072E-03 & 2.9475 \\
		& 32  & 5.9400E-02 & \bf{1.3615} & 7.1698E-03 & 1.8132 & 1.5639E-03 & 2.6258 & 3.6292E-04 & 3.4156 \\
		& 64  & 3.1519E-02 & \bf{1.2405} & 2.7124E-03 & 1.9029 & 3.6785E-04 & 2.8331 & 5.3918E-05 & 3.7326 \\
		& 128 & 1.6788E-02 & \bf{1.1687} & 9.4905E-04 & 1.9483 & 7.5658E-05 & 2.9340 & 6.6092E-06 & 3.8943 \\
		& 256 & 8.9549E-03 & \bf{1.1230} & 3.1788E-04 & 1.9545 & 1.4313E-05 & 2.9753 & 7.2009E-07 & 3.9614 \\
		& 512 & 4.7758E-03 & \bf{1.0926} & 1.0470E-04 & 1.9302 & 2.5610E-06 & 2.9908 & 7.2665E-08 & 3.9862 \\
		$10^{-8}$  & 8   & 2.2152E-01 & ---      & 2.2936E-02 & ---      & 2.0605E-03 & ---      & 6.1865E-04 & ---      \\
		& 16  & 1.1271E-01 & \bf{1.6663} & 5.9520E-03 & 3.3270 & 5.8386E-04 & 3.1101 & 1.8587E-04 & 2.9658 \\
		& 32  & 5.6630E-02 & \bf{1.4645} & 1.6032E-03 & 2.7909 & 1.6247E-04 & 2.7216 & 3.7302E-05 & 3.4170 \\
		& 64  & 2.8366E-02 & \bf{1.3534} & 4.5073E-04 & 2.4840 & 3.7915E-05 & 2.8486 & 5.5404E-06 & 3.7331 \\
		& 128 & 1.4199E-02 & \bf{1.2839} & 1.3074E-04 & 2.2963 & 7.7822E-06 & 2.9379 & 6.7905E-07 & 3.8945 \\
		& 256 & 7.1071E-03 & \bf{1.2367} & 3.8578E-05 & 2.1810 & 1.4712E-06 & 2.9766 & 7.3980E-08 & 3.9615 \\
		& 512 & 3.5577E-03 & \bf{1.2027} & 1.1459E-05 & 2.1098 & 2.6317E-07 & 2.9913 & 7.4651E-09 & 3.9863 \\
		$10^{-12}$ & 8   & 2.2156E-01 & ---      & 2.2785E-02 & ---      & 1.5212E-03 & ---      & 9.6995E-05 & ---      \\
		& 16  & 1.1271E-01 & \bf{1.6670} & 5.7367E-03 & \bf{3.4016} & 1.9824E-04 & 5.0258 & 1.9174E-05 & 3.9982 \\
		& 32  & 5.6600E-02 & \bf{1.4655} & 1.4378E-03 & \bf{2.9442} & 2.8747E-05 & 4.1083 & 3.7427E-06 & 3.4760 \\
		& 64  & 2.8331E-02 & \bf{1.3548} & 3.6018E-04 & \bf{2.7098} & 4.8140E-06 & 3.4983 & 5.5449E-07 & 3.7381 \\
		& 128 & 1.4169E-02 & \bf{1.2855} & 9.0297E-05 & \bf{2.5668} & 8.6217E-07 & 3.1908 & 6.7933E-08 & 3.8953 \\
		& 256 & 7.0852E-03 & \bf{1.2385} & 2.2667E-05 & \bf{2.4699} & 1.5428E-07 & 3.0747 & 7.4003E-09 & 3.9617 \\
		& 512 & 3.5427E-03 & \bf{1.2047} & 5.7002E-06 & \bf{2.3992} & 2.6953E-08 & 3.0323 & 7.4673E-10 & 3.9863 \\
		\bottomrule
	\end{tabular}
\end{table}

\begin{table}[h]
	\setlength{\tabcolsep}{3pt}
	\footnotesize
	\centering
	\caption{Balanced-norm errors and convergence rates in Example \ref{exa:1}.}
	\label{table:ex1:balanced}
	\begin{tabular}{ccccc ccccc}
		\toprule
		& & \multicolumn{2}{c}{$k=0$} & \multicolumn{2}{c}{$k=1$} & \multicolumn{2}{c}{$k=2$}  & \multicolumn{2}{c}{$k=3$}\\
		\cmidrule(r){3-4} \cmidrule(r){5-6} \cmidrule(r){7-8} \cmidrule(r){9-10}
		$\sq$ & $N$ & Error & $r_S$ & Error & $r_S$  & Error & $r_S$ & Error & $r_S$\\
		\midrule
		$10^{-4}$  & 8   & 4.4334E-01 & ---      & 1.4461E-01 & ---      & 7.1644E-02 & ---      & 3.3035E-02 & ---      \\
		& 16  & 2.7238E-01 & 1.2014 & 7.8450E-02 & 1.5083 & 2.9108E-02 & 2.2214 & 1.0091E-02 & 2.9247 \\
		& 32  & 1.6140E-01 & 1.1134 & 3.6792E-02 & 1.6110 & 8.6937E-03 & 2.5711 & 2.0432E-03 & 3.3982 \\
		& 64  & 9.3081E-02 & 1.0775 & 1.4964E-02 & 1.7611 & 2.0813E-03 & 2.7986 & 3.0693E-04 & 3.7110 \\
		& 128 & 5.2783E-02 & 1.0525 & 5.5119E-03 & 1.8530 & 4.3383E-04 & 2.9093 & 3.8035E-05 & 3.8741 \\
		& 256 & 2.9599E-02 & 1.0336 & 1.9334E-03 & \bf{1.8720} & 8.2856E-05 & 2.9584 & 4.1775E-06 & 3.9470 \\
		& 512 & 1.6438E-02 & 1.0223 & 6.7488E-04 & \bf{1.8293} & 1.4915E-05 & 2.9802 & 4.2373E-07 & 3.9772 \\
		$10^{-8}$  & 8   & 4.3963E-01 & ---      & 1.4297E-01 & ---      & 7.1989E-02 & ---      & 3.3309E-02 & ---      \\
		& 16  & 2.6965E-01 & 1.2056 & 7.6999E-02 & 1.5263 & 2.9301E-02 & 2.2170 & 1.0176E-02 & 2.9245 \\
		& 32  & 1.6054E-01 & 1.1034 & 3.6106E-02 & 1.6113 & 8.7617E-03 & 2.5685 & 2.0607E-03 & 3.3979 \\
		& 64  & 9.2997E-02 & 1.0688 & 1.4598E-02 & 1.7728 & 2.0991E-03 & 2.7972 & 3.0967E-04 & 3.7102 \\
		& 128 & 5.2832E-02 & 1.0491 & 5.2793E-03 & 1.8869 & 4.3780E-04 & 2.9082 & 3.8390E-05 & 3.8733 \\
		& 256 & 2.9603E-02 & 1.0351 & 1.7792E-03 & 1.9435 & 8.3647E-05 & 2.9577 & 4.2179E-06 & 3.9464 \\
		& 512 & 1.6430E-02 & 1.0233 & 5.7439E-04 & 1.9651 & 1.5062E-05 & 2.9798 & 4.2792E-07 & 3.9769 \\
		$10^{-12}$ & 8   & 4.3967E-01 & ---      & 1.4296E-01 & ---      & 7.1992E-02 & ---      & 3.3311E-02 & ---      \\
		& 16  & 2.6966E-01 & 1.2056 & 7.6974E-02 & 1.5269 & 2.9303E-02 & 2.2169 & 1.0177E-02 & 2.9245 \\
		& 32  & 1.6053E-01 & 1.1036 & 3.6090E-02 & 1.6116 & 8.7624E-03 & 2.5685 & 2.0608E-03 & 3.3979 \\
		& 64  & 9.2983E-02 & 1.0690 & 1.4586E-02 & 1.7734 & 2.0993E-03 & 2.7972 & 3.0970E-04 & 3.7102 \\
		& 128 & 5.2818E-02 & 1.0493 & 5.2710E-03 & 1.8885 & 4.3784E-04 & 2.9082 & 3.8394E-05 & 3.8733 \\
		& 256 & 2.9592E-02 & 1.0353 & 1.7727E-03 & 1.9472 & 8.3655E-05 & 2.9577 & 4.2183E-06 & 3.9464 \\
		& 512 & 1.6405E-02 & 1.0253 & 5.6920E-04 & 1.9745 & 1.5063E-05 & 2.9798 & 4.2795E-07 & 3.9769 \\
		$EOC$		&  &  & 1.0000 &  & 2.0000 &  & 3.0000 &  & 4.0000  \\
		\bottomrule
	\end{tabular}
\end{table}

\begin{table}[h]
	\setlength{\tabcolsep}{3pt}
	\footnotesize
	\centering
	\caption{Energy-norm errors and convergence rates in Example \ref{exa:2}.}
	\label{table:ex2:energy}
	\begin{tabular}{ccccc ccccc}
		\toprule
		& & \multicolumn{2}{c}{$k=0$} & \multicolumn{2}{c}{$k=1$} & \multicolumn{2}{c}{$k=2$}  & \multicolumn{2}{c}{$k=3$}\\
		\cmidrule(r){3-4} \cmidrule(r){5-6} \cmidrule(r){7-8} \cmidrule(r){9-10}
		$\sq$ & $N$ & Error & $r_S$ & Error& $r_S$  & Error & $r_S$ & Error & $r_S$\\
		\midrule
		$10^{-4}$  & 8   & 6.1983E-02 & ---      & 4.9279E-02 & ---      & 3.1381E-02 & ---      & 1.8802E-02 & ---      \\
		& 16  & 4.4270E-02 & \bf{0.8300} & 3.0021E-02 & 1.2223 & 1.4923E-02 & 1.8332 & 7.2120E-03 & 2.3633 \\
		& 32  & 2.9954E-02 & \bf{0.8312} & 1.4391E-02 & 1.5644 & 5.1216E-03 & 2.2754 & 1.8666E-03 & 2.8758 \\
		& 64  & 1.8972E-02 & \bf{0.8940} & 5.7360E-03 & 1.8008 & 1.3516E-03 & 2.6078 & 3.6289E-04 & 3.2062 \\
		& 128 & 1.1424E-02 & \bf{0.9412} & 2.0372E-03 & 1.9206 & 3.0390E-04 & 2.7688 & 7.2035E-05 & 2.9999 \\
		& 256 & 6.6434E-03 & \bf{0.9686} & 6.7682E-04 & 1.9691 & 6.5259E-05 & 2.7489 & ---        & ---    \\
		& 512 & 3.7725E-03 & \bf{0.9835} & 2.1586E-04 & 1.9862 & ---        & ---    & ---        & ---    \\
		$10^{-8}$  & 8   & 3.7461E-02 & ---      & 5.6853E-03 & ---      & 3.1490E-03 & ---      & 1.8769E-03 & ---      \\
		& 16  & 1.9127E-02 & \bf{1.6579} & 3.0916E-03 & 1.5025 & 1.4887E-03 & 1.8478 & 7.1270E-04 & 2.3881 \\
		& 32  & 9.7901E-03 & \bf{1.4249} & 1.4500E-03 & 1.6109 & 5.0607E-04 & 2.2956 & 1.7934E-04 & 2.9357 \\
		& 64  & 5.0337E-03 & \bf{1.3022} & 5.7349E-04 & 1.8158 & 1.3100E-04 & 2.6457 & 3.1447E-05 & 3.4082 \\
		& 128 & 2.5963E-03 & \bf{1.2283} & 2.0272E-04 & 1.9293 & 2.8168E-05 & 2.8516 & 4.2398E-06 & 3.7176 \\
		& 256 & 1.3418E-03 & \bf{1.1795} & 6.7093E-05 & 1.9759 & 5.4229E-06 & 2.9441 & ---        & ---    \\
		& 512 & 6.9434E-04 & \bf{1.1450} & 2.1322E-05 & 1.9923 & ---        & ---    & ---        & ---    \\
		$10^{-12}$ & 8   & 3.7135E-02 & ---      & 2.8491E-03 & ---      & 3.4516E-04 & ---      & 1.8782E-04 & ---      \\
		& 16  & 1.8706E-02 & \bf{1.6912} & 7.6520E-04 & 3.2423 & 1.4992E-04 & 2.0568 & 7.1263E-05 & 2.3901 \\
		& 32  & 9.3721E-03 & \bf{1.4704} & 2.2736E-04 & 2.5821 & 5.0649E-05 & 2.3088 & 1.7927E-05 & 2.9363 \\
		& 64  & 4.6896E-03 & \bf{1.3554} & 7.2153E-05 & 2.2468 & 1.3099E-05 & 2.6475 & 3.1395E-06 & 3.4106 \\
		& 128 & 2.3459E-03 & \bf{1.2851} & 2.3038E-05 & 2.1181 & 2.8147E-06 & 2.8528 & 4.1996E-07 & 3.7322 \\
		& 256 & 1.1735E-03 & \bf{1.2378} & 7.2454E-06 & 2.0671 & 5.4109E-07 & 2.9467 & ---        & ---    \\
		& 512 & 5.8702E-04 & \bf{1.2039} & 2.2390E-06 & 2.0410 & ---        & ---   & ---         & ---    \\
		\bottomrule
	\end{tabular}
\end{table}

\begin{table}[ht]
	\setlength{\tabcolsep}{3pt}
	\footnotesize
	\centering
	\caption{Balanced-norm errors and convergence rates in Example \ref{exa:2}.}
	\label{table:ex2:balanced}
	\begin{tabular}{ccccc ccccc}
		\toprule
		& & \multicolumn{2}{c}{$k=0$} & \multicolumn{2}{c}{$k=1$} & \multicolumn{2}{c}{$k=2$}  & \multicolumn{2}{c}{$k=3$}\\
		\cmidrule(r){3-4} \cmidrule(r){5-6} \cmidrule(r){7-8} \cmidrule(r){9-10}
		$\sq$ & $N$ & Error & $r_S$ & Error & $r_S$  & Error & $r_S$ & Error & $r_S$\\
		\midrule
		$10^{-4}$  & 8   & 4.0057E-01 & ---      & 3.5249E-01 & ---      & 2.2535E-01 & ---      & 1.3534E-01 & ---      \\
		& 16  & 2.9053E-01 & 0.7921 & 2.0805E-01 & 1.3004 & 1.0716E-01 & 1.8333 & 5.2129E-02 & 2.3529 \\
		& 32  & 1.9084E-01 & 0.8942 & 1.0007E-01 & 1.5572 & 3.7065E-02 & 2.2588 & 1.3720E-02 & 2.8401 \\
		& 64  & 1.1716E-01 & 0.9550 & 4.0531E-02 & 1.7693 & 9.9153E-03 & 2.5813 & 2.8465E-03 & 3.0789 \\
		& 128 & 6.9002E-02 & 0.9823 & 1.4572E-02 & 1.8980 & 2.3016E-03 & 2.7096 & 6.5048E-04 & 2.7387 \\
		& 256 & 3.9583E-02 & 0.9931 & 4.8755E-03 & 1.9564 & 5.2821E-04 & 2.6301 & ---        & ---    \\
		& 512 & 2.2301E-02 & 0.9972 & 1.5620E-03 & 1.9784 & ---        & ---    & ---        & ---    \\
		$10^{-8}$  & 8   & 4.0113E-01 & ---      & 3.5052E-01 & ---      & 2.2318E-01 & ---      & 1.3318E-01 & ---      \\
		& 16  & 2.8985E-01 & 0.8014 & 2.0592E-01 & 1.3119 & 1.0532E-01 & 1.8522 & 5.0477E-02 & 2.3928 \\
		& 32  & 1.8968E-01 & 0.9022 & 9.8594E-02 & 1.5670 & 3.5880E-02 & 2.2910 & 1.2675E-02 & 2.9402 \\
		& 64  & 1.1611E-01 & 0.9608 & 3.9745E-02 & 1.7786 & 9.2875E-03 & 2.6458 & 2.2217E-03 & 3.4089 \\
		& 128 & 6.8241E-02 & 0.9861 & 1.4214E-02 & 1.9077 & 1.9958E-03 & 2.8528 & 3.0185E-04 & 3.7033 \\
		& 256 & 3.9093E-02 & 0.9955 & 4.7292E-03 & 1.9665 & 3.8452E-04 & 2.9427 & ---        & ---    \\
		& 512 & 2.2006E-02 & 0.9987 & 1.5061E-03 & 1.9888 & ---        & ---    & ---        & ---    \\
		$10^{-12}$ & 8   & 4.0113E-01 & ---      & 3.5050E-01 & ---      & 2.2316E-01 & ---      & 1.3316E-01 & ---      \\
		& 16  & 2.8984E-01 & 0.8014 & 2.0590E-01 & 1.3120 & 1.0530E-01 & 1.8524 & 5.0460E-02 & 2.3932 \\
		& 32  & 1.8967E-01 & 0.9022 & 9.8579E-02 & 1.5671 & 3.5868E-02 & 2.2913 & 1.2664E-02 & 2.9413 \\
		& 64  & 1.1610E-01 & 0.9608 & 3.9737E-02 & 1.7787 & 9.2810E-03 & 2.6465 & 2.2145E-03 & 3.4136 \\
		& 128 & 6.8233E-02 & 0.9861 & 1.4211E-02 & 1.9078 & 1.9925E-03 & 2.8545 & 2.9630E-04 & 3.7318 \\
		& 256 & 3.9088E-02 & 0.9955 & 4.7277E-03 & 1.9666 & 3.8281E-04 & 2.9477 & ---        & ---    \\
		& 512 & 2.2003E-02 & 0.9987 & 1.5055E-03 & 1.9889 & ---        & ---    & ---        & ---    \\
		$EOC$		&  &  & 1.0000 &  & 2.0000 &  & 3.0000 &  & 4.0000  \\
		\bottomrule
	\end{tabular}
\end{table}

\begin{figure}[h]
	\begin{minipage}{0.49\linewidth}
		\vspace{2pt}
		\centerline{\includegraphics[width=1.0\textwidth]{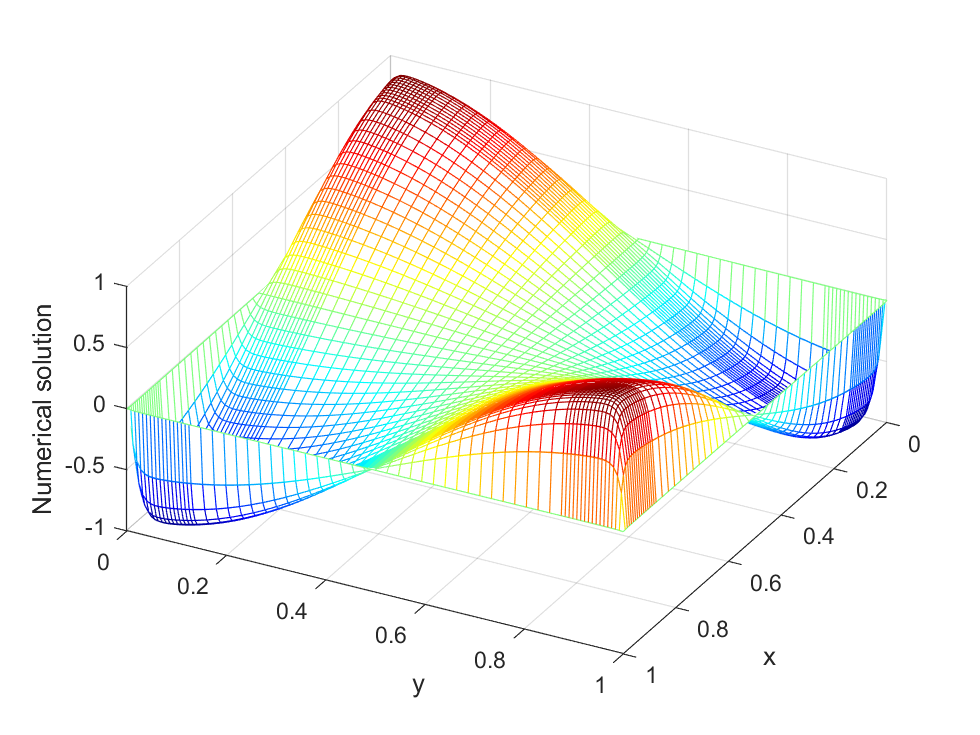}}
		\centerline{Numerical solution $\uph_{64}$ for $\sq=10^{-4}$}
		\vspace{2pt}
		\centerline{\includegraphics[width=1.0\textwidth]{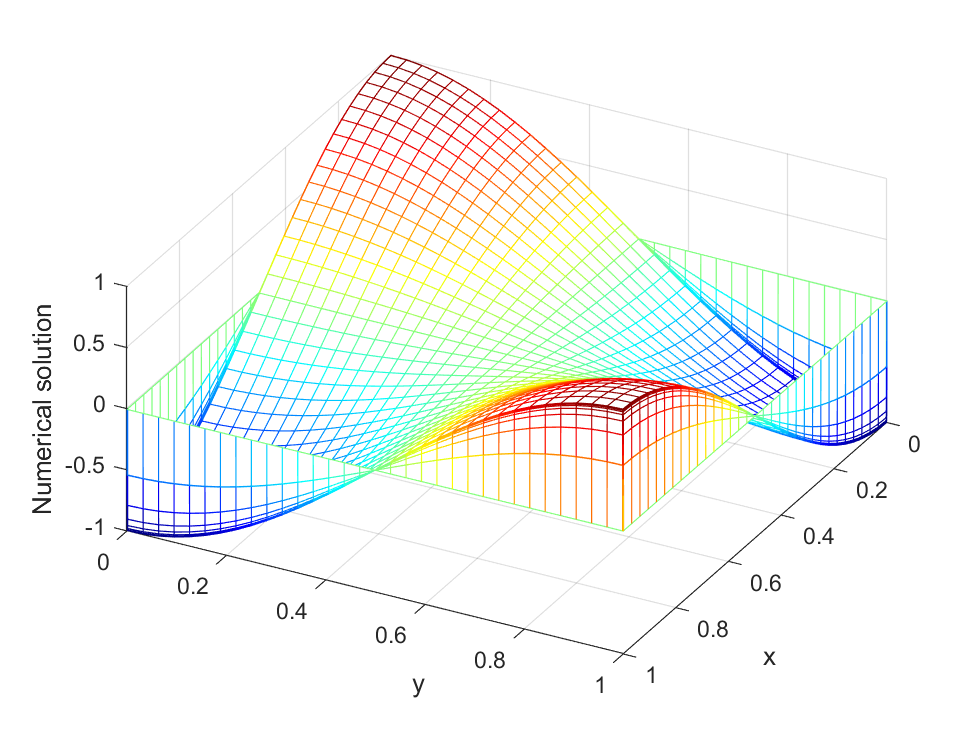}}
		\centerline{Numerical solution $\uph_{64}$ for $\sq=10^{-8}$}
	\end{minipage}
	\begin{minipage}{0.49\linewidth}
		\vspace{2pt}
		\centerline{\includegraphics[width=1.0\textwidth]{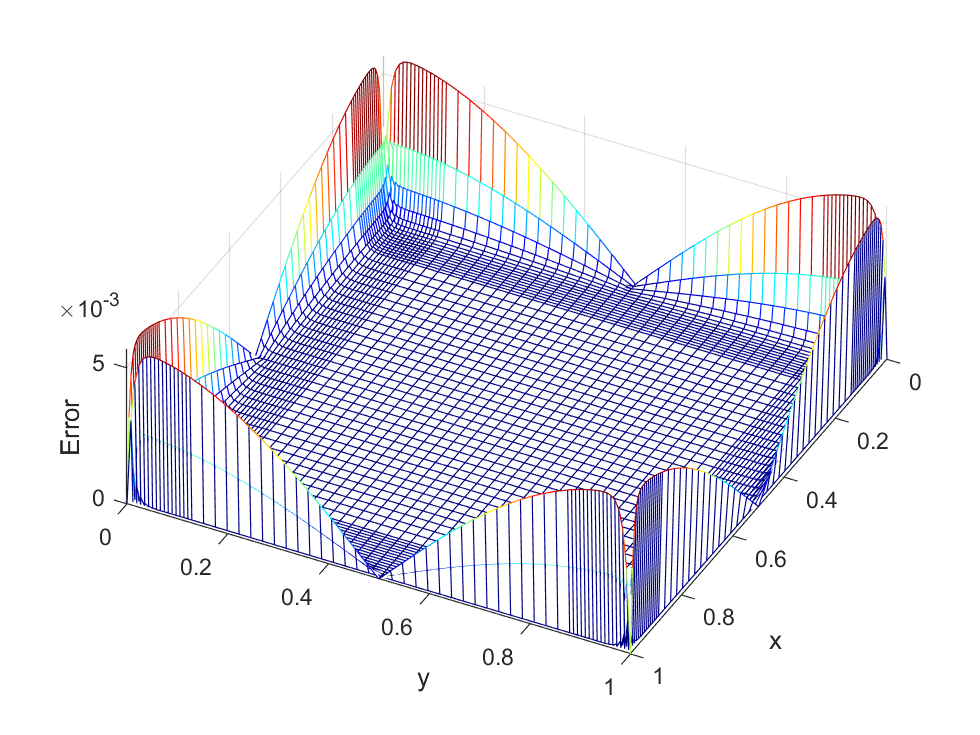}}
		\centerline{Pointwise error $|u-\uph_{64}|$ for $\sq=10^{-4}$}
		\vspace{2pt}
		\centerline{\includegraphics[width=1.0\textwidth]{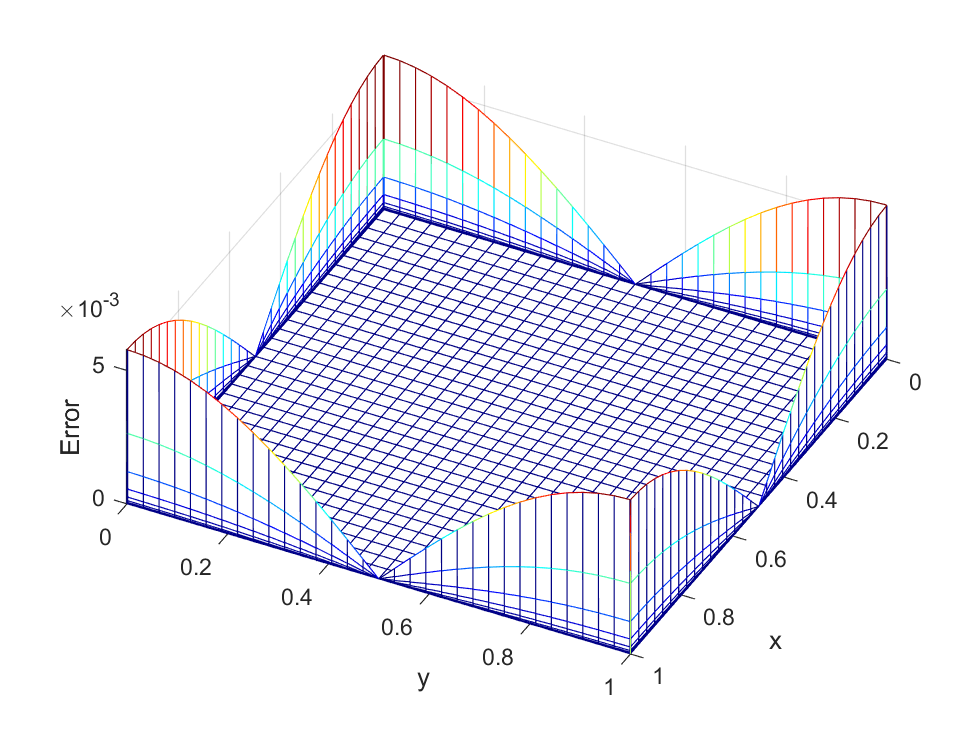}}
		\centerline{Pointwise error $|u-\uph_{64}|$ for $\sq=10^{-8}$}
	\end{minipage}
	\caption{Numerical solution $\uph_{64}$ and pointwise error $|u-\uph_{64}|$ in Example \ref{exa:1}.
	Here $k=2$.}
	\label{ex1:Fig:nmr:err}
\end{figure}

\begin{figure}[h]
	\begin{minipage}{0.49\linewidth}
		\vspace{3pt}
		\centerline{\includegraphics[width=1.0\textwidth]{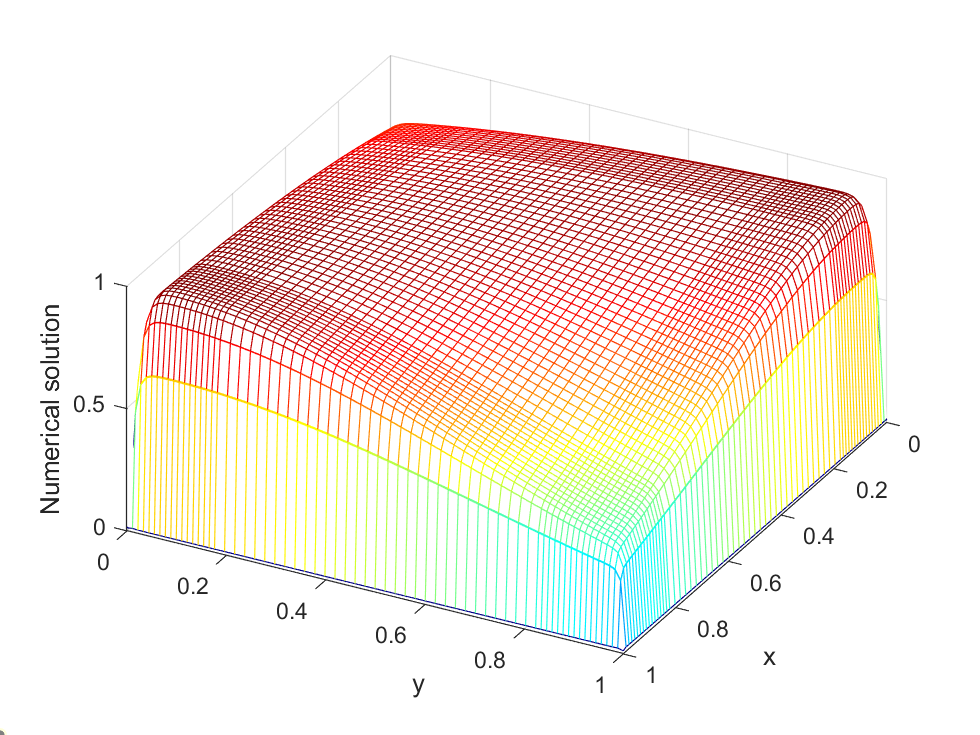}}
		\centerline{Numerical solution $\uph_{64}$ for $\sq=10^{-4}$}
		\vspace{3pt}
		\centerline{\includegraphics[width=1.0\textwidth]{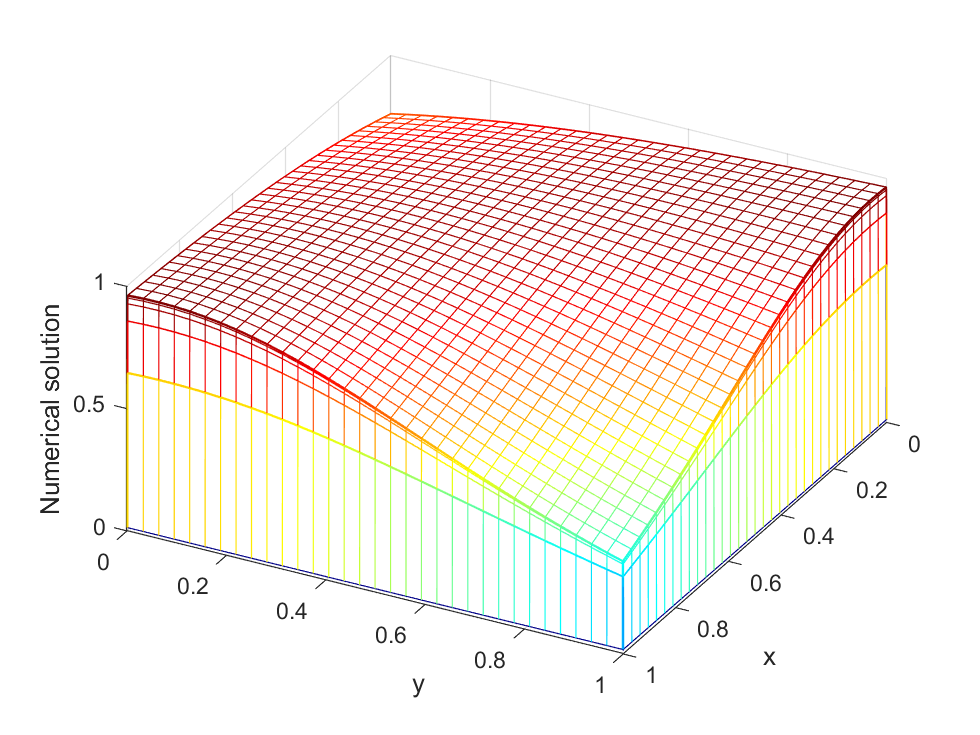}}
		\centerline{Numerical solution $\uph_{64}$ for $\sq=10^{-8}$}
	\end{minipage}
	\begin{minipage}{0.49\linewidth}
		\vspace{3pt}
		\centerline{\includegraphics[width=1.0\textwidth]{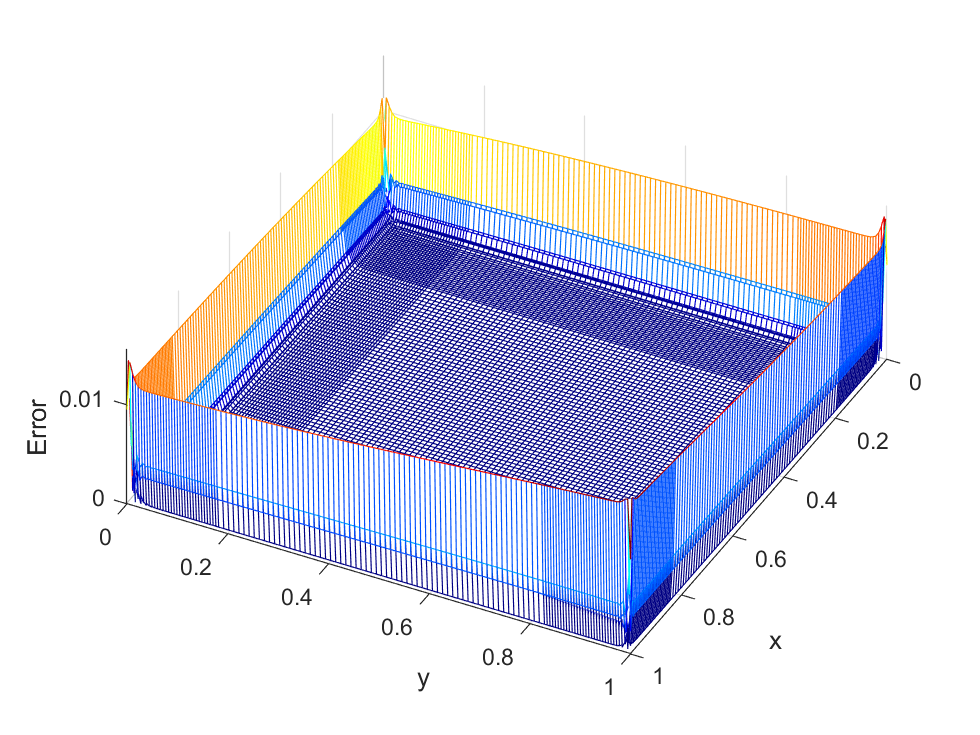}}
		\centerline{Pointwise error $|\uph_{64}-\widetilde{\uph}_{128}|$ for $\sq=10^{-4}$}
		\vspace{3pt}
		\centerline{\includegraphics[width=1.0\textwidth]{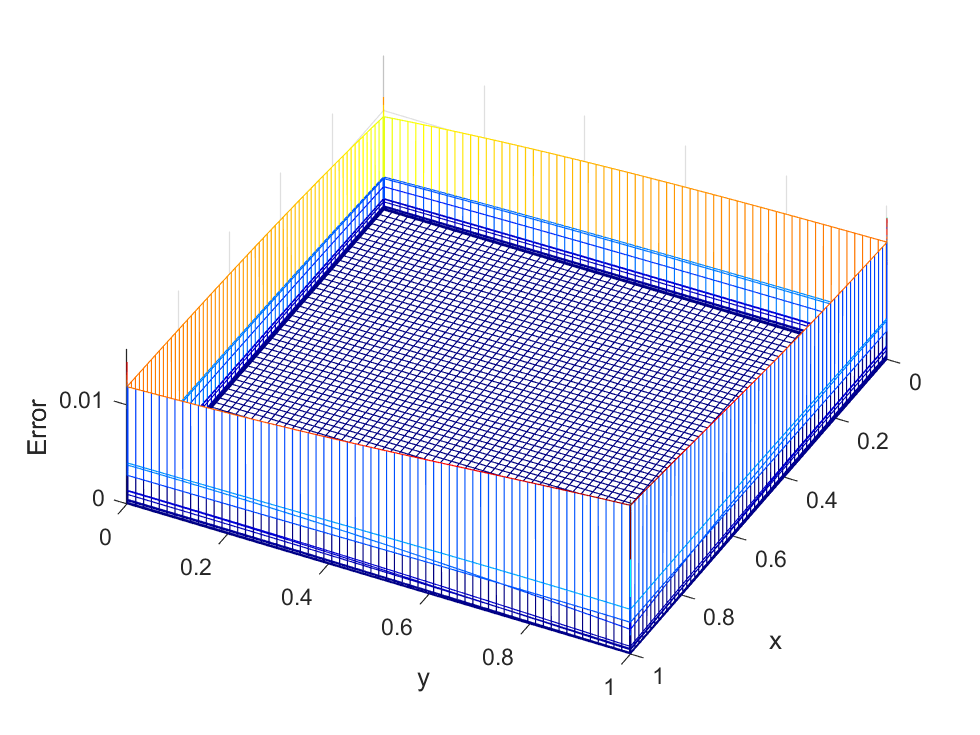}}
		\centerline{Pointwise error $|\uph_{64}-\widetilde{\uph}_{128}|$ for $\sq=10^{-8}$}
	\end{minipage}
	\caption{Numerical solution $\uph_{64}$ and pointwise error $|\uph_{64}-\widetilde{\uph}_{128}|$ in Example \ref{exa:2}. Here $k=2$.}
	\label{ex2:Fig:nmr:err}
\end{figure}

\section{Concluding remarks}\label{sec:concluding:remarks}

In this paper we considered a singularly perturbed reaction-diffusion problem
posed on the unit square and derived an optimal-order balanced-norm error estimate for the 
LDG method on Shishkin meshes by introducing layer-upwind numerical fluxes, 
which are a new and very effective way of choosing numerical fluxes. 
The terminology ``layer-upwind" means that on the fine Shishkin mesh, 
the value of the flux on each element is chosen at 
the point where the layer is weakest,
somewhat analogously to the use of ``upwinding" in discretisations of 
singularly perturbed convection-diffusion problems. 
Meanwhile, on the coarse Shishkin mesh one uses a standard central flux.
Our layer-upwind flux doesn't require any penalty parameters at the domain boundary
and leads to optimal-order $O((N^{-1}\ln N)^{k+1})$ error bounds 
in a balanced norm associated with this problem.
This is the best theoretical result that has been proved for the LDG method
applied to singularly perturbed reaction-diffusion problems,
and numerical experiments show that our theoretical bounds are sharp. 
In future work we aim to use similar ideas in applying the LDG method
to solve other types of singularly perturbed problems on layer-adapted meshes.

\bibliography{ChengWangStynes_full_version_4}
\bibliographystyle{amsplain}

\end{document}